%% file: DolbeaultNazaretSavare_submitted.tex
\documentclass[numbook,envcountsame]{svjour_modified}                     % onecolumn (standard format)
\smartqed  % flush right qed marks, e.g. at end of proof
\usepackage{graphicx}
\usepackage{mathptmx}
\usepackage{amssymb}
\usepackage{amsmath}
\usepackage{dsfont}
\usepackage{eucal}
\usepackage{amsfonts}
\numberwithin{equation}{section}
\usepackage{mathrsfs,enumerate}
\input{Macro_times.tex}

\begin{document}

\title{A new class of transport distances between
  measures}
\titlerunning{A new class of transport distances}

\author{Jean Dolbeault \and Bruno Nazaret \and Giuseppe Savar\'e}
\authorrunning{J. Dolbeault \and B. Nazaret \and G. Savar\'e}

\institute{J. Dolbeault, B. Nazaret \at Ceremade (UMR CNRS no. 7534),
Universit\'e Paris-Dauphine, place de Lattre de Tassigny, 75775 Paris
C\'edex~16, France.\\
\email{\{dolbeaul, nazaret\}@ceremade.dauphine.fr}
\and
G. Savar\'e \at Universit\`a degli studi di Pavia, Department of
Mathematics, Via Ferrata 1, 27100, Pavia, Italy.
\email{giuseppe.savare@unipv.it}
}
\date{\today}

\maketitle
\begin{abstract}
  We introduce a new class of distances between
  nonnegative Radon measures in $\Rd$. They are modeled
  on the dynamical characterization of the
  Kantorovich-Rubinstein-Wasserstein
  distances proposed by \textsc{Benamou-Brenier}
  \cite{Benamou-Brenier00}
  and provide a wide family interpolating
  between the Wasserstein and the homogeneous
  $W^{-1,p}_\gamma$-Sobolev distances.\\
  From the point of view of optimal transport theory,
  these distances minimize a dynamical cost
  to move a given initial distribution of mass
  to a final configuration.
  An important difference with the classical setting in mass transport
  theory is that the cost not only depends on
  the velocity of the moving particles but also on
  the densities of the intermediate configurations with respect to
  a given reference measure $\gamma$.\\
  We study the topological and
  geometric properties of these new distances,
  comparing them with the notion of weak convergence of
  measures and the well established
  Kantorovich-Rubinstein-Wasserstein theory.
  An example of possible applications
  to the geometric theory of gradient flows is also given.
\end{abstract}

%%%%%%%%%%%%%%%%%%%%%%%%%%%%%%%%%%%%%%%%%%%%%%%%%%%%%%%%%%%%%%%%%%%%%%%%%%
%%%%%%%%%%%%%%%%%%%%%%%%%%%%%%%%%%%%%%%%%%%%%%%%%%%%%%%%%%%%%%%%%%%%%%%%%% 
\section{Introduction}

Starting from the contributions by
\textsc{Y.\ Brenier, R.\ McCann,
  W.\ Gangbo, L.C.\ Evans, F.\ Otto, C.\ Villani}
\cite{Brenier91,Gangbo-McCann96,Otto99,Evans-Gangbo99,Otto-Villani00},
the theory of Optimal Transportation has received a lot of attention
and many deep applications to various mathematical fields,
such as
PDE's, Calculus of Variations,
functional and geometric inequalities, geometry of metric-measure
spaces,
have been found (we refer here to the monographs
\cite{Rachev-Ruschendorf98I,Evans99,Villani03,Ambrosio-Gigli-Savare05,Villani08}).
Among all possible transportation costs,
those inducing the so-called $L^p$-\textsc{Kantorovich-Rubinstein-Wasserstein}
distances $W_p(\mu_0,\mu_1)$, $p\in (1,+\infty)$, between
two probability measures $\mu,\nu\in \Probabilities\Rd$
\begin{equation}\label{defwaspre}
W_p(\mu_0,\mu_1):=\inf\left\{
  \left(\int_{\Rd\times \Rd}|y-x|^p\,\d\Sigma\right)^{\frac1p}:\
  \Sigma\in\Gamma(\mu_0,\mu_1)
\right\}
\end{equation}
play a distinguished role.
Here $\Gamma(\mu_0,\mu_1)$ is the set of all \emph{couplings} between $\mu_0$ and $\mu_1$:
they are probability measures $\Sigma$ on $\Rd\times \Rd$ whose first
and second marginals are respectively $\mu_0$ and $\mu_1$, i.e. $\Sigma(B\times \Rd)=\mu_0(B)$
and $\Sigma(\Rd\times B)=\mu_1(B)$ for all Borel sets $B\in\BorelSets{\Rd}$.

It was one of the most surprising achievement of
\cite{Otto96,Otto99,Jordan-Kinderlehrer-Otto98,Otto01}
that many evolution partial differential equations 
of the form
\begin{equation}
  \label{eq:cap1:1}
%  \left\{
 % \begin{aligned}
    \partial_t\rho+\nabla\cdot\big(\rho \,
    \big|\xxi\big|^{q-2}\xxi\big)
    % &
    =0,\qquad
    %\\
    \xxi
    % &
    =-\nabla\Big(\frac{\delta\mathscr F} {\delta \rho}\Big)
  %\end{aligned}
  %\right.
\qquad
  \text{in }\Rd\times (0,+\infty),
\end{equation}
can be, at least formally, interpreted as \emph{gradient flows} of
suitable integral functionals $\mathscr F$ with
respect to $W_p$ (see also the general approach developed in 
\cite{Villani03,Ambrosio-Gigli-Savare05,Villani08}).
In \eqref{eq:cap1:1} $\delta\mathscr F/\delta\rho$ is
the Euler first variation of $\mathscr F$, $q:=p/(p-1)$
is the H\"older's conjugate exponent of $p$, and
$t\mapsto \rho_t$ (a time dependent solution
of \eqref{eq:cap1:1}) can be interpreted as
a flow of probability measures $\mu_t=\rho_t\,\Leb d$
with density $\rho_t$ with respect to the Lebesgue measure $\Leb d$ in $\Rd$.

Besides showing deep relations with entropy estimates and functional
inequalities
\cite{Otto-Villani00}, this point of view provides a powerful
variational method to prove existence of solutions to
\eqref{eq:cap1:1},
by the so-called \emph{Minimizing movement} scheme
\cite{Jordan-Kinderlehrer-Otto98,DeGiorgi93,Ambrosio-Gigli-Savare05}:
given a time step $\tau>0$ and an initial datum $\mu_0=\rho_0\Leb d$,
the solution $\mu_t=\rho_t\Leb d$ at time $t\approx n\tau$ can be
approximated by the discrete solution $\mu_\tau^n$ obtained by
a recursive minimization of the functional
\begin{equation}
  \label{eq:DS:2}
  \mu\mapsto \frac 1{p\tau^{p-1}}W_p^p(\mu,\mu_\tau^k)+\mathscr
  F(\mu),\qquad
  k=0,1,\cdots
\end{equation}
The link between the Wasserstein distance and equations exhibiting the
characteristic structure of \eqref{eq:cap1:1}
(in particular the presence of the diffusion coefficient $\rho$,
the fact that $\xxi$ is a gradient vector field, and the
presence of the $q$-duality map
$\xxi\mapsto |\xxi|^{q-2}\,\xxi$), 
is well explained by the \emph{dynamic characterization} of $W_p$
introduced by
\textsc{Benamou-Brenier} \cite{Benamou-Brenier00}:
it relies in the minimization
of the ``action'' integral functional
%\begin{subequations}
%  \label{subeq:BB}
\begin{equation}
  \label{eq:cap1:2}
  \begin{aligned}
    &W_p^p(\mu_0,\mu_1)= \inf\Big\{\int_0^1 \int_\Rd
    \rho_t(x)\,|\vv_t (x)|^p\,\d x\,\d t\,:\\
    &\quad
    \partial_t\rho_t+\nabla\cdot(\rho_t\mathbf\vv_t)=0\ \text{in
      $\Rd\times (0,1)$,}\quad \mu_0=\rho\restr{t=0}\Leb d,\quad
    \mu_1=\rho\restr{t=1}\Leb d\Big\}.
  \end{aligned}
\end{equation}
%\end{subequations}
% in the sense of weak$^*$ convergence (i.e.\ with
% respect to the duality with continuous functions with
% compact support) of measures as $t\down0$ or $t\up1$.
% Being $\mu_0,\mu_1$ probability measures, the continuity equation
% yields the conservation of the total mass
% \begin{equation}
%   \label{eq:cap1:7}
%   \int_\Rd \rho_t(x)\,\d x=1\quad
%   \forall\, t\in (0,1).
% \end{equation}
% As showed in \cite{Ambrosio-Gigli-Savare05},
% this characterization has an important counterpart
% at the level of absolutely continuous curves w.r.t.\ $W_p$:
% from a metric point of view (and restricting to
% the absolutely continuous case for simplicity),
% these are time dependent
% family of densities $t\in [0,T]\to\rho_t$ satisfying \eqref{eq:cap1:7}
% such that
% \begin{equation}
%   \label{eq:cap1:31}
%   W_p(\rho_s\Leb d,\rho_t\Leb d)\le \int_s^t m(r)\,\d r\quad
%   \forall\, 0\le s<t\le T,\qquad
%   \text{for some }m\in L^1(0,T).
% \end{equation}
% It is then possible to prove that
% there exists a unique Borel family of vector fields $\vv_t\in
% L^p(\rho_t;\Rd)$
% such that \eqref{eq:cap1:3} is satisfied,  
% \begin{equation}
%   \label{eq:cap1:32}
%   W_p(\rho_s\Leb d,\rho_t\Leb d)\le \int_s^t
%   \|\vv_r\|_{L^p(\rho_r;\Rd)}\,\d r,\quad
%   \lim_{h\to0}\frac{W_p(\rho_t,\rho_{t+h})}{|h|}=
%   \|\vv_t\|_{L^p(\rho_t;\Rd)}
% \end{equation}
% $\Leb 1$-a.e.\ in $(0,T)$, and
% \begin{equation}
%   \label{eq:cap1:33}
%   \xxi_t=|\vv_t|^{p-2}\vv_t\in \overline{\big\{\nabla\zeta:\zeta\in C^\infty_{\rm c}(\Rd)
%   \big\}}^{L^p(\rho_t;\Rd)}.
% \end{equation}
\paragraph{Towards more general cost functionals.}
If one is interested to study the more general class of diffusion equations
\begin{equation}
  \label{eq:cap1:5}
  %\left\{
  %\begin{aligned}
    \partial_t\rho+\nabla\cdot\big(h(\rho) \,
    \big|\xxi\big|^{q-2}\xxi\big)
    % &
    =0,\qquad
    % \\
    \xxi
    % &
    =-\nabla\Big(\frac{\delta\mathscr F} {\delta \rho}\Big)
  %\end{aligned}
  %\right.
    \qquad
  \text{in }\Rd\times (0,+\infty),
\end{equation}
obtained from \eqref{eq:cap1:1} replacing the
mobility coefficient $\rho$ by an increasing nonlinear function
$h(\rho)$, $h:[0,+\infty)\to [0,+\infty)$ whose
typical examples are the functions $h(\rho)=\rho^\alpha$,
$\alpha\ge0$,
it is then natural to investigate the properties of the
``distance''
%\begin{subequations}
%  \label{subeq:BBbis}
\begin{equation}
  \label{eq:cap1:2bis}
  \begin{aligned}
    &\widetilde W_p^p(\mu_0,\mu_1)= \inf\Big\{\int_0^1 \int_\Rd
    h\big(\rho_t(x)\big)|\vv_t(x)|^p\,\d x\,\d t:\\
    &\quad
   \partial_t\rho_t+\nabla\cdot(h(\rho_t)\,\vv_t)=0\ \text{in
    $\Rd\times (0,1)$,}\quad
   \mu_0=\rho\restr{t=0}\Leb d,\quad
  \mu_1=\rho\restr{t=1}\Leb d\Big\}.
  \end{aligned}
\end{equation}
In the limiting case $\alpha=0$, $h(\rho)\equiv 1$, one can easily 
recognize that \eqref{eq:cap1:2bis}
provides an equivalent description of
the homogeneous (dual) 
$\dot W^{-1,p}(\Rd)$ Sobolev (pseudo)-distance
\begin{equation}
  \label{eq:cap1:8}
  \|\mu_0-\mu_1\|_{\dot W^{-1,p}(\Rd)}:=
  \sup\Big\{\int_\Rd \zeta\,\d(\mu_0-\mu_1):
  \zeta\in C^1_{\rm c}(\Rd),\ 
  \int_\Rd |\D\zeta |^q\,\d x\le 1\Big\}.
\end{equation}
Thus the distances defined by \eqref{eq:cap1:2bis}
for $0\le\alpha\le1$ (we shall see that this is the natural range
for the parameter $\alpha$) can be
considered as a natural ``interpolating'' family  between the Wasserstein
and the (dual) Sobolev ones.

Notice that if one wants to keep the usual transport interpretation
given by a ``dynamic cost'' to be minimized along the solution
of the continuity equation, one can simply
introduce the velocity vector field
$ \tilde\vv_t:=\rho_t^{-1}
h(\rho_t)\vv_t$
% , thus rewriting \eqref{eq:cap1:3bis} as
% \begin{subequations}
%   \begin{equation}
%     \label{eq:cap1:28}
%     \partial_t\rho_t+\nabla\cdot (\rho_t\tilde\vv_t)=0,
%   \end{equation}
and minimize the cost 
  \begin{equation}
    \label{eq:cap1:29}
    \int_0^1 \int_\Rd \rho \,f(\rho)\, |\tilde\vv_t|^p\,\d x\,\d t,\quad
    \text{where}\quad
    f(\rho):=\Big(\frac \rho{h(\rho)}\Big)^{p-1}.
  \end{equation}
%\end{subequations}
Therefore, in this model the usual $p$-energy $\int_\Rd \rho_t|\tilde\vv_t|^p\,\d x$ of the moving
masses $\rho_t$ with velocity $\tilde\vv_t$ results locally
modified by a factor $f(\rho_t)$ depending on the local density
of the mass occupied at the time $t$. Different non-local models
have been considered in \cite{Brancolini-Buttazzo-Santambrogio06,Ambrosio-Santambrogio07}.

In the present paper we try to present a systematic
study of these families of intermediate distances, in view
of possible applications, e.g., to the study of evolution equations
like \eqref{eq:cap1:5}, the Minimizing movement approach
\eqref{eq:DS:2}, and
functional inequalities.
\paragraph{Examples: PDE's as gradient flows.}
Let us show a few examples evolution equations which can be \emph{formally}
interpreted as gradient flows of suitable integral functionals in this
setting: the scalar conservation law
\begin{displaymath}
  \partial_t\rho -\nabla\cdot \big(\rho^\alpha \nabla
   V\big)=0
\quad\text{corresponds to the linear functional}
  \quad
  \mathscr F(\rho):=\int_\Rd V(x)\,\rho\,\d x,
\end{displaymath}
 for some
 smooth potential $V:\Rd\to\R$ and $p=2$.
 Choosing for $m>0$
\begin{displaymath}
  p=2,\quad \mathscr F(\rho)=c_{\alpha,m}\int
\rho^{m+1-\alpha}\,\d x ,\quad
c_{\alpha,m}:=\frac m{(m+1-\alpha)(m-\alpha)},
\end{displaymath}
one gets the porous media/fast diffusion equation
\begin{equation}
  \label{eq:newclass_times:1}
  \partial_t\rho-\frac m{m-\alpha}\nabla\cdot\big(\rho^\alpha
  \nabla \rho^{m-\alpha}\big)=
  \partial_t\rho-\Delta \rho^m=0,
\end{equation}
and in particular the heat equation for the entropy functional
$\frac 1{(2-\alpha)(1-\alpha)}\int \rho^{2-\alpha}\,\d x$.
Choosing
\begin{displaymath}
  \mathscr F(\rho)=c_{\alpha,m,q}\int
  \rho^{\frac{m+2q-3-\alpha}{q-1}}\,\d x ,\quad
  c_{\alpha,m,q}:=\frac {m(q-1)^q}{(m+2q-3-\alpha)(m+q-2-\alpha)},
\end{displaymath}
one obtains the doubly nonlinear equation
\begin{equation}
  \label{eq:newclass_times:2}
  \partial_t\rho-m\nabla\cdot(\rho^{m-1}|\nabla\rho|^{q-2}\nabla \rho)=0
\end{equation}
and in particular the evolution equation for the $q$-Laplacian when
$m=1$.
The Dirichlet integral for $p=2$
\begin{equation}
  \label{eq:newclass_times:3}
  \mathscr F(\rho)=\frac 12\int |\nabla \rho|^2\,\d x\quad
  \text{yields}\quad
  \partial_t\rho +\nabla\cdot\big(\rho^\alpha \nabla \Delta \rho\big)=0,
\end{equation}
a thin-film like equation.
\paragraph{The measure-theoretic point of view: Wasserstein distance.}
We present now the main points of our approach
(see also, in a different context, \cite{Buttazzo-Jimenez-Oudet07}).
First of all, even if the language of densities
and vector fields (as $\rho$ and $\vv,\tilde\vv$ in \eqref{eq:cap1:2}
or (\ref{eq:cap1:2bis}))
is simpler and suggests interesting interpretations,
the natural framework for considering the variational
problems \eqref{eq:cap1:2} and \eqref{eq:cap1:2bis}
is provided by time dependent families of Radon measures in $\Rd$.
Following this point of view, one can replace $\rho_t$ by
a continuous curve $t\in [0,1]\mapsto \mu_t$
($\mu_t=\rho_t\,\Leb d$ in the absolutely continuous case)
in
the space $\RPM(\Rd)$ of nonnegative Radon measures in $\Rd$
endowed with the usual weak$^*$ topology induced by the duality
with functions in $C^0_{\rm c}(\Rd)$.
The (Borel) vector field $\vv_t$ in \eqref{eq:cap1:2}
induces a time dependent family of vector measures
$\nnu_t:=\mu_t\vv_t\ll\mu_t$.
In terms of the couple $(\mu,\nnu)$ the
continuity equation \eqref{eq:cap1:2} reads
\begin{equation}
  \label{eq:cap1:9}
  \partial_t \mu_t+\nabla\cdot\nnu_t=0\quad\text{in the sense of
    distributions
    in }\mathscr D'(\Rd\times(0,1)),
\end{equation}
and it is now a \emph{linear} equation.
Since $\vv_t=\d \nnu_t/\d \mu_t$
is the density of $\nnu_t$ w.r.t.\ $\mu_t$,
the action functional
which has to be minimized in \eqref{eq:cap1:2}
can be written as
\begin{equation}
  \label{eq:cap1:10}
  \mathscr E_{p,1}(\mu,\nnu)=\int_0^1
  \Phi_{p,1}(\mu_t,\nnu_t)\,\d t,\quad
  \Phi_{p,1}(\mu,\nnu):=\int_\Rd \left|\frac{\d \nnu}{\d
      \mu}\right|^p\,
  \d\mu.
\end{equation}
Notice that in the case of  absolutely continuous measures
with respect to $\Leb d$, i.e.
$\mu=\rho\Leb d$ and $\nnu=\ww\Leb d$,
the functional $\Phi_{p,1}$ can also be expressed as
\begin{equation}
  \label{eq:cap1:12}
  \Phi_{p,1}(\mu,\nnu):=\int_\Rd \phi_{p,1}(\rho,\ww)\,\d\Leb
  d(x),\quad
  \phi_{p,1}(\rho,\ww):=\rho\left|\frac \ww\rho\right|^p.
\end{equation}
Denoting by $\ce01\Rd$ the class of measure-valued
distributional solutions $\mu,\nnu$ of the continuity equation
\eqref{eq:cap1:9}, we end up with the equivalent
characterization of the Kanto\-ro\-vich-Rubinstein-Wasserstein distance
\begin{equation}
  \label{eq:cap1:13}
  W_p^p(\mu_0,\mu_1):=\inf\Big\{\mathscr E_{p,1}(\mu,\nnu):
  (\mu,\nnu)\in \ce01\Rd,\
  \mu\restr{t=0}=\mu_0,\
  \mu\restr{t=1}=\mu_1
  \Big\}.
\end{equation}
\paragraph{Structural properties and convexity issues.}
The density function $\phi=\phi_{p,1}:(0,+\infty)\times\Rd
\to\Rd$ appearing in \eqref{eq:cap1:12} exhibits some crucial features
\begin{enumerate}
\item $\ww\mapsto\phi(\cdot,\ww)$ is symmetric, positive
  (when $\ww\neq0$), and
  $p$-homogeneous with respect to
  the vector variable $\ww$: this ensures that
  $W_{p}$ is symmetric and satisfies the triangular inequality.
\item $\phi$ is jointly convex in $(0,+\infty)\times \Rd$:
  this ensures that the functional $\Phi_{p,1}$
  (and therefore also $\mathscr E$) defined
  in \eqref{eq:cap1:10} is lower semicontinuous
  with respect to the weak$^*$ convergence of Radon measures.
  It is then possible to show that the infimum
  in \eqref{eq:cap1:13} is attained, as soon as it is finite
  (i.e.\ when there exists at least one curve $(\mu,\nnu)\in \ce01\Rd$
  with finite energy $\mathscr E(\mu,\nnu)$ joining
  $\mu_0$ to $\mu_1$); in particular $W_{p}(\mu_0,\mu_1)=0$
  yields $\mu_0=\mu_1$.
  Moreover, the distance map $(\mu_0,\mu_1)\mapsto
  W_p(\mu_0,\mu_1)$ is lower semicontinuous
  with respect to the weak$^*$ convergence, a crucial property
  in many variational problems involving $W_p$, as
  \eqref{eq:DS:2}.
\item
  $\phi$ is jointly positively $1$-homogeneous:
  this a distinguished feature of the Wasserstein case,
  which shows that the functional
  $\Phi_{p,1}$ depends only on $\mu,\nnu$ and not
  on the Lebesgue measure $\Leb d$, even if it
  can be represented as in \eqref{eq:cap1:12}.
  In other words, suppose that
  $\mu=\tilde\rho\gamma$ and $\nnu=\tilde\ww\gamma$, where
  $\gamma$ is another reference (Radon, nonnegative) measure
  in $\Rd$.
  Then
  \begin{equation}
    \label{eq:cap1:14}
    \Phi_{p,1}(\mu,\nnu)=\int_\Rd \phi_{p,1}(\tilde\rho,\tilde\ww)\,\d\gamma.
  \end{equation}
  As we will show in this paper, the $1$-homogeneity assumption
  yields also two ``quantitative'' properties: if $\mu_0$ is a probability
  measure, then any solution $(\mu,\nnu)$ of
  the continuity equation \eqref{eq:cap1:9}
  with finite energy $\mathscr E(\mu,\nnu)<+\infty$
  still preserves the mass
  $\mu_t(\Rd)\equiv 1$ for every time $t\ge0$
  (and it is therefore equivalent to assume this condition in the
  definition of $\ce01\Rd$, see e.g.\
  \cite[Chap.\ 8]{Ambrosio-Gigli-Savare05}).
  Moreover, if the $p$-moment of $\mu_0$
  $    \sfm_p(\mu_0):=\int_\Rd |x|^p\,\d \mu_0(x)$ is finite,
  then $W_p(\mu_0,\mu_1)<+\infty$ if and only if
  $\sfm_p(\mu_1)<+\infty.$
\end{enumerate}
\paragraph{Main definitions.}
Starting from the above remarks, it is then
natural to consider the more general case
when the density functional $\phi:(0,+\infty)\times
\Rd\to[0,+\infty)$
still satisfies 1.~($p$-homogeneity w.r.t.\ $\ww$) and 2.~(convexity),
but not 3.~($1$-homogeneity).
Due to this last choice, the associated integral functional $\Phi$
is no more independent of a reference measure $\gamma$
and it seems therefore too restrictive to
consider only the case of the Lebesgue
measure $\gamma=\Leb d$.

In the present paper we will thus introduce 
a further nonnegative \emph{reference}
Radon measure $\gamma\in \RPM(\Rd)$ and
a general convex functional  $\phi:(0,+\infty)\times\Rd
\to [0,+\infty)$ which is $p$-homogeneous w.r.t.\ its second (vector)
variable and non degenerate (i.e.\ $\phi(\rho,\ww)>0$ if $\ww\neq0$).
Particularly interesting examples of density functionals $\phi$,
corresponding
to \eqref{eq:cap1:2bis}, are given by
\begin{equation}
  \label{eq:cap1:20}
  \phi(\rho,\ww):=h(\rho)\left|\frac\ww{h(\rho)}\right|^p,
\end{equation}
where $h:(0,+\infty)\to(0,+\infty)$ is an increasing and
\emph{concave}
function; the concavity of $h$ is a necessary and sufficient
condition for the convexity of $\phi$ in \eqref{eq:cap1:20} (see
\cite{Rockafellar68} and \S \ref{sec:action}).
Choosing $h(\rho):=\rho^\alpha$, $\alpha\in (0,1)$, one obtains
\begin{equation}
  \label{eq:cap1:21}
  \phi_{p,\alpha}(\rho,\ww):=\rho^\alpha\,
  \left|\frac\ww{\rho^\alpha}\right|^p=\rho^{\theta-p}\,|\ww|^p,\qquad
  \theta:=(1-\alpha)\,p+\alpha\in (1,p),
\end{equation}
which is jointly $\theta$-homogeneous in $(\rho,\ww)$.

In the case, e.g., when $\alpha<1$ in
\eqref{eq:cap1:21}
or more generally
$\lim_{\rho\up\infty}h(\rho)/\rho=0$,
the \emph{recession function} of $\phi$ satisfies 
\begin{equation}
  \label{eq:cap1:18}
  \phi^\infty(\rho,\ww)=
  \lim_{\lambda\up+\infty}\lambda^{-1}\phi(\lambda\rho,\lambda\ww)=+\infty\qquad
  \text{if }\rho,\ww\neq0,
\end{equation}
so that
the associated integral functional reads as
\begin{equation}
  \label{eq:cap1:17}
  \Phi(\mu,\nnu|\gamma):=
  \int_\Rd \phi(\rho,\ww)\,\d\gamma\quad
  \mu=\rho\gamma+\mu^\perp,\quad
  \nnu=\ww\gamma\ll \gamma,
\end{equation}
extended to $+\infty$ when $\nnu$ is not absolutely continuous
with respect to $\gamma$ or $\supp(\mu)\not\subset\supp(\gamma)$.
Notice that only the
density $\rho$ of the $\gamma$-absolutely continuous
part of $\mu$ enters in the functional, but the functional
could be finite even if $\mu$ has a singular part $\mu^\perp$.
This choice is crucial in order to obtain a lower semicontinuous functional
w.r.t.\ weak$^*$ convergence of measures.
The associated $(\phi,\gamma)$-Wasserstein distance is therefore
\begin{equation}
  \label{eq:cap1:16}
  \cW_{\phi,\gamma}^p(\mu_0,\mu_1):=\inf\Big\{\mathscr E_{\phi,\gamma}(\mu,\nnu):
  (\mu,\nnu)\in \ce01\Rd,\
  \mu\restr{t=0}=\mu_0,\
  \mu\restr{t=1}=\mu_1
  \Big\},
\end{equation}
where
the energy $\EE_{\phi,\gamma}$ of a curve
$(\mu,\nnu)\in \ce01\Rd$ is
\begin{equation}
  \label{eq:cap1:19}
  \mathscr
  E_{\phi,\gamma}(\mu,\nnu):=\int_0^1\Phi(\mu_t,\nnu_t|\gamma)\,\d t.
\end{equation}
The most important case associated to the functional
\eqref{eq:cap1:21} deserves the distinguished notation
\begin{equation}
  \label{eq:cap1:25}
  W_{p,\alpha;\gamma}(\cdot,\cdot):=
  \cW_{\phi_{p,\alpha},\gamma}(\cdot,\cdot).
\end{equation}
The limiting case $\alpha=\theta=1$ corresponds to the $L^p$-Wasserstein
distance,
the Sobolev $\dot W^{-1,p}_\gamma$ corresponds to $\alpha=0$,
$\theta=p$.
The choice of $\gamma$ allows for a great flexibility:
besides the Lebesgue measure in $\Rd$, we quote
\begin{itemize}
\item $\gamma:=\Leb d\restr{\Omega}$, $\Omega$ being
  an open subset of $\Rd$. The measures are then supported
  in $\bar\Omega$ and, 
  with the choice \eqref{eq:cap1:20} and $\vv=\ww/h(\rho)$,
  \eqref{eq:cap1:9} is a weak formulation of
  the continuity equation ($\nn_{\partial\Omega}$ being
  the exterior unit normal to $\partial\Omega$)
  \begin{equation}
    \label{eq:cap1:22}
    \partial_t\rho_t+\nabla\cdot\big(h(\rho_t)\,\vv_t\big)=0\quad
    \text{in }\Omega\times(0,1),\quad
    \vv_t\cdot \nn_{\partial\Omega}=0\quad\text{on }\partial\Omega.
  \end{equation}
  This choice is useful for studying equations
  \eqref{eq:newclass_times:1} (see \cite{Carrillo-Lisini-Savare08}),
  \eqref{eq:newclass_times:2},
  \eqref{eq:newclass_times:3} in bounded domains with Neumann
  boundary conditions.
\item $\gamma:=e^{-V}\Leb d$ for some $C^1$ potential $V:\Rd\to
  \R$.  With the choice \eqref{eq:cap1:20} and $\vv=\ww/h(\rho)$
  \eqref{eq:cap1:9} is a weak formulation of
  the equation
  \begin{equation}
    \label{eq:cap1:23}
    \partial_t\rho_t +\nabla\cdot \big(h(\rho_t)\,\vv_t\big)-
    h(\rho_t)\nabla V\cdot\vv_t=0\quad
    \text{in }\Rd\times(0,1).
  \end{equation}
  When $h(\rho)=\rho^\alpha,\, p=2$, the gradient flow of
  $\mathscr F(\mu):=\frac 1{(2-\alpha)(1-\alpha)}\int_{\Rd}
  \rho^{2-\alpha}\,d\gamma$ is
  the Kolmogorov-Fokker-Planck equations 
  \cite{Dolbeault-Nazaret-Savare08b}
  \begin{displaymath}
    \partial_t \mu-\Delta\mu-\nabla\cdot(\mu\nabla V)=0,\quad
    \partial_t\rho-\Delta \rho+\nabla V\cdot \nabla\rho=0,
  \end{displaymath}
  which in the Wasserstein framework is generated by
  the logarithmic entropy
  (\cite{Jordan-Kinderlehrer-Otto98,Ambrosio-Gigli-Savare05,Ambrosio-Savare-Zambotti07}).
\item $\gamma:=\mathscr H^k\restr{\M}$, $\M$ being
  a smooth $k$-dimensional manifold embedded in $\Rd$
  with the Riemannian metric induced by the
  Euclidean distance; $\mathscr H^k$ denotes the
  $k$-dimensional Hausdorff measure.
  \eqref{eq:cap1:9} is a weak formulation of
  \begin{equation}
    \label{eq:cap1:24}
    \partial_t \rho_t+{\rm div}_{\M}\big(h(\rho)\vv_t\big)=0\quad
    \text{on }\M\times(0,1).
  \end{equation}
  Thanks to Nash embedding theorems \cite{Nash54,Nash56},
  the study of the continuity equation and
  of the weighted Wasserstein distances
  on arbitrary Riemannian manifolds
  can be reduced to this case, which could be therefore applied to study
  equations
  \eqref{eq:newclass_times:1},
  \eqref{eq:newclass_times:2},
  \eqref{eq:newclass_times:3}
  on Riemannian manifolds.
\end{itemize}
\paragraph{Main results.}
Let us now summarize some of the
main properties of $W_{p,\alpha;\gamma}(\cdot,\cdot)$
we will prove in the last section of the present paper.
In order to deal with distances (instead of
pseudo-distances, possibly assuming
the value $+\infty$), for
a nonnegative Radon measure $\sigma$ we
will denote by
$\mathcal M_{p,\alpha;\gamma}[\sigma]$
the set of all measures $\mu$
with $W_{p,\alpha;\gamma}(\mu,\sigma)<+\infty$ endowed
with the
$W_{p,\alpha;\gamma}$-distance.
\begin{enumerate}
\item $\mathcal M_{p,\alpha;\gamma}[\sigma]$ is a complete metric space
  (Theorem \ref{thm:completeness}).  
\item $W_{p,\alpha;\gamma}$ induces a stronger convergence
  than the usual weak$^*$ one (Theorem \ref{thm:1}).
\item Bounded sets in $\mathcal M_{p,\alpha;\gamma}[\sigma]$ are
  weakly$^*$ relatively compact (Theorem \ref{thm:1}).
\item The map $(\mu_0,\mu_1)\mapsto
  W_{p,\alpha;\gamma}(\mu_0,\mu_1)$ is weakly$^*$ lower semicontinuous
  (Theorem \ref{thm:2}),
  convex (Theorem \ref{thm:3}), and subadditive
  (Theorem \ref{thm:subadditivity}).
  It enjoys some useful monotonicity properties
  with respect to $\gamma$ (Proposition \ref{prop:monotonicity})
  and to convolution (Theorem \ref{thm:wass_conv}).
\item The infimum in \eqref{eq:cap1:13}
  is attained, $\mathcal M_{p,\alpha;\gamma}[\sigma]$ is a geodesic space
  (Theorem \ref{thm:0}),
  and constant speed geodesics connecting two measures
  $\mu_0,\mu_1$ are unique (Theorem \ref{thm:3}).
\item
  If
  \begin{equation}
    \label{eq:cap1:27}
    \int_{|x|\ge1}|x|^{-p/(\theta-1)}\,\d\gamma(x)<+\infty\qquad
    \theta=(1-\alpha)p+\alpha,\quad
    \frac{p}{\theta-1}=\frac q{1-\alpha},   
  \end{equation}
  and $\sigma\in \Probabilities\Rd$, then
  $\mathcal M_{p,\alpha;\gamma}[\sigma]\subset \Probabilities\Rd$
  (Theorem \ref{thm:narrow_convergence}).
  If moreover $\gamma$ satisfies stronger summability
  assumptions, then the distances $W_{p,\alpha;\gamma}$
  provide a control of various moments of the measures
  (Theorem \ref{thm:moment_convergence}).
  Comparison results with $W_p$ and $\dot W^{-1,p}$ are also discussed 
  in \S \ref{subsec:comp}.
\item Absolutely continuous curves w.r.t. $W_{p,\alpha;\gamma}$
  can be characterized in completely analogous ways
  as in the Wasserstein case
  (\S \ref{subsec:curves}).
\item In the case $\gamma=\Leb d$ the functional
  \begin{equation}
    \label{eq:cap1:30}
    \Psi_\alpha(\mu|\gamma):=
    \frac1{(2-\alpha)(1-\alpha)}\int_\Rd \rho^{2-\alpha}\,\d x\qquad
    \mu=\rho\Leb d\ll\Leb d,
\end{equation}
is geodesically convex w.r.t.\ the distance $W_{2,\alpha;\Leb d}$
and the heat equation in $\Rd$ is its gradient flow, as formally suggested by
\eqref{eq:newclass_times:1} (\S
\ref{subsec:Leb}: we prove this
property in the case
$\alpha>1-2/d$, when $\Probabilities\Rd$ is complete
w.r.t.\ $W_{2,\alpha;\Leb d}$.)
\end{enumerate}

\paragraph{Plan of the paper.}
Section 2 recalls some basic notation and preliminary facts about
weak$^*$ convergence and integral functionals
of Radon measures;  \ref{subsec:duality} recalls 
a simple duality result in convex analysis, which plays a crucial
role in the analysis of the integrand $\phi(\rho,\ww)$.

The third section is devoted to 
the class of admissible action integral functionals
$\Phi$ like \eqref{eq:cap1:17}
and their density $\phi$. Starting from a few basic structural
assumptions on $\phi$ we deduce its main properties and we present
some important examples in Section \ref{subsec:phi_examples}.
The corresponding properties of $\Phi$ (in particular, lower
semicontinuity
and relaxation with respect to weak$^*$ convergence, monotonicity, etc)
are considered in Section \ref{subsec:action_functional}.

Section \ref{sec:continuity} is devoted to the study
of measure-valued solutions of the continuity equation
\eqref{eq:cap1:9}. It starts with some preliminary basic results,
which extend the theory presented in
\cite{Ambrosio-Gigli-Savare05} to the
case of general Radon measures: this extension is motivated by
the fact that the class of probability measures (and therefore
with finite mass) is too restrictive to study
the distances $W_{p,\alpha;\gamma}$, in particular when
$\gamma(\Rd)=+\infty$
as in the case of the Lebesgue measure.
We shall see (Remark \ref{rem:LebComp}) that $\Probabilities\Rd$ with
the distance $W_{p,\alpha;\Leb d}$ is not complete
if $d>p/(\theta-1)=q/(1-\alpha)$.
We consider in Section \ref{subsec:FE} the class of solutions
of \eqref{eq:cap1:9} with finite energy $\mathscr E_{\phi,\gamma}$
\eqref{eq:cap1:19}, deriving all basic estimates
to control their mass and momentum.

As we briefly showed, Section \ref{sec:modWass} contains all main results of the paper
concerning the modified Wasserstein distances.

%%%%%%%%%%%%%%%%%%%%%%%%%%%%%%%%%%%%%%%%%%%%%%%%%%%%%%%%%%%%%%%%%%%%%%%%%%
%%%%%%%%%%%%%%%%%%%%%%%%%%%%%%%%%%%%%%%%%%%%%%%%%%%%%%%%%%%%%%%%%%%%%%%%%% 
\section{Notation and preliminaries}

Here is a list of the main
notation used throughout the paper:
\medskip

 \halign{$#$\hfil\qquad &#\hfil\cr
   \Ball R&The open ball (in some $\Rh$)
   of radius $R$ centered at $0$\cr
   \mathcal B(\Rh)\quad(\text{resp. }\Bloc(\Rh))&
   Borel subsets of $\Rh$ (resp. with compact closure)\cr
   \Probabilities\Rh&Borel probability measures in $\Rh$\cr
   \FPM(\Rh)\quad
   (\text{resp. }\RPM(\Rh))&Finite (resp.\ Radon), nonnegative Borel
   measures
   on $\Rh$\cr
   \Probabilities\Rh&Borel probability measures in $\Rh$\cr
   \mathcal M(\Rh;\Rm)
   &
   $\Rm$-valued Borel measures
   with finite variation\cr
   \Mloc(\Rh;\Rm)&$\Rm$-valued Radon measures\cr
   \|\mmu\|&Total variation of $\mmu\in \Mloc(\Rh;\Rm)$,
   see \eqref{eq:scheme3:28}\cr
   C^0_b(\Rh)&Continuous and \emph{bounded} real functions\cr
   %\text{Weak$^*$ convergence in }\Mloc(\Rh;\Rm)&
   %Convergence in the duality with $C^0_{\rm c}(\Rh;\Rm)$\cr
   %\text{Narrow convergence in }\FinitePositiveMeasures{\Rd}
   %&Convergence
   %in the duality with $C^0_b(\Rh)$\cr
   \sf m_p(\mu)&$p$-moment $\int_{\R^d}|x|^p\,\d\mu$ of $\mu\in \FPM(\Rh)$
   \cr
   \psi^\infty&Recession function of $\psi$, see \eqref{eq:cap2:1}\cr
   \Psi(\mmu|\gamma),\ \Phi(\mu,\nnu|\gamma)&Integral functionals on
   measures,
   see \ref{sec:convex_functionals}
   and
   \ref{subsec:action_functional}\cr
   \mu(\zeta),\ \la\mu,\zeta\ra, \ \la \mmu,\zzeta\ra&
   the integrals $\int_\Rd \zeta\,\d\mu$, $\int_\Rd
   \zzeta\cdot\d \mmu$\cr
   \ce0T\Rd,\CE\phi\gamma0T\Rd,&Classes of measure-valued
   solutions of
   the continuity 
   \cr
   \ \cce0T\Rd{\mu_0}{\mu_1}&\quad equation, see Def.\ \ref{def:CE} and
   Sec.\ \ref{subsec:FE}.\cr   
 }

\subsection{\bfseries Measures and weak convergence}
We recall some basic notation and properties of weak convergence
of (vector) radon measures (see e.g.\ \cite{Ambrosio-Fusco-Pallara00}).
A Radon vector measure
in $\Mloc(\Rh;\Rm)$ is a $\Rm$-valued map
$\mmu:\mathcal B_{\rm c}(\Rh)\to
\Rm$ defined on the Borel sets of $\Rh$ with compact closure.
We identify 
$\mmu\in \Mloc(\Rh;\Rm)$ with
a vector
$(\mmu^1,\mmu^2,\cdots,\mmu^m)$ of $m$
measures in $\Mloc(\Rh)$:
its integral with a continuous
vector valued function with compact support
$\zzeta\in C^0_{\rm c}(\Rh;\Rm)$
is given by
\begin{equation}
  \label{eq:scheme3:20}
  \la \mmu,\zzeta\ra:=\int_\Rh \zzeta\cdot \,\d\mmu=
  \sum_{i=1}^m\int_\Rh \zzeta^i(x)\,\d\mmu^i(x).
\end{equation}
It is well known that $\Mloc(\Rh;\Rm)$
can be identified with the dual of $C^0_{\rm c}(\Rh;\Rm)$
by the above duality pairing and it is therefore endowed with the
corresponding
of weak$^*$ topology.
If $\|\cdot\|$ is a norm in $\Rd$ with dual $\|\cdot\|_*$
(in particular the euclidean norm $|\cdot|$)
for every open subset $A\subset \Rh$ we have
\begin{equation}
  \label{eq:scheme3:28}
  \|\mmu\|(A)=\sup\Big\{
  \int_\Rh \zzeta\cdot \,\d\mmu:\quad
  \supp(\zzeta)\subset A,\quad
  \|\zzeta(x)\|_*\le 1\quad
  \forall\, x\in \Rh\Big\}.
\end{equation}
$\|\mmu\|$ is in fact a Radon positive measure
in $\RPM(\Rh)$
and $\mmu$ admits the polar decomposition
$\mmu=\ww \|\mmu\|$ where the
Borel vector field $\ww$ belongs to 
$L^1_{\rm loc}(\|\mmu\|;\Rm)$. We thus have
\begin{equation}
  \label{eq:cap2:22}
  \la \mmu,\zzeta\ra=\int_\Rh \zzeta\cdot \,\d\mmu=
  \int_\Rh \zzeta\cdot \ww\,\d\|\mmu\|.
\end{equation}
If $(\mmu_k)_{k\in \N}$
is a sequence in $\Mloc(\Rh;\Rm)$
with 
$\sup_n \|\mmu\|(\Ball R)<+\infty$ for every
open ball $\Ball R$,
then it is possible to extract a subsequence
$\mmu_{k_n}$  weakly$^*$ convergent to $\mmu\in \mathcal
M(\Rh;\Rm)$, whose total variation
$\|\mmu_{k_n}\|$ weakly$^*$ converges to $\lambda\in \mathcal
M^+(\Rh)$
with
$\|\mmu\|\le \lambda$.
\subsection{\bfseries Convex functionals defined on Radon measures}
\label{sec:convex_functionals}
Let $\psi:\Rm\to [0,+\infty]$ be a convex and lower semicontinuous
function with $\psi(0)=0$, 
whose proper domain $D(\psi):=\{x\in \Rm:\psi(x)<+\infty\}$
has non empty interior.
Its \emph{recession function} (see e.g.\ \cite{Ambrosio-Fusco-Pallara00})
$\psi^\infty:\Rm\to[0,+\infty]$ is defined as
\begin{equation}
  \label{eq:cap2:1}
  \psi^\infty(y):=\lim_{r\to+\infty}\frac{\psi(ry)}r=
  \sup_{r>0}\frac{\psi(ry)}r.
\end{equation}
%%the above formula being independent of $x_0\in D(\psi)$.
$\psi^\infty$ is still convex, lower semicontinuous, and
positively $1$-homogeneous, so that its proper
domain $D(\psi^\infty)$ is a convex cone always containing $0$.
%since $\psi^\infty(0)=0$.
We say that
\begin{equation}
  \label{eq:cap2:25}
  \begin{gathered}
    \text{$\psi$ has a \emph{superlinear growth} if
$\psi^\infty(y)=\infty$ for every $y\neq 0$:
$D(\psi^\infty)=\{0\}$,}\\
\text{$\psi$ has a \emph{sublinear growth} if
$\psi^\infty(y)\equiv 0$ for every $y\in\Rm$.}
  \end{gathered}
\end{equation}
Let now $\gamma\in \RPM(\Rh)$ and $\mmu\in \Mloc(\Rh;\Rm)$ with
$\supp(\mmu)\subset \supp(\gamma)$;
the Lebesgue decomposition of $\mmu$ w.r.t.\ $\gamma$ reads
$\mmu=\ttheta\gamma+\mmu^\perp$, where $\ttheta=\d\mmu/\d\gamma$.
We can introduce a nonnegative Radon measure
$\sigma\in \RPM(\Rh)$ such
that
$\mmu^\perp=\ttheta^\perp\sigma\ll \sigma $, e.g. $\sigma=|\mmu^\perp|$
and we set
\begin{equation}
  \label{eq:cap2:4}
  \Psi^a(\mmu|\gamma):=\int_\Rh \psi(\ttheta(x))\,\d\gamma(x),\qquad
  \Psi^\infty(\mmu|\gamma):=\int_\Rh \psi^\infty(\ttheta^\perp(y))
  \,\d\sigma(y),  
\end{equation}
and finally
\begin{equation}
  \label{eq:cap2:7}
  \Psi(\mmu|\gamma):=
  \Psi^a(\mmu|\gamma)+
  \Psi^\infty(\mmu|\gamma);\quad
  \Psi^\infty(\mmu|\gamma)%=\Psi(\mmu|\gamma)
  =+\infty\
  \text{if}\
  \supp(\mmu)\not\subset\supp(\gamma).
\end{equation}
Since $\psi^\infty$ is $1$-homogeneous, the definition of
$\Psi^\infty$ depends on $\gamma$ only through its support and
it is independent of the particular choice of
$\sigma$ in \eqref{eq:cap2:4}. When $\psi$ has a
superlinear growth then the functional $\Psi$ is finite iff
% and only if
$\mmu\ll\gamma$ and $\Psi^a(\mmu|\gamma)$ is finite;
in this case $\Psi(\mu|\gamma)=\Psi^a(\mmu|\gamma)$.
\begin{theorem}[L.s.c.\ and relaxation of
  integral functionals of measures \cite{Ambrosio-Buttazzo88,Ambrosio-Fusco-Pallara00}]
  \label{thm:lsc_functional_measures}
  Let us consider two sequences
  $\gamma_n\in \RPM(\Rh),\mmu_n\in \Mloc(\Rh;\Rm)$
  weakly$^*$ converging to $\gamma\in \RPM(\Rh)$ and
  $\mmu\in \Mloc(\Rh;\Rm)$ respectively.
  We have
  \begin{equation}
    \label{eq:cap2:8}
    \liminf_{n\up+\infty}\Psi(\mmu_n|\gamma_n)\ge 
    \Psi(\mmu|\gamma).
  \end{equation}
  Let conversely $\mmu,\gamma$ be such that $\Psi(\mmu|\gamma)<+\infty$.
  Then there exists a sequence $\mmu_n=\ttheta_n\gamma\ll\gamma$
  weakly$^*$ converging to $\mmu$ such that
  \begin{equation}
    \label{eq:cap2:9}
    \lim_{n\up+\infty}\Psi^a(\mmu_n|\gamma)=
    \lim_{n\up+\infty}\int_\Rh \psi(\ttheta_n(x))\,\d \gamma(x)=
    \Psi(\mmu|\gamma).
  \end{equation}
\end{theorem}
\begin{theorem}[Montonicity w.r.t.\ $\gamma$]
  If $\gamma_1\le \gamma_2$ then
  \begin{equation}
    \label{eq:cap2:28}
    \Psi(\mmu|\gamma_2)\le \Psi(\mmu|\gamma_1).
  \end{equation}
\end{theorem}
\begin{proof}
  Thanks to Theorem \ref{thm:lsc_functional_measures},
  it is sufficient to prove
  the above inequality for $\mmu\ll\gamma^1$.
  Since $\gamma_1=\theta\gamma_2$, with density $\theta\le 1$
  $\gamma_2$-a.e., we have
  $\mmu=\ttheta^i\gamma^i$ with
  $\ttheta^2=\theta\,\ttheta^1$, and therefore
  \begin{equation}
    \label{eq:scheme3:42}
    \int_\Rd \psi(\ttheta_1)\,\d\gamma_1=
    \int_\Rd \psi(\theta^{-1}\ttheta_2)\theta\,\d\gamma_2\ge
    \int_\Rd \psi(\ttheta_2)\,\d\gamma_2,
  \end{equation}
  where we used the property $\theta\psi(\theta^{-1} x)\ge
  \psi(x)$ for $\theta\le 1$, being $\psi(0)=0$.
  \qed\end{proof}
% Let us now consider a usual convolution kernel
% $\kernel\in C^\infty_{\rm c}(\Rd)$
% satisfying
% \begin{equation}
%   \label{eq:cap2:56}
%   \kernel(x)\ge 0,\quad
%   \int_{\Rd}\kernel(x)\,\d x=1.
%   \kernel_\eps(x):=\eps^{-d}\kernel(x/\eps).
% \end{equation}
\begin{theorem}[Monotonicity with respect to convolution]
  \label{thm:mono_conv}
  If $\kernel\in C^\infty_{\rm c}(\Rd)$ is a convolution kernel
  satisfying $
  %\begin{equation}
  %  \label{eq:cap2:58}
    \kernel(x)\ge 0,\quad
  \int_{\Rd}\kernel(x)\,\d x=1,$
  %\end{equation}
  then
  \begin{equation}
    \label{eq:cap2:57}
    \Psi(\mmu\ast \kernel|\gamma\ast\kernel)\le
    \Psi(\mmu|\gamma).
  \end{equation}
\end{theorem}
The \emph{proof} follows the same argument
of \cite[Lemma 8.1.10]{Ambrosio-Gigli-Savare05}, by observing that
the map $(x,y)\mapsto x\psi(y/x)$ is convex and
positively $1$-homogeneous in $(0,+\infty)\times\Rd$.

\subsection{\bfseries A duality result in convex analysis}
\label{subsec:duality}
Let $X,Y$ be Banach spaces and let $A$ be an open
convex subset of $X$.
We consider a convex (and a fortiori continuous)
function $\phi:A\times Y\to \R$ and its partial Legendre transform
\begin{equation}
  \label{eq:cap2:41}
  \tilde\phi(x,y^*):=\sup_{y\in Y} \la y^*,y\ra-\phi(x,y)\in
  (-\infty,+\infty],
  \quad
  \forall\, x\in A,\ y^*\in Y^*.
\end{equation}
The following duality result is well known in the framework of minimax problems
\cite{Rockafellar68}.
\begin{theorem}
  \label{thm:convex_analysis}
  $\tilde\phi$ is a l.s.c.\ function
  and there exists a convex set $Y^*_o\subset Y^*$ such that
  \begin{equation}
    \label{eq:cap2:42}
    \tilde\phi(x,y^*)<+\infty\quad
    \Leftrightarrow\quad
    y^*\in Y^*_o,
  \end{equation}
  so that $\tilde\phi(\cdot,y^*)\equiv +\infty$ for every $y^*\in
  Y^*\setminus Y^*_o$ and $\phi$ admits the dual representation
  formula
  \begin{equation}
    \label{eq:cap2:44}
    \phi(x,y)=\sup_{y^*\in Y^*_o}\la y,y^*\ra-\tilde\phi(x,y^*)\quad
    \forall\, x\in A,\ y\in Y.
  \end{equation}
  For every $y^*\in Y^*_o$ we have
  \begin{equation}
    \label{eq:cap2:43}
    \text{the map}\quad
    x\mapsto \tilde\phi(x,y^*)\quad\text{is concave (and continuous)
      in $A$}.
  \end{equation}
  Conversely, a function $\phi:A\times Y\to \R$ is convex if
  it admits the dual representation \eqref{eq:cap2:44}
  for a function $\tilde\phi$ satisfying \eqref{eq:cap2:43}.
\end{theorem}
\begin{proof}
  Let us first show that
  \eqref{eq:cap2:43} holds. For a fixed $y^*\in Y^*$,
  $x_0,x_1>0$, $\theta\in [0,1]$, and arbitrary
  $y_{i}\in Y$, 
  we get
  \begin{align*}
    &\tilde\phi((1-\vartheta)x_0+\vartheta x_1,y^*)
    \ge
    \la y^*,(1-\vartheta)y_{0}+\vartheta y_{1}\ra-
    \phi((1-\vartheta)x_{0}+\vartheta x_{1},
    (1-\vartheta)y_{0}+\vartheta y_{1})\\
    &\quad
    \ge
    (1-\vartheta)\Big(\la y^*,y_{0}\ra-
    \phi(x_{0},
    y_{0})\Big)+
    \vartheta\Big(\la y^*,y_{1}\ra-
    \phi(x_{1},
    y_{1})\Big).
  \end{align*}
  Taking the supremum with respect to $y_0,y_1$ we eventually get
  \begin{equation}
    \label{eq:cap2:45}
    \tilde\phi((1-\vartheta)x_0+\vartheta x_1,y^*)\ge
    (1-\vartheta)\tilde\phi(x_0,y^*)+
    \vartheta\tilde\phi(x_1,y^*)
  \end{equation}
  and we conclude that $\tilde\phi(\cdot,y^*)$ is
  concave. In particular, if it takes the value $+\infty$ at some
  point it should be identically $+\infty$ so that \eqref{eq:cap2:42}
  holds.

  The converse implication is even easier, since
  \eqref{eq:cap2:44} exhibits $\phi$ as
  a supremum of continuous and convex functions (jointly in $x\in A, y\in Y$).
\qed\end{proof}

%%%%%%%%%%%%%%%%%%%%%%%%%%%%%%%%%%%%%%%%%%%%%%%%%%%%%%%%%%%%%%%%%%%%%%%%%%
%%%%%%%%%%%%%%%%%%%%%%%%%%%%%%%%%%%%%%%%%%%%%%%%%%%%%%%%%%%%%%%%%%%%%%%%%% 

\section{Action functionals}
\label{sec:action}
The aim of this section is to study some property of
integral functionals of the type
\begin{equation}
  \label{eq:cap2:32}
  \Phi^a(\mu,\nnu|\gamma):=\int_\Rd \phi(\rho,\ww)\,\d\gamma,\quad
  \mu=\rho\gamma\in \RPM(\Rd),\ \nnu=\ww\gamma\in \Mloc(\Rd;\Rd)
\end{equation}
and their relaxation,
when $\phi$ satisfies suitable convexity and
homogeneity
properties.

\subsection{\bfseries Action density functions}
\label{subsec:action_density}
Let us therefore consider a nonnegative
density function $\phi:(0,+\infty)\times \Rd\to [0,+\infty)$
and an exponent $p\in (1,+\infty)$
satisfying the following assumptions 
\begin{subequations}
  \label{subeq:phi_prop}
  \begin{equation}
    \label{eq:cap2:10}
    \phi\quad\text{is convex and (a fortiori) continuous,}
  \end{equation}
  \begin{equation}
    \label{eq:cap2:11}
    \begin{gathered}
      \text{$\ww\mapsto \phi(\cdot,\ww)$ is homogeneous of
        degree $p$, i.e.}\\
      \phi(\rho,\lambda \ww)=|\lambda|^p
      \phi(\rho,\ww)\quad \forall\,\rho>0,\, \lambda\in \R,\,\ww\in \Rd,
    \end{gathered}
  \end{equation}
  \begin{equation}
    \label{eq:cap2:12}
    \exists\, \rho_0>0:\quad
    \phi(\rho_0,\cdot)\quad\text{is non degenerate, i.e.}\qquad
    \phi(\rho_0,\ww)>0\quad
    \forall\, \ww\in \Rd\setminus\{0\}.
  \end{equation}
Let $q=p/(p-1)\in (1,+\infty)$ be
the usual conjugate exponent of $p$.
We denote by $\tilde\phi:(0,+\infty)\times \Rd\to (-\infty,+\infty]$
the partial Legendre transform
\begin{equation}
  \label{eq:cap2:13}
  \frac 1q\tilde\phi(\rho,\zz):=\sup_{\ww\in
    \Rd}\zz\cdot\ww-\frac 1p\phi(\rho,\ww)\quad
  \forall\, \rho>0,\zz\in \Rd.
\end{equation}
\end{subequations}
We collect 
some useful properties of such functions in the following result.
\begin{theorem}
  \label{le:phi_list}
  Let $\phi:(0,+\infty)\times \Rd\to\Rd$ satisfy
  \emph{(\ref{subeq:phi_prop}a,b,c)}. Then
  \begin{enumerate}
  \item For every $\rho>0$ the function $\ww\mapsto
    \phi(\rho,\ww)^{1/p}$ is a norm of $\Rd$ whose dual norm
    is given by $\zz\mapsto \tilde\phi(\rho,\zz)^{1/q}$, i.e.\
    \begin{equation}
      \label{eq:cap3:47}
      \tilde\phi(\rho,\zz)^{1/q}=
      \sup_{\ww\neq 0}\,\frac{\ww\cdot\zz}{\phi(\rho,\ww)^{1/p}},\qquad
      \phi(\rho,\ww)^{1/p}=
      \sup_{\zz\neq
        0}\,\frac{\ww\cdot\zz}{\tilde\phi(\rho,\zz)^{1/q}}.
    \end{equation}
     In particular $\tilde\phi(\cdot,\zz)$ is $q$-homogeneous with
     respect to
     $\zz$.
% and
%     \begin{equation}
%       \label{eq:cap2:33}
%       \ww\cdot\zz\le
%       \phi(\rho,\ww)^{1/p}\,\tilde\phi(\rho,\zz)^{1/q}\le
%       \frac 1p \phi(\rho,\ww)+\frac 1q\tilde\phi(\rho,\zz)\quad
%       \forall\, \rho>0,\ \ww,\zz\in \Rd.
%     \end{equation}
  \item The marginal conjugate function $\tilde\phi$
    takes its values in $[0,+\infty)$ and for every $\zz\in \Rd$
    \begin{equation}
      \label{eq:cap2:14}
      \text{the map }
      \rho\mapsto \tilde\phi(\rho,\zz)
      \quad\text{is \emph{concave and non decreasing} in }(0,+\infty).
    \end{equation}
    In particular, for every $\ww\in \Rd$
    \begin{equation}
      \label{eq:cap2:49}
      \text{the map }
      \rho\mapsto \phi(\rho,\ww)
      \quad\text{is \emph{convex and non increasing} in }(0,+\infty).
    \end{equation}
  \item There exist constants $a,b\ge 0$ such that
  \begin{equation}
    \label{eq:cap2:15}
    \tilde\phi(\rho,\zz)\le \big(a+b\,\rho\big) |\zz|^q,\quad
    \phi(\rho,\zz)\ge  \big(a+b\,\rho\big)^{1-p}|\ww|^p\qquad
    \forall\, \rho>0,\ \zz,\ww\in \Rd.
  \end{equation}
\item
  For every closed interval $[\rho_0,\rho_1]\subset (0,+\infty)$
  there exists a constant $C=C_{\rho_0,\rho_1}>0$ such that for every
  $\rho\in [\rho_0,\rho_1]$
  \begin{equation}
    \label{eq:cap3:45}
    C^{-1}|\ww|^p\le \phi(\rho,\ww)\le C|\ww|^p,\quad
    C^{-1}|\zz|^q\le \tilde\phi(\rho,\zz)\le C|\zz|^q
    \quad
    \forall\,\ww,\zz\in \Rd.
  \end{equation}
  \end{enumerate}
  Equivalently, a function $\phi$ satisfies
  \emph{(\ref{subeq:phi_prop}a,b,c)} if and only if it admits
  the dual representation formula
  \begin{equation}
    \label{eq:cap2:20}
    \frac 1p\phi(\rho,\ww)=\sup_{\zz\in \Rd}\ww\cdot\zz-
    \frac 1q\tilde\phi(\rho,\zz)
    \quad
    \forall\, \rho>0,\ww\in \Rd,
  \end{equation}
  where
  $\tilde\phi:(0,+\infty)\times\Rd\to (0,+\infty)$ is a nonnegative
  function which is 
  convex and $q$-homogeneous
  w.r.t.\ $\zz$ and \emph{concave} with respect to $\rho$.
\end{theorem}
\begin{proof}
  Let us first assume that $\phi$ satisfies
  (\ref{subeq:phi_prop}a,b,c).
  The function $\ww\mapsto \phi(\rho,\ww)^{1/p}$ is
  $1$-homogeneous and its sublevels are convex, i.e.\ it is 
  the gauge function of a (symmetric) convex set
  and therefore it is a (semi)-norm.
%   \begin{quote}
%     {\small Here is a proof by direct calculations: for $\vartheta\in
%       [0,1], \ww_0,\ww_1\in \Rd$, we set
%       $\kappa_i^p:=\phi(\rho,\ww_i)$, $\bar\ww_i:=\kappa_i^{-1}
%       \ww_i$, $\kappa:=(1-\vartheta)\kappa_0+\vartheta \kappa_1$; it
%       is not restricitive to assume $\kappa>0$.  We easily get
%       \begin{align*}
%         &\phi(\rho,(1-\vartheta)\ww_0+\vartheta\ww_1)^{1/p}=
%         \phi(\rho,(1-\vartheta)\kappa_0\bar\ww_0+\vartheta
%         \kappa_1\bar\ww_1)^{1/p}
%         \\
%         &\quad\topref{eq:cap2:11}= \kappa
%         \phi(\rho,(1-\vartheta)\kappa_0\kappa^{-1}\bar\ww_1+
%         \vartheta\kappa_1\kappa^{-1}\bar\ww_2)^{1/p}
%         \topref{eq:cap2:10}\le
%         \kappa^{1-1/p}\Big((1-\vartheta)\kappa_0\phi(\rho,\bar\ww_0)+
%         \vartheta\kappa_1\phi(\rho,\bar\ww_1)\Big)^{1/p}\\
%         &\quad \le \kappa^{1-1/p}\Big((1-\vartheta)\kappa_0+
%         \vartheta\kappa_1\Big)^{1/p}= \kappa=
%         (1-\vartheta)\phi(\rho,\ww_0)^{1/p}+
%         \vartheta\phi(\rho,\ww_1)^{1/p}
%       \end{align*}
%       since $\phi(\rho,\bar\ww_i)=1$; the subadditivity of $\ww\mapsto
%       \phi(\rho,\ww)^{1/p}$ then follows.}
%   \end{quote}
  The concavity of $\tilde\phi$ follows from
  Theorem \ref{thm:convex_analysis};
  taking $\ww=0$ in \eqref{eq:cap2:13}, we easily get that
  $\tilde\phi$ is nonnegative; \eqref{eq:cap2:12} yields,
  for a suitable constant $c_0>0$,
  \begin{equation}
    \label{eq:cap2:48}
    \phi(\rho_0,\ww)\ge c_0|\ww|^p\quad\forall\, \ww\in \Rd,\quad
    \text{so that}\quad
    \tilde\phi(\rho_0,\zz)\le  c_0|\zz|^q<+\infty\quad \forall\,
    \zz\in \Rd.
  \end{equation}
  Still applying Theorem \ref{thm:convex_analysis},
  we obtain that $\rho\mapsto \tilde\phi(\rho,\zz)$
  is finite, strictly
  positive and nondecreasing 
  in the interval $(0,+\infty)$. Since
  $\tilde\phi(\rho,0)=0$
  we easily get
  \begin{equation}
    \label{eq:cap2:17}
    \tilde\phi(\rho,\zz)\le \tilde\phi(\rho_0,\zz)\le 
    c_0|\zz|^q
    \quad\forall\, \zz\in \Rd,\ \rho\in (0,\rho_0);
  \end{equation}
  \begin{equation}
    \label{eq:cap2:18}
    \tilde\phi(\rho,\zz)\le \frac{\rho}\rho_0
    \tilde\phi(\rho_0,\zz)\le \frac{c_0}{\rho_0}\rho|\zz|^q
    \quad\forall\, \zz\in \Rd,\ \rho\in (\rho_0,+\infty).
  \end{equation}
  Combining the last two bounds we get \eqref{eq:cap2:15}.
  \eqref{eq:cap3:45} follows by homogeneity and
  by the fact that 
  the continuous map $\phi$ has a maximum and a strictly positive minimum
  on the compact set $[\rho_0,\rho_1]\times\{\ww\in \Rd:|\ww|=1\}$.
  
  The final assertion concerning 
  \eqref{eq:cap2:20} still follows
  by Theorem \ref{thm:convex_analysis}.
\qed\end{proof}
\subsection{\bfseries Examples}
\label{subsec:phi_examples}
\begin{example}
  \upshape
  \label{ex:main_alpha}
  Our main example is provided by the function
  \begin{equation}
    \label{eq:scheme3:1}
    \phi_{2,\alpha}(\rho,\ww)=\frac{|\ww|^2}{\rho^\alpha},
    \quad
    \tilde\phi_{2,\alpha}(\rho,\zz):=\rho^\alpha|\zz|^2,
    \qquad
    0\le \alpha\le 1.
  \end{equation}
  Observe that $\phi_{2,\alpha}$
  is positively $\theta$-homogeneous,
      $\theta:=2-\alpha$, i.e.
  \begin{equation}
    \label{eq:cap2:36}
    \phi_{2,\alpha}(\lambda\rho,\lambda \ww)=\lambda^\theta \phi(\rho,\ww)\quad
    \forall\, \lambda,\rho>0,\ \ww\in \Rd.
  \end{equation}
  It can be considered as
  a family of
  interpolating densities
  between the case 
  $\alpha=0$, when
  \begin{equation}
    \label{eq:Cap2:24}
    \phi_{2,0}(\rho,\ww):=|\ww|^2,
  \end{equation}
  and $\alpha=1$, corresponding to the $1$-homogeneous functional
  \begin{equation}
    \label{eq:Cap2:25}
    \phi_{2,1}(\rho,\ww):=\frac{|\ww|^2}\rho.
  \end{equation}
  % Denoting by $\D$ the differential with respect to
%   $\ww$, a direct calculation shows that
%   \begin{equation}
%     \label{eq:scheme:14}
%     \D^2\phi_{2,\alpha}(\rho,\ww)\vv\cdot\vv\ge 2\beta\, \phi_{2,\alpha}(\rho,|
%     \yy|^2)\quad
%     \forall\, \vv=(x,\yy)\in \R^{d+1},\quad
%     \beta:=\frac{1-\alpha}{1+\alpha}.
%   \end{equation}
\end{example}
\begin{example}
  \upshape
  \label{ex:main_h}
  More generally,
  we introduce a concave function $h:(0,+\infty)\to (0,+\infty)$,
  which is a fortiori continuous and nondecreasing, and
  we consider the density function
  \begin{equation}
    \label{eq:cap2:34}
    \phi(\rho,\ww):=\frac{|\ww|^2}{h(\rho)},\quad
    \tilde\phi(\rho,\zz):=h(\rho)|\ww|^2.
  \end{equation}
  If $h$ is of class $C^2$, we can express
  the concavity condition in terms of the function $g(\rho):=1/h(\rho)$
  as
  \begin{equation}
    \label{eq:cap2:35}
    \text{$h$ is concave}\quad
    \Leftrightarrow\quad
    g''(\rho)g(\rho)\ge 2\big(g'(\rho)\big)^2\quad
    \forall\, \rho>0,
  \end{equation}
  which is related to a condition introduced in
  \cite[Section 2.2, (2.12c)]{Arnold-Markowich-Toscani-Unterreiter00}
  to study entropy functionals.
  % It is not difficult to check that $\phi$ satisfies
%   the refined coercivity property \eqref{eq:scheme:14}
%   if
%    \begin{equation}
%     \label{eq:cap3:3}
%     h^{\frac 1\alpha}\quad\text{is concave.}
%   \end{equation}
%   In terms of $g$ this is equivalent to
%   \begin{equation}
%     \label{eq:cap2:37}
%     g(\rho)g''(\rho)\ge (1+\alpha^{-1})\big(g'(\rho)\big)^2.
%   \end{equation}
\end{example}
\begin{example}
  \upshape
    We consider matrix-valued
    functions $\sfH,\sfG :(0,+\infty)\to \M^{d\times d}$ such that
\begin{equation}
  \label{eq:Cap2:26}
  \sfH(\rho),\sfG(\rho)
  \quad\text{are symmetric and positive definite,
    $\sfH(\rho)=\sfG^{-1}(\rho)$}
  \quad
  \forall\, \rho>0.
\end{equation}
They induce the action density
$\phi:(0,+\infty)\times \R^d\to [0,+\infty)$ defined as
\begin{equation}
  \label{eq:scheme:9}
  \phi(\rho,\ww):=\Scalar{\sfG(\rho)\ww}\ww=
  \Scalar{\sfH^{-1}(\rho)\ww}\ww.
\end{equation}
Taking into account Theorem \ref{le:phi_list},
$\phi$ satisfies conditions (\ref{subeq:phi_prop}) if
and only if the maps
\begin{equation}
  \label{eq:Cap2:27}
  \rho\mapsto \Scalar{\sfH(\rho)\ww}\ww\quad\text{are concave in
    $(0,+\infty)$}
  \quad
  \forall\,\ww\in \R^d.
\end{equation}
Equivalently,
\begin{equation}
  \label{eq:cap2:38}
  \sfH((1-\vartheta)\rho_0+\vartheta\rho_1)\ge
  (1-\vartheta)\sfH(\rho_0)+
  \vartheta\sfH(\rho_1)\quad
  \text{as quadratic forms.}
\end{equation}
When $\sfG$ is of class $C^2$ 
this is also equivalent to ask that
\begin{equation}
  \label{eq:scheme:3}
  \sfG''(\rho)\ge 2\sfG'(\rho)\sfH(\rho)\sfG'(\rho)
    \qquad
    \forall\, \rho>0,
  \end{equation}
  in the sense of the associated quadratic forms.
  In fact, differentiating $\sfH=\sfG^{-1}$ 
  with respect to $\rho$ we get
  \begin{displaymath}
    \sfH'=-\sfH\,\sfG'\,\sfH,\qquad
    \sfH''=-\sfH\,\sfG''\,\sfH +2\sfH\,\sfG'\,\sfH\,\sfG'\,\sfH,
    % \frac{d^2}{d\rho^2}\frac 1{g(\rho)}=
%     \frac{-g''(\rho)g(\rho)+2\big(g'(\rho)^2\big)}{g^3(\rho)}\le 0.
  \end{displaymath}
  so that
  \begin{displaymath}
    \frac{\d^2}{\d^2 \rho}\Scalar{\sfH(\rho)\ww}\ww=
    -\Scalar{\sfG''\tilde\ww}{\tilde\ww}+2
    \Scalar{\sfG'\sfH\sfG'\tilde\ww}{\tilde\ww}\quad\text{where}\quad
    \tilde\ww:=\sfH\ww;
  \end{displaymath}
  we eventually recall that $\sfH(\rho)$ is invertible
  for every $\rho>0$.
  % %
%   Evaluating now the second derivative of $\phi$
%   along the direction of the vector $\zz=(x,\yy)\in \R\times\R^d$,
%   we easily get
%   \begin{equation}
%     \label{eq:scheme:10}
%     \Scalar{\D^2\phi(\rho,\ww)\zz} \zz=
%     x^2\Scalar{\sfG''(\rho)\ww}\ww+4
%     x\Scalar{\sfG'(\rho)\ww}\yy+
%     2\Scalar{\sfG(\rho)\yy}\yy.
%    %  \\&=
% %     \Big|\sqrt{g''(\rho)}|x\ww|-\sqrt{2g(\rho)}|\yy|\Big|
% %     +2\sqrt{2g(\rho)g''(\rho)}|x\ww|\,|\yy|+4g'(\rho)
% %     x\ww\cdot\yy\ge0%
%   \end{equation}
%   Minimizing with respect to $x\in \R$ we get
%   \begin{equation}
%     \label{eq:Cap2:32}
%     \Scalar{\D^2\phi(\rho,\ww)\zz} \zz\ge
%     2\frac{\Scalar{\sfG''(\rho)\ww}\ww
%     \Scalar{\sfG(\rho)\yy}\yy-
%     2\Scalar{\sfG'(\rho)\ww}\yy^2}{\Scalar{\sfG''(\rho)\ww}\ww}\quad
%   \text{if }\Scalar{\sfG''(\rho)\ww}\ww>0.
%   \end{equation}
%   If we denote by $\alpha\in [0,1]$ the best constant
%   such that (compare with \eqref{eq:cap2:37})
%   \begin{gather}
%     \label{eq:Cap2:31}
%     \sfG''(\rho)\ge
%     (1+\alpha^{-1})\sfG'(\rho)\sfH(\rho)\sfG'(\rho)
%     \quad
%     \forall\,\rho>0,
%     \intertext{or, equivalently,}
%     \label{eq:Cap2:33}
%     (1+\alpha^{-1})\Scalar{\sfG'(\rho)\ww}\yy^2
%     \le  \Scalar{\sfG''(\rho)\ww}\ww
%     \Scalar{\sfG(\rho)\yy}\yy
%     \quad
%     \forall\, \rho>0,\,\yy,\ww\in \R^d,
%   \end{gather}
%   we get that $\phi$ satisfies \eqref{eq:scheme:14}, too.
\end{example}
\begin{example}
  \label{ex:hp}
  \upshape
  Let $\|\cdot\|$ be any norm in $\Rd$ with
  dual norm $\|\cdot\|_*$, and let $h:(0,+\infty)\to
  (0,+\infty)$ be a concave (continuous, nondecreasing) function
  as in Example \ref{ex:main_h}.
  We can thus consider
  \begin{equation}
    \label{eq:cap2:39}
    \phi(\rho,\ww):=
    h(\rho)\left\|\frac{\ww}{h(\rho)}\right\|^p,\qquad
    \tilde\phi(\rho,\zz):=h(\rho)\|\zz\|_*^q.
  \end{equation}
  See \cite{MarechalI05,MarechalII05} for a in-depth study of this class 
  of functions.
\end{example}
\begin{example}[$(\alpha\text{-}\theta)$-homogeneous functionals]
  \label{ex:hom}
  \upshape
  In the particular case $h(\rho):=\rho^\alpha$
  the functional $\phi$ of the previous example is jointly positively
  $\theta$-homogeneous, with $\theta:=\alpha+(1-\alpha)p$.
  This is in fact the most general example of $\theta$-homogeneous
  functional,
  since
  if $\phi$ is $\theta$-positively homogeneous, $1\le\theta\le p$, then
  \begin{equation}
    \label{eq:cap2:40}
    \phi(\rho,\ww)=\rho^\theta\phi(1,\ww/\rho)=
    \rho^{\theta-p}\phi(1,\ww)=
    \rho^\alpha\|\ww/\rho^\alpha\|^p,\quad
    \alpha=\frac{p-\theta}{p-1},
  \end{equation}
  where $\|\ww\|:=\phi(1,\ww)^{1/p}$ is a norm in $\Rd$
  by Theorem \ref{le:phi_list}.
  The dual marginal density $\tilde\phi$ in this case takes the form
  \begin{equation}
    \label{eq:cap2:53}
    \tilde\phi(\rho,\zz)=\rho^\alpha \|\zz\|^q_*\quad
    \forall\, \rho>0,\ \zz\in \Rd,
  \end{equation}
  and it is $q+\alpha$-homogeneous.
  Notice that $\alpha$ and $\theta$ are related by
  \begin{equation}
    \label{eq:cap2:55}
    \frac\theta p+\frac\alpha q=1.
  \end{equation}
  In the particular case when $\|\cdot\|=\|\cdot\|_*=|\cdot|$
  is the Euclidean norm, we set
  as in \eqref{eq:cap2:34}
  \begin{equation}
    \label{eq:cap3:42}
    \phi_{p,\alpha}(\rho,\ww):=\rho^\alpha\left|\frac{\ww}{\rho^\alpha}
    \right|^p,\quad
    \tilde\phi_{q,\alpha}(\rho,\zz):=\rho^\alpha |\zz|^q,\qquad
    0\le \alpha\le 1.
  \end{equation}
\end{example}

\subsection{\bfseries The action functional on measures}
\label{subsec:action_functional}
\paragraph{Lower semicontinuity envelope and recession function.}

Thanks to the monotonicity property \eqref{eq:cap2:49}, we can extend
$\phi$ also for $\rho=0$ by setting
for every $\ww\in \Rd$
\begin{equation}
  \label{eq:cap2:50}
  \phi(0,\ww)=\sup_{\rho>0}\phi(\rho,\ww)=
  \lim_{\rho\down0}\phi(\rho,\ww);\quad
  \text{in particular }
  \begin{cases}
    \phi(0,{\bf 0})=0,&\\
    \phi(0,\ww)>0&\text{if $\ww\neq{\bf 0}$}.
  \end{cases}
\end{equation}
When $\rho<0$ we simply set $\phi(\rho,\ww)=+\infty$,
observing that this extension is lower semicontinuous in
$\R\times\Rd$.
It is not difficult to check that
$\tilde\phi(0,\cdot)$ satisfies an analogous
formula
\begin{equation}
  \label{eq:cap2:51}
  \tilde\phi(0,\zz)=\sup_{\ww\in\Rd}\zz\cdot\ww-\phi(0,\ww)=
  \inf_{\rho>0}\tilde\phi(\rho,\zz)=
  \lim_{\rho\down0}\tilde\phi(\rho,\zz)\qquad
  \forall\, \zz\in \Rd.
\end{equation}
Observe that, as in the $(\alpha\text{-}\theta)$-homogeneous case
of Example \ref{ex:hom} with
$\alpha>0$,
\begin{equation}
  \label{eq:cap3:43}
  \tilde\phi(0,\zz)\equiv 0\quad
  \Rightarrow\quad
  \phi(0,\ww)=
  \begin{cases}
    +\infty&\text{if }\ww\neq 0\\
    0&\text{if }\ww=0.
  \end{cases}
\end{equation}
As in \eqref{eq:cap2:1},
we also introduce the \emph{recession functional}
\begin{equation}
  \label{eq:scheme:4}
  \phi^\infty(\rho,\ww)=\sup_{\lambda >0}
  \frac 1\lambda\phi(\lambda\rho,\lambda \ww)
  =\lim_{\lambda \uparrow+\infty}
  \frac 1\lambda\phi(\lambda\rho,\lambda \ww)=
  \lim_{\lambda \uparrow+\infty}
  \lambda^{p-1}\phi(\lambda\rho,\ww).
\end{equation}
$\phi^\infty$ is still convex, $p$-homogeneous w.r.t.\ $\ww$, and l.s.c.
with values in $[0,+\infty]$; moreover, it is $1$-homogeneous
so that it can be expressed as
\begin{equation}
  \label{eq:scheme:6bis}
  \phi^\infty(\rho,\ww)=
  \begin{cases}
    \frac{\varphi^\infty(\ww)}{\rho^{p-1}}=
    \rho\,\varphi^\infty(\ww/\rho)
    &\quad\text{if }\rho\neq 0,\\
    +\infty&\quad\text{if }\rho=0\ \text{and}\ \ww\neq 0,
  \end{cases}  
\end{equation}
where $\varphi^\infty:\Rd\to[0,+\infty]$ is a convex
and $p$-homogeneous function which is non degenerate,
i.e.\ $\varphi^\infty(\ww)>0$ if $\ww\neq0$.
$\varphi^\infty$ admits a dual representation,
based on
\begin{equation}
  \label{eq:cap2:52}
  \tilde\varphi^\infty(\zz):=
  \inf_{\lambda>0} \frac 1\lambda \tilde\phi(\lambda,\zz)
  =\lim_{\lambda\up+\infty} \frac 1\lambda \tilde\phi(\lambda\rho,\zz).
\end{equation}
$\tilde\varphi^\infty$ is finite, convex, nonnegative, and $q$-homogeneous,
so that $\tilde\varphi^\infty(\zz)^{1/q}$
is a seminorm, which does not vanish at $\zz\in \Rd$
if and only if $\rho\mapsto \tilde\phi(\rho,\zz)$ has a linear
growth when $\rho\up+\infty$.
It is easy to check that 
\begin{equation}
  \label{eq:cap3:48}
  \varphi^\infty(\ww)^{1/p}=\sup\Big\{\ww\cdot\zz:\tilde\varphi^\infty(\zz)\le 1\Big\}.
\end{equation}
In the case $\tilde\phi$ has a sublinear
growth w.r.t.\ $\rho$, as for $(\alpha\text{-}\theta)$-homogeneous
functionals with $\alpha<1$ (see Example \ref{ex:hom}), 
we have in particular
\begin{equation}
  \label{eq:scheme:6}
  \text{when $\tilde\varphi^\infty(\zz)\equiv0$},\qquad
  \varphi^\infty(\ww)=
  \begin{cases}
    +\infty&\quad\text{if }\ww\neq 0,\\
    0&\quad\text{if }\ww=0.
  \end{cases}
 \end{equation}

\paragraph{The action functional.}
Let $\gamma,\mu\in \RPM(\Rd)$ be
nonnegative Radon measures and let
$\nnu\in \Mloc(\Rd;
\R^d)$
be a vector Radon measure on $\Rd$.
We assume that $\supp(\mu),\supp(\nnu)\subset \supp(\gamma)$,
and we write their Lebesgue decomposition with respect
to the reference measure $\gamma$
\begin{equation}
  \label{eq:scheme:7}
  \mu:=\rho\gamma+ \mu^\perp,\quad
  \nnu:=\ww\gamma+ \nnu^\perp.
\end{equation}
We can always introduce a nonnegative Radon measure $\sigma\in
\FPM(\ClosedDomain)$ such that
$\mu^\perp =\rho^\perp\sigma\ll \sigma, \nnu^\perp=
\ww^\perp\sigma\ll \sigma$, e.g.
$\sigma:=\mu^\perp+|\nnu^\perp|$.
We can thus define the \emph{action functional}
\begin{equation}
  \label{eq:scheme:8}
  \Phi(\mu,\nnu|\gamma)=\Phi^a(\mu,\nnu|\gamma)+
  \Phi^\infty(\mu,\nnu|\gamma):=
  \int_{\Rd}\phi(\rho,\ww)\,\d\gamma+
  \int_\Rd \phi^\infty(\rho^\perp,\ww^\perp)\,\d\sigma.
\end{equation}
Observe that, being $\phi^\infty$ $1$-homogeneous,
this definition is independent of $\sigma$.
We will also use a localized version of $\Phi$: if $B\in \BorelSets{\Rd}$ we set
\begin{equation}\label{eq:localized_action}
 \Phi(\mu,\nnu|\gamma,B):=\int_{B}\phi(\rho,\ww)\,\d\gamma+
  \int_B \phi^\infty(\rho^\perp,\ww^\perp)\,\d\sigma.
\end{equation} 
\begin{lemma}
  \label{le:Phi_char}
  Let $\mu=\rho\gamma+\mu^\perp,
  \nnu=\ww\gamma+\nnu^\perp$ be such that $\Phi(\mu,\nnu|\gamma)$ is
  finite.
  Then $\nnu^\perp=\ww^\perp\mu^\perp\ll\mu^\perp$ and
  \begin{equation}
    \label{eq:scheme:17}
    \Phi^\infty(\mu,\nnu|\gamma)=
    \int_\Rd \varphi^\infty(\ww^\perp)\,\d\mu^\perp,\
    \Phi(\mu,\nnu|\gamma)=
    \int_\Rd \phi(\rho,\ww)\,\d\gamma+
    \int_\Rd \varphi^\infty(\ww^\perp)\,\d\mu^\perp.
  \end{equation}
  Moreover, if
  $\tilde\phi$ has a \emph{sublinear growth with respect to $\rho$}
  (e.g.\ in the $(\alpha\text{-}\theta)$-homogeneous case of
  Example \ref{ex:hom}, with $\alpha<1$)
  then $\tilde \varphi^\infty(\cdot)\equiv0$
  and
  \begin{equation}
    \label{eq:scheme:16}
    \Phi(\mu,\nnu)<+\infty\
    \Rightarrow\
    \nnu=\ww\cdot \gamma\ll\gamma,\quad
    \Phi(\mu,\nnu)=
    \Phi^a(\mu,\nnu)=\int_\Rd \phi(\rho,\ww)\,\d\gamma,
  \end{equation}
  independently on the singular part $\mu^\perp$. 
\end{lemma}
\begin{proof}
  Let $\tilde\sigma\in \RPM(\Rd)$ any
  measure such that $\mu^\perp\ll\tilde\sigma,|\nnu^\perp|\ll\tilde\sigma$ 
  so that
  $\Phi^\infty(\mu,\nnu|\gamma)$ can be represented as
  \begin{displaymath}
    \Phi^\infty(\mu,\nnu|\gamma)=
    \int_\Rd
    \phi^\infty(\tilde\rho^\perp,\tilde\ww^\perp)\,\d\tilde\sigma,\quad
    \tilde\rho^\perp=\frac{\d\mu^\perp}{\d\tilde\sigma},\
    \tilde\ww^\perp=\frac{\d\nnu^\perp}{\d\tilde\sigma}.
  \end{displaymath}
  When $\Phi^\infty(\mu,\nnu|\gamma)<+\infty$,
  \eqref{eq:scheme:6bis} yields
  $\tilde\ww^\perp(x)=0$ for 
  $\tilde\sigma$-a.e.\ $x$ such that $\tilde\rho^\perp(x)=0$.
  It follows that
  \begin{equation}
    \label{eq:cap3:49}
    \Phi(\mu,\nnu)<+\infty\quad\Rightarrow\quad
    \nnu^\perp\ll\mu^\perp,
  \end{equation}
  so that one can always choose $\tilde\sigma=\mu^\perp$, $\tilde\rho^\perp=1$,
  and decompose $\nnu^\perp$ as $\ww^\perp\mu^\perp$
  obtaining \eqref{eq:scheme:17}.
  \eqref{eq:scheme:16} is then an immediate consequence of
  \eqref{eq:scheme:6}.
\qed\end{proof}
\begin{remark}
  \label{rem:w_abs}
  \upshape
  When $\tilde\phi(0,\zz)\equiv 0$ (e.g. in the
  $(\alpha\text{-}\theta)$-homogeneous
  case of Example \ref{ex:hom}, with $\alpha>0$)
  the density $\ww$ of $\nnu$ w.r.t.\ $\gamma$ vanishes
  if $\rho$ vanishes, i.e.
  \begin{equation}
    \label{eq:cap3:44}
    \Phi(\mu,\nnu|\gamma)<+\infty\quad
    \Rightarrow\quad
    \ww(x)=0\text{ if $\rho(x)=0$, for $\gamma$-a.e.\ $x\in \Rd$.}
  \end{equation}
  In particular $\nnu^a$ is absolutely continuous also with respect to
  $\mu$.
\end{remark}

Applying Theorem \ref{thm:lsc_functional_measures} we immediately get
\begin{lemma}[Lower semicontinuity and
  approximation of the action functional]
  \label{le:action_lsc}
  The action functional is lower semicontinuous with respect to
  weak$^*$ convergence of measures, i.e.\ if
  \begin{displaymath}
    \mu_n\weaksto \mu,\ \gamma_n\weaksto \gamma\quad
    \text{weakly$^*$ in }\RPM(\Rd),\qquad
    \nnu_n\weaksto \nnu\quad
    \text{in $\Mloc(\Rd;\R^d)$ as $n\up+\infty$},
  \end{displaymath}
  then
  \begin{displaymath}
    \liminf_{n\uparrow \infty}\Phi(\mu_n,\nnu_n|\gamma_n)\ge
    \Phi(\mu,\nnu|\gamma).
  \end{displaymath}
 %  Moreover, for every $\mu\in\FPM(\Rd),\nnu\in
%   \mathcal M(\Rd;\Rd)$ with $\Phi(\mu,\nnu|\gamma)<+\infty$
%   there exists a sequence
%   \begin{equation}
%     \label{eq:scheme3:6}
%     \mu_n:=\rho_n\gamma,\quad
%     \nnu_n:=\ww_n\gamma,\quad
%     \rho_n\in C^0_b(\Rd) \text{ with }
%     \inf \rho_n>0,\quad
%     \ww_n\in C^0_b(\Rd;\Rd)
%   \end{equation}
%   such that
%   \begin{equation}
%     \label{eq:scheme3:7}
%     \mu_n\weakto \mu,\quad
%     \nnu_n\weakto \nnu,\quad
%     \lim_{n\to\infty}\int_\Rd \phi(\rho_n(x),\ww_n(x))\,\d\gamma(x)=
%     \Phi(\mu,\nnu).
%   \end{equation}
\end{lemma}
\paragraph{Equiintegrability estimate.}
We collect in this section some basic estimates
on $\phi$ which will turn to be useful in the sequel.
Let us first introduce the notation 
\begin{gather}
  \label{eq:phinorm}
  \dphinorm\zz:=\tilde\phi(1,\zz)^{1/q},\quad
  \phinorm\ww:=\phi(1,\ww)^{1/p},\quad
  \eta^{-1}|\zz|\le \dphinorm\zz\le \eta|\zz|,\\
  \label{eq:cap2:29}
  \Gamma_\phi:=\Big\{(a,b):\sup_{\dphinorm\zz=1}\tilde\phi(\rho,\zz)\le
  a+b\rho\Big\},\quad
  \sfh(\rho):=\inf \Big\{a+b\rho:(a,b)\in \Gamma_\phi\Big\},\\
  \label{eq:cap2:54}
  H(s,\rho):=s\sfh(\rho/s)=\inf \Big\{as+b\rho:(a,b)\in \Gamma_\phi\Big\}.
\end{gather}
Observe that $\sfh$ is a concave increasing function defined in
$[0,+\infty)$, satisfying, 
in the homogeneous case
$\sfh(\rho)=h(\rho)=\rho^\alpha$.
It
provides the bounds 
\begin{equation}
  \label{eq:cap2:30}
  \begin{aligned}
    \tilde\phi(\rho,\zz)&\le \sfh(\rho)\dphinorm\zz^q,&\quad
    \phinorm\ww &\le
    \sfh(\rho)^{1/q}\phi(\rho,\ww)^{1/p},\\
    \tilde\varphi^\infty(\zz)&\le \sfh^\infty
    \dphinorm\zz^q,&\quad \phinorm\ww &\le
    \big(\sfh^\infty\big)^{1/q}\varphi^\infty(\ww)^{1/p},\quad
    \text{if }\sfh^\infty
    :=\lim\limits_{\lambda\up+\infty}\lambda^{-1}\sfh(\lambda)>0.
  \end{aligned}
\end{equation}
Observe that when $\sfh^\infty=0$ then $\tilde\varphi^\infty\equiv0$ and
$\varphi^\infty(\ww)$ is given by \eqref{eq:scheme:6}.
\begin{proposition}[Integrability estimates]
  Let $\zeta$ be a nonnegative
  Borel function such that
  $$\mu(\zeta^q):=\int_\Rd \zeta^q\,\d\mu\quad
  \text{and}\quad
  \gamma(\zeta^q):=
  \int_\Rd\zeta^q\,\d\gamma\quad
  \text{are finite},$$
  and let
  $Z:=\big\{x\in \Rd:\zeta(x)>0\big\}$.
  If $\Phi(\mu,\nnu|\gamma)<+\infty$ we have
  \begin{equation}
    \label{eq:Cap3:1}
    \int_\Rd\zeta(x)\,\d\phinorm\nnu(x)
    \le \Phi^{1/p}\big(\mu,\nnu|\gamma,Z\big)\,
    H^{1/q}\big(\gamma(\zeta^q),\mu(\zeta^q)\big).
  \end{equation}
  In particular, for every Borel set $A\in \BorelSets\Rd$ we have
   \begin{equation}
    \label{eq:scheme3:5}
    \phinorm\nnu(A)\le
    \Phi^{1/p}\big(\mu,\nnu|\gamma,A\big)\,H^{1/q}(\gamma(A),\mu(A)\big).
  \end{equation}
 %  If moreover $\sfm_q(\mu)<+\infty$ we
%   can bound the first moment of $|\nnu|$ by
%   \begin{equation}
%     \label{eq:cap3:1}
%     \sfm_1(|\nnu|)=\int_\Rd |x|\,d|\nnu|\le 
%     \Phi^{1/p}(\mu,\nnu|\gamma)\,
%     H^{1/q}\big(\sfm_q(\gamma),\sfm_q(\mu)\big).
%   \end{equation}
\end{proposition}
\begin{proof}
  It is sufficient to prove
  \eqref{eq:Cap3:1}.
  Observe that
   if $(a,b)\in \Gamma_\phi$ then
   $a\ge 0,$ and $\sfh^\infty\le b$ so that by \eqref{eq:cap2:30}
   we have
  \begin{align*}
    &\int_\Rd \zeta(x)\,\d\phinorm\nnu(x)\le
    \int_Z\zeta \phinorm\ww\,\d\gamma+ \int_Z
    \zeta \phinorm{\ww^\perp}\,\d\mu^\perp
    \\&\le
    \Big(\int_Z \phi(\rho,\ww)\,\d\gamma\Big)^{1/p}
    \Big(\int_Z \zeta^q \sfh(\rho)\,\d\gamma\Big)^{1/q}
    +\!
    \Big(\int_Z \varphi^\infty(\ww^\perp)\,\d\mu^\perp\Big)^{1/p}
    \Big(\sfh^\infty\!\!\!\int_Z \zeta^q\,\d\mu^\perp\Big)^{1/q}
    \\&\le
    \Big(\Phi(\mu,\nnu|\gamma,Z)\Big)^{1/p}
    \Big(a\int_\Rd \zeta^q\,\d\gamma+
    b\int_\Rd \zeta^q\,\d\mu\Big)^{1/q},
  \end{align*}
  Taking the infimum of the last
  term over all the couples $(a,b)\in \Gamma_\phi$ we obtain
  \eqref{eq:scheme3:5}.
%   To prove \eqref{eq:cap3:1} we argue in a similar way:
%   \begin{align*}
%     \int_\Rd |x|\,d|\nnu|&=
%     \int_\Rd |x|\,|\ww(x)|\,\d\gamma+
%     \int_\Rd |x|\,|\ww^\perp(x)|\,\d\mu^\perp
%     \\&\le
%     \Big(\int_\Rd \phi(\rho,\ww)\,\d\gamma\Big)^{1/2}
%     \Big(\int_\Rd h(\rho)|x|^2\,\d\gamma\Big)^{1/2}+
%     \Big(g^\infty\int_\Rd |\ww^\perp|^2\,\d\mu^\perp\Big)^{1/2}
%     \Big(h^\infty|x|^2\,\d\mu^\perp\Big)^{1/2}
%     \\&\le
%     \Big(\Phi(\mu,\nnu)\Big)^{1/2}
%     \Big(\alpha \sfm_2(\gamma)+\beta \sfm_2(\mu)\Big)^{1/2}.
%   \end{align*}
%   \eqref{eq:cap3:1} still follows by taking the infimum of the last
%   term
%   w.r.t.\ $(\alpha,\beta)\in \Gamma(h)$
\qed\end{proof}

%%%%%%%%%%%%%%%%%%%%%%%%%%%%%%%%%%%%%%%%%%%%%%%%%%%%%%%%%%%%%%%%%%%%%%%%%%
%%%%%%%%%%%%%%%%%%%%%%%%%%%%%%%%%%%%%%%%%%%%%%%%%%%%%%%%%%%%%%%%%%%%%%%%%% 

\section{Measure valued solutions of the
  continuity equation in $\R^d$}
\label{sec:continuity}
In this section
we collect some results
on the continuity equation 
\begin{equation}
  \label{eq:continuity1}
  \partial_t {\mu}_t+\nabla\cdot \nnu_t=0
  \qquad\hbox{\text in $\R^d\times (0,\FinalT)$,}
\end{equation}
which we will need in the sequel.
Here $\mu_t,\nnu_t$ are Borel families of measures
(see e.g.\ \cite{Ambrosio-Gigli-Savare05})
in $\RPM(\Rd)$
and $\Mloc(\Rd;\Rd)$ respectively,
defined for $t$ in the open interval $(0,\FinalT)$,
such that
\begin{equation}
  \label{eq:integrability_of_vt}
  \int_0^\FinalT \mu_t(\Ball R)\,\d t<+\infty,\qquad
  V_R:=\int_0^T |\nnu_t|(\Ball R)\,\d t<+\infty
  \qquad
  \forall\, R>0,
\end{equation}
and we suppose that \eqref{eq:continuity1} holds in the sense
of distributions, i.e.
\begin{equation}
  \label{eq:distribution_sense}
    \int_0^\FinalT\int_{\R^d}\partial_t\zeta(x,t)
    \,\d\mu_t(x)\,\d t+
    \int_0^\FinalT\int_{\R^d}
    \nabla_x\zeta(x,t)\cdot \,\d\nnu_t(x)\,\d t
   =0
 \end{equation}
 for every $\zeta\in C^1_c(\R^d\times (0,\FinalT)).$
 Thanks to the disintegration theorem
 \cite[4, III-70]{Dellacherie-Meyer78}, we can identify $(\nnu_t)_{t\in
   (0,T)}$
 with the
 measure $\nnu=\int_0^T \nnu_t\,\d t
 \in \Mloc(\Rd\times(0,T);\Rd)$ defined by the formula
 \begin{equation}
   \label{eq:DS:4}
   \la \nnu,\zzeta\ra=
   \int_0^T
   \left(\int_\Rd \zzeta(x,t)\cdot \d\nnu_t(x)\right)\,\d t\qquad
   \forall\, \zzeta\in C^0_{\rm c}(\Rd\times(0,T);\Rd).
 \end{equation}
 \subsection{\bfseries Preliminaries}
\label{subsec:cont1}
Let us first adapt the results of
\cite[Chap. 8]{Ambrosio-Gigli-Savare05} (concerning a family of
\emph{probability}
measures $\mu_t$) to the more general case of Radon measures.
% Thanks to \eqref{eq:integrability_of_vt}, we can also consider
% bounded test functions $\zeta$, with bounded gradient,
% whose support has a
% compact projection in $(0,\FinalT)$
% (that is, the support in $x$ need not be compact):
% it suffices to approximate 
%$\zeta$ by $\zeta\chi_R$ where $\chi_R\in C^\infty_c(\R^d)$,
% $0\leq\chi_R\leq 1$, $|\nabla\chi_R|\leq 2$ and $\chi_R=1$ on $\Ball R$.
First of all we recall some  (technical) preliminaries.
\begin{lemma}[Continuous representative]
  \label{le:continuous_representative}
  Let $\mu_t,\nnu_t$ be Borel families of measures
  satisfying
  \eqref{eq:integrability_of_vt}
  and \eqref{eq:distribution_sense}.
  Then there exists a unique \emph{weakly$^*$ continuous} curve
  $t\in [0,\FinalT]\mapsto \tilde\mu_t\in \RPM(\Rd)$
  such that $\mu_t=\tilde \mu_t$ for $\Leb 1$-a.e.\ $t\in (0,\FinalT)$;
  if $\zeta\in C^1_c(\R^d\times[0,\FinalT])$
  and $t_1\leq t_2\in [0,\FinalT]$, we have
  \begin{equation}
    \label{eq:traces}
    \begin{aligned}
      \int_{\R^d}\zeta_{t_2}\,\d\tilde\mu_{t_2}-
      \int_{\R^d}\zeta_{t_1}\,\d\tilde\mu_{t_1}=
      \int_{t_1}^{t_2}\int_{\R^d}
      \partial_t\zeta\,
      \d\mu_t(x)\,\d t
      +
      \int_{t_1}^{t_2}\int_{\R^d}
      \nabla\zeta\cdot \,\d\nnu_t(x)\,\d t,
  \end{aligned}
\end{equation}
and the mass of $\tilde\mu_t$ can be uniformly bounded by
\begin{equation}
  \label{eq:cap4:33}
  \sup_{t\in [0,T]}\tilde\mu_t(\Ball {R})\le
  \tilde\mu_s(\Ball{2R})+2R^{-1}V_{2R}\qquad
  \forall\, s\in [0,T].
\end{equation}
Moreover, if $\tilde\mu_s(\Rd)<+\infty$ for some $s\in [0,T]$ and
$\lim_{R\up+\infty}R^{-1}V_R=0$, 
then the total mass $\tilde\mu_t(\Rd)$
is (finite and) constant.
\end{lemma}
\begin{proof}
  Let us take $\zeta(x,t)=\eta(t)\zeta(x)$, $\eta\in C^\infty_c(0,\FinalT)$
  and $\zeta\in C^\infty_c(\R^d)$ with
  $\supp\zeta\subset \Ball R$; we have
  \begin{displaymath}
    -\int_0^{\FinalT}\eta'(t)\Big(\int_{\R^d}\zeta(x)\,\d \mu_t(x)\Big)\,\d t=
    \int_0^{\FinalT}\eta(t)
    \Big(\int_{\R^d}\nabla\zeta(x)\cdot \,\d \nnu_t(x)
    \Big)\,\d t,
  \end{displaymath}
  so that the map
  %\begin{displaymath}
  $  t\mapsto \mu_t(\zeta)=\int_{\R^d}\zeta\,\d \mu_t$
  % \end{displaymath}
  belongs to $W^{1,1}(0,\FinalT)$ with distributional derivative
  \begin{equation}
    \label{eq:derivative_mu}
    \dot\mu_t(\zeta)=\int_{\R^d}
    \nabla\zeta(x)\cdot \,\d\nnu_t(x)
    \quad\text{for $\Leb{1}$-a.e.\ }t\in (0,\FinalT),   
  \end{equation}
  satisfying
  \begin{equation}
     \label{eq:uniform_bound}
    |\dot\mu_t(\zeta)|\le
    V_R(t)\sup_{\R^d}|\nabla\zeta|
    ,\quad
    V_R(t):=|\nnu_t|(\Ball R),\quad
    \int_0^T V_R(t)\,\d t=V_R<+\infty.
  \end{equation}
  If $L_\zeta$ is the set of
  its Lebesgue points, we know that
  $\Leb 1((0,\FinalT)\setminus L_\zeta)=0$.
  Let us now take an increasing sequence $R_n:=2^n\up+\infty$
  and countable sets $Z_n\subset C^\infty_{\rm c}(\Ball{R_n})$
  which are dense in $C^1_0(\Ball{R_n}):=
  \{\zeta\in C^1(\Rd):\supp(\zeta)\subset \overline{\Ball{R_n}}\}$,
  the closure of $C^1_{\rm c}(\Ball{R_n})$ with respect the usual
  $C^1$ norm $\|\zeta\|_{C^1}=\sup_{\R^d}(|\zeta|,|\nabla\zeta|)$.
  We also set $L_Z:=\cap_{n\in\N,\zeta\in Z_n}L_\zeta$.
  The restriction of the curve $\mu$ to $L_Z$ provides
  a uniformly continuous family of functionals
  on each space $C^1_0(\Ball{R_n})$, since \eqref{eq:uniform_bound} shows
  \begin{displaymath}
    |\mu_t(\zeta)-\mu_s(\zeta)|\le
    \|\zeta\|_{C^1}\int_s^t V_{R_n}(\lambda)\,\d \lambda\quad
    \forall\, s,t\in L_Z\quad
    \forall\, \zeta\in Z_n.
  \end{displaymath}
  Therefore, for every $n\in \N$
  it can be extended in a unique way
  to a continuous curve $\{\tilde\mu^n_t\}_{t\in [0,\FinalT]}$
  in $[C^1_0(\Ball{R_n})]'$
  which is uniformly bounded and satisfies
  the compatibility condition
  \begin{equation}
    \label{eq:cap4:2}
    \tilde\mu^m_t (\zeta)=\tilde\mu^n(\zeta)\quad
    \text{if $m\le n$ and $\zeta\in C^1_{\rm c}(\Ball{R_m})$}.
  \end{equation} 
 If $\zeta\in C^1_{\rm c}(\Rd)$ we can thus define
  \begin{equation}
    \label{eq:cap4:27}
    \tilde\mu_t(\zeta):=\tilde\mu^n_t(\zeta)\quad
    \text{for every $n\in \N$ such that $\supp(\zeta)\subset
      \Ball {R_n}$}.    
  \end{equation}
  If we show that $\{\mu_t(\Ball{R_n})\}_{t\in L_Z}$ is uniformly
  bounded
  for every $n\in\N$,
  the extension provides a continuous curve in $\RPM(\R^d)$.
  To this aim, let us consider
  nonnegative, smooth functions
  \begin{subequations}
    \begin{gather}
      \label{eq:cap4:21}
      \text{$\zeta_k:\R^d\to [0,1]$,
        such that}
      \quad
      \zeta_k(x):=\zeta_0(x/2^k),\\
      \label{eq:cap4:29}
      \zeta_k(x)=1\ \text{if }|x|\le 2^k,\quad
      \zeta_k(x)=0\ \text{if }|x|\ge 2^{k+1},\quad
      |\nabla\zeta_k(x)|\le A\, 2^{-k},
    \end{gather}
  \end{subequations}
  for some constant $A>1$.
  It is not restrictive to suppose that $\zeta_k\in Z_{k+1}$.
  Applying the previous formula \eqref{eq:derivative_mu},
  for $t,\,s\in L_Z$ we have
  \begin{equation}
    \label{eq:cap3:2}
    |\mu_t(\zeta_k)-\mu_s(\zeta_k)|\le
    a_k:=
    2^{1-k}\int_0^\FinalT |\nnu_r|\big(\Ball{2R_k}\setminus
    \Ball{R_k}\big)\,\d r\le
    A\,2^{-k}V_{2R_k}.
  \end{equation}
  It follows that
  \begin{equation}
    \label{eq:cap4:57}
    \mu_t(\Ball{R_k})\le \mu_t(\zeta_k)\le \mu_s(\zeta_k)+
    A\,2^{-k}V_{2R_{k}}\le \mu_s(\Ball{2R_{k}})+
    A\,2^{-k}\,V_{2R_{k}}\quad
    \forall\, t\in L_Z.
  \end{equation}
  Integrating with respect to $s$ we end up with the uniform bound
  \begin{displaymath}
    \mu_t(\Ball{R_k})\le A\,2^{-k} \,V_{R_{k+1}}+\int_0^T
    \mu_s(\Ball{2R_k})\,\d s<+\infty\quad
    \forall\, t\in L_Z.
  \end{displaymath}
  Observe that the extension $\tilde\mu_t$ satisfies
  \eqref{eq:cap4:57} (and therefore, in a completely analogous way,
  \eqref{eq:cap4:33}) and \eqref{eq:cap3:2}
  for every $s,t\in [0,T]$.
  
  Now we show \eqref{eq:traces}.
  Let us choose $\zeta\in C^1_c(\R^d\times[0,\FinalT])$
  and
  $\eta_\eps\in C^\infty_c(t_1,t_2)$ such that
  \begin{displaymath}
    0\le \eta_\eps(t)\le 1,\quad
    \lim_{\eps\down0}\eta_\eps(t)=\chi_{(t_1,t_2)}(t)\quad
    \forall\, t\in [0,T],\quad
    \lim_{\eps\down0}\eta_\eps'=\delta_{t_1}-\delta_{t_2}
  \end{displaymath}
  in the duality with continuous functions in $[0,\FinalT]$.
  We get
  \begin{align*}
    0&=\int_0^\FinalT\int_{\R^d}\partial_t(\eta_\eps\zeta)
    \,\d \mu_t(x)\,\d t+\int_0^\FinalT\int_{\R^d}
    \nabla_x(\eta_\eps\zeta)\cdot \,\d\nnu_t\,\d t\\
    &=\int_0^\FinalT\eta_\eps(t)\int_{\R^d}\partial_t\zeta
    \,\d\mu_t\,\d t
    +\int_0^\FinalT\eta_\eps(t)\int_{\R^d}
    \nabla_x\zeta\cdot\,\d\nnu_t\,\d t+
    \int_0^{\FinalT}\eta_\eps'(t)\int_{\R^d}
    \zeta\,\d\tilde\mu_{t}\,\d t.
  \end{align*}
  Passing to the limit as $\eps$ vanishes and invoking
  the continuity of $\tilde\mu_t$, we get \eqref{eq:traces}.

  Finally, if $\lim_{R\up+\infty}R^{-1}V_R=0$
  we can pass to the limit as $R_k\up+\infty$ in 
  the inequality \eqref{eq:cap3:2}, which also holds
  for every $t,s\in [0,T]$ if we
  replace $\mu$ by $\tilde\mu$, by choosing $s$ 
  so that
  \begin{displaymath}
    m:=\tilde\mu_s(\Rd)=\lim_{k\up+\infty}
    \tilde\mu_s(\zeta_k)<+\infty .
  \end{displaymath}
  It follows that
  $\tilde\mu_t(\Rd)=\lim_{k\up+\infty}
    \tilde\mu_s(\zeta_k)= m$ for every $t\in [0,T]$.
  % By \eqref{eq:cap3:1bis} there exists 
%   $s\in L_Z$ such that $\mu_s(\Rd)=\tilde\mu_s(\Rd)<+\infty$
%   and $\mu_s$ is tight; for every $\eps>0$
%   we can find $k\in\N$ such that
%   $\mu_s(\zeta_k)> \mu_s(\Rd)-\eps/2$ and $a_k<\eps/2$. We thus obtain
%   \begin{displaymath}
%     \mu_t(\overline{\Ball{R_{k+1}}})\ge \mu_t(\zeta_k)\ge
%     \mu_s(\Rd)-\eps\quad
%     \forall\, t\in L_Z.
%   \end{displaymath}
%   It follows that $\{\mu_t\}_{t\in L_Z}$ is tight; this property then
%   holds even for $\{\tilde \mu_t\}_{t\in [0,T]}$
%   and $\mu_t(\Rd)$ is constant.
\qed\end{proof}
Thanks to Lemma
\ref{le:continuous_representative}
we can introduce the following class of solutions
of the continuity equation.
\begin{definition}[Solutions of the continuity equation]
  \label{def:CE}
  We denote by $\ce 0T\Rd$ the set of time dependent measures
  $(\mu_t)_{t\in [0,T]}, (\nnu_t)_{t\in (0,T)}$ such that
  \begin{enumerate}
  \item $t\mapsto \mu_t$ is weakly$^*$ continuous in $\RPM(\Rd)$ (in
    particular,
    $\sup_{t\in [0,T]}\mu_t(\Ball R)<+\infty$ for every $R>0$),
  \item
    $(\nnu_t)_{t\in (0,T)}$
    is a Borel family with
    $\displaystyle 
      \int_0^T |\nnu_t|(\Ball R)\,\d t<+\infty\qquad
      \forall\, R>0;
    $
   \item
    $(\mu,\nnu)$ is a distributional solution of \eqref{eq:continuity1}.
  \end{enumerate}
  $\cce0T\Rd\sigma\eta$ denotes the subset of $(\mu,\nnu)\in \ce0T\Rd$
  such that $\mu_0=\sigma,\ \mu_1=\eta$.
\end{definition}
Solutions of the continuity equation can be rescaled in time:
\begin{lemma}[Time rescaling]
  \label{le:time_rescaling}
  Let ${\sft}:s\in [0,T']\to {\sft}(s)\in [0,\FinalT]$
  be a strictly increasing absolutely continuous map
  with absolutely continuous inverse ${\sfs}:={\sft}^{-1}$.
  Then $(\mu,\nnu)$ is a distributional solution of
  \eqref{eq:continuity1} if and only if
  \begin{displaymath}
    \hat\mu:=\mu\circ {\sft},\ 
    \hat \nnu:={\sft}'\big(\nnu\circ {\sft}\big),\ 
    \quad\text{is a distributional solution of
      \eqref{eq:continuity1} on }(0,T').
  \end{displaymath}
\end{lemma}
We refer to \cite[Lemma 8.1.3]{Ambrosio-Gigli-Savare05} for the proof.

The proof of the next lemma follows directly from
\eqref{eq:traces}.
\begin{lemma}[Glueing solutions]
  \label{le:glueing}
  Let $(\mu^i,\nnu^i)\in \ce0{T_i}\Rd$, $i=1,2$, with
  $\mu^1_{T_1}=\mu^2_0$.
  Then the new family $(\mu_t,\nnu_t)_{t\in (0,T_1+T_2)}$ defined as
  \begin{equation}
    \label{eq:cap3:33}
    \mu_t:=
    \begin{cases}
      \mu^1_t&\text{if }0\le t\le T_1\\
      \mu^2_{t-T_1}&\text{if }T_1\le t\le T_1+T_2
    \end{cases}
    \qquad
    \nnu_t:=
    \begin{cases}
      \nnu^1_t&\text{if }0\le t\le T_1\\
      \nnu^2_{t-T_1}&\text{if }T_1\le t\le T_1+T_2
    \end{cases}
  \end{equation}
  belongs to $\ce0{T_1+T_2}\Rd$.
\end{lemma}
\begin{lemma}[Compactness for solutions of the continuity equation (I)]
  \label{le:compactness_continuity}
  Let $(\mu^n,\nnu^n)$ be a sequence 
  in $\ce0T\Rd$ such that
  \begin{enumerate}
  \item
    for some $s\in [0,T]$\quad
    $      \sup_{n\in \N}\mu^n_s(\Ball R)<+\infty\quad
      \forall\, R>0;
   $
  \item
    the sequence of maps
    $t\mapsto |\nnu^n_t|(\Ball R)$ is equiintegrable
    in $(0,T)$, for every $R>0$.
     \end{enumerate}
  Then there exists a subsequence (still indexed by $n$)
  and a couple $(\mu_t,\nnu_t)\in \ce0T\Rd$
  such that (recall \eqref{eq:DS:4})
  \begin{equation}
    \label{eq:scheme3:34}
    \begin{aligned}
      \mu^n_t\weaksto \mu_t\quad&\text{weakly$^*$ in }\RPM(\Rd)
      \quad\forall\, t\in [0,T],\\
      \nnu^n\weaksto \nnu\quad&\text{weakly$^*$ in }
      \Mloc(\Rd\times (0,T);\Rd).
    \end{aligned}
  \end{equation}
  \eqref{eq:scheme3:34} yields in particular
  \begin{equation}
    \label{eq:cap4:32}
    \int_0^T \Phi(\mu_t,\nnu_t|\gamma)\,\d t\le
    \liminf_{n\up+\infty}
    \int_0^T \Phi(\mu^n_t,\nnu^n_t|\gamma^n)\,\d t
  \end{equation}
  for every sequence of Radon measures
  $\gamma^n\weaksto\gamma$ in $\RPM(\Rd)$, where $\Phi$ is an integral
  functional as in \eqref{eq:scheme:8}.
 %  Moreover, if 
%   $(\mu^n_s)_{n\in\N}$ is tight in $\Rd$,
%   % according to
% %     \eqref{eq:scheme3:33},
%     and \begin{equation}
%     \label{eq:scheme3:32}
%     %\lim_{R\up+\infty}
%     \sup_{n\in \N}
%     \int_0^T\int_{\Rd}    \frac 1{1+|x|}
%     \,\d|\nnu^n_t|(x)\,\d t<+\infty,
%   \end{equation}
%   then $\mu^n_t\to \mu_t$ narrowly in $\FPM(\Rd)$ and $\mu^n_t(\Rd)
%   \to \mu_t(\Rd)$ for every $t>0$.
\end{lemma}
\begin{proof}
  Since $\nnu^n:=\int_0^T\nnu^n_t\,\d t$ and
  $\mu^n_s$ have total variation uniformly
  bounded on each compact subset of $\Rd\times[0,T]$,
  we can extract a subsequence (still
  denoted by $\mu^n_s,\nnu^n$) such that
  $\mu^n_s\weaksto \mu_s$ in $\Mloc(\Rd)$ and
  $\nnu^n\weaksto \nnu$ in $\Mloc(\Rd\times[0,T];\Rd)$.
  The estimate \eqref{eq:cap4:33} shows that
  \begin{equation}
    \label{eq:cap3:7}
    \sup_{n\in \N}\mu^n_t(\Ball R)<+\infty\quad
    \forall\, t\in [0,T],\ R>0.
  \end{equation}
  The equiintegrability condition on $\nnu^n$ shows that
  $\nnu$ satisfies
  \begin{displaymath}
    |\nnu|(B_R\times I)=\int_I m_R(t)\,\d t\quad
    \quad\forall\, I\in \mathcal B(0,T),\ R>0,\quad
    \text{for some }m_R\in L^1(0,T),
  \end{displaymath}
  so that by
  the disintegration theorem 
  we can represent it as $\nnu=\int_0^T \nnu_t$ for a Borel family
  $\{\nnu_t\}_{t\in (0,T)}$ still satisfying \eqref{eq:integrability_of_vt}.
  Let us now consider a function $\zeta\in C^1_c(\Rd)$
  and for a given interval $I=[t_0,t_1]\subset [0,T]$
  the time dependent function
  %\begin{displaymath}
  $  \zzeta(t,x):=\nchi_I(t)\nabla\zeta(x).$
  % \end{displaymath}
  Since the discontinuity set of $\zzeta$ is concentrated on
  $N=\Rd\times\{t_0,t_1\}$ and $|\nnu|(N)=0$, general convergence
  theorems (see e.g.\ \cite[Prop.\ 5.1.10]{Ambrosio-Gigli-Savare05} yields
  \begin{equation}
    \label{eq:scheme3:35}
    \begin{aligned}
      \lim_{n\to\infty}&\int_{I}\int_{\Rd}\nabla \zeta(x)\cdot
      \d\nnu^n_t(x)\,\d t= \lim_{n\to\infty}\int_{\Rd\times (0,T)}
      \zzeta\cdot \d\nnu^n(t,x)
      \\&= \int_{\Rd\times(0,T)
        }\zzeta\cdot \d\nnu(t,x)=
      \int_{I}\int_{\Rd}\nabla \zeta(x)\cdot
      \d\nnu_t(x)\,\d t.
    \end{aligned}
  \end{equation}
  Applying \eqref{eq:traces} with $\zeta(t,x):=\zeta(x)$ and
  $t_0:=s$
  and the estimate \eqref{eq:cap3:7}
  we thus obtain the weak convergence of $\mu^n_t$ to
  a measure $\mu_t\in \FPM(\Rd)$ for every $t\in [0,T]$.
  It is immediate to check that the couple
  $(\mu_t,\nnu_t)$ belongs to $\ce0T\Rd$.
    \eqref{eq:cap4:32} follows now by the representation 
  \begin{displaymath}
    \int_0^T\Phi(\mu_t,\nnu_t|\gamma)\,\d t
    =
    \Phi(\mu,\nnu|\bar\gamma),\quad
    \mu:=\int_0^T \mu_t\, \d t,\
    \bar\gamma=\gamma\otimes\Leb 1\in \RPM(\Rd\times(0,T) )
  \end{displaymath}
  and the lower semicontinuity property
  stated in Theorem \ref{thm:lsc_functional_measures}.
    % Finally, let us assume that \eqref{eq:scheme3:32} holds
%   and let us denote its supremum by
%   $S$.
%   Let $\zeta_k$ be defined as
%   in \eqref{eq:cap4:29};
%   by definition of weak convergence we easily have
%   \begin{displaymath}
%     \int_0^T\int_\Rd \zeta_k(x)\frac1{1+|x|}\,\d \nnu_t(x)\,\d t=
%     \lim_{n\up+\infty}
%     \int_0^T\int_\Rd \zeta_k(x)\frac1{1+|x|}\,\d \nnu^n_t(x)\,\d t\le S
%   \end{displaymath}
%   so that by monotone convergence
%   \begin{displaymath}
%      \int_0^T\int_\Rd \zeta_k(x)\frac1{1+|x|}\,\d \nnu_t(x)\,\d t\le S<+\infty.
%   \end{displaymath}
%   If $\mu^n_s$ is a tight sequence and \eqref{eq:scheme3:32}
%   holds,
%   then $\mu^n_s(\Rd)\to\mu_s(\Rd)<+\infty$ and the total masses
%   of $\mu^n_t$, $\mu_t$ are independent of time;
%   we conclude
%   that $\mu^n_t(\Rd)\to\mu_t(\Rd)$ for every $t\in [0,T]$.  
\qed\end{proof}

\subsection{\bfseries Solutions of the continuity equation with finite
  $\Phi$-energy}
\label{subsec:FE}
For all this section we
will assume that $\phi:(0,+\infty)\times\Rd\to
(0,+\infty)$ is an admissible action density function
as in (\ref{subeq:phi_prop}a,b,c) for some $p\in (1,+\infty)$,
$\gamma\in \RPM(\Rd)$ is a given reference Radon measure, and
$\Phi$ is the corresponding integral functional as in
\eqref{eq:scheme:8}.
We want to study the properties of
measure valued solutions $(\mu,\nnu)$ of the continuity equation
\eqref{eq:continuity1} with finite $\Phi$-energy
\begin{equation}
  \label{eq:FE}
  \Energy:=\int_0^T \Phi(\mu_t,\nnu_t|\gamma)\,\d t<+\infty.
\end{equation}
We denote by $\CE\phi\gamma0T\Rd$ the subset of $\ce0T\Rd$ whose
elements $(\mu,\nnu)$ satisfies \eqref{eq:FE}.
\begin{remark}
  \label{rem:obvious}
  \upshape
  If $(\mu_t)_{t\in [0,T]}$ is weakly$^*$ continuous in $\RPM(\Rd)$
  and \eqref{eq:FE} holds,   
  then $\mu_t,\nnu_t$ also satisfy \eqref{eq:integrability_of_vt}:
  in fact, the weak$^*$ continuity of $\mu_t$ yields
  for every $R>0$
  $\sup_{t\in [0,T]}\mu_t(\Ball R)=M_R<+\infty$,
  and the estimate \eqref{eq:scheme3:5} yields
  (recall \eqref{eq:phinorm})
  \begin{equation}
    \label{eq:cap4:24}
    V_R\le \eta\,\int_0^T \phinorm{\nnu_t}(\Ball R)\,\d t\le
    \eta\,T^{1/q}\,\Energy^{1/p}\, H(\gamma(\Ball
    R),M_R)^{1/q}<+\infty.
  \end{equation}
  \end{remark}
Recalling that the function $\sfh$ is defined by
\eqref{eq:cap2:29}, we also introduce
the concave function
\begin{equation}
  \label{eq:cap3:14}
  \omega(s):=\int_0^s \frac{1}{\sfh(r)^{1/q}}\,\d r,\quad
  \omega(0)=0,\quad
  \omega'(s)=\frac{1}{\sfh(s)^{1/q}},\quad
  \lim_{s\to\infty}\omega(s)=+\infty.
\end{equation}
In the homogeneous case $\phi(\rho,\zz)=\rho^\alpha\dphinorm\zz^q$ we
have
\begin{equation}
  \label{eq:cap3:15}
  \omega(s)=\int_0^s r^{-\alpha/q}\,\d r=\frac
  q{q-\alpha}s^{1-\alpha/q}=
  \frac p\theta s^{\theta/p}.
\end{equation}
% In general, the obious bounds
% \begin{equation}
%   \label{eq:cap3:21}
%   h(r_0)\left(\frac r{r_0}\land 1\right)\le h(r)\le
%   h(r_0)\left(\frac r{r_0}\lor 1\right)
% \end{equation}
% yields
% \begin{align}
%   \frac r{h(r_0)^{1/q}}&\le h(r)\le p \frac
%   {r_0^{1/q}r^{1/p}}{h(r_0)^{1/q}}
%   &\text{if}&\quad 0<r\le r_0;\\
%    \frac
%   {r_0^{1/q}r^{1/p}}{p\,h(r_0)^{1/q}}
%   &\le h(r)\le \frac r{h(r_0)^{1/q}}
%     &\text{if}&\quad r\ge r_0.
% \end{align}
For given nonnegative $\zeta\in
C^1_{\rm c}(\Rd)$ and $\mu\in \RPM(\Rd)$
we will use the short notation
\begin{equation}
  \label{eq:cap3:12}
  \begin{gathered}
    \SZ:=\supp(\D\zeta)\subset\Rd,\quad
    G_p(\zeta):=\int_{\SZ}\zeta^p\,\d\gamma, \quad
    D(\zeta):=\sup_{\Rd}\dphinorm{\D\zeta}.
  \end{gathered}
  % \quad
  %\mu(\zeta|\gamma,D):=\mu(\zeta)/\gamma(\nchi_D \zeta)\quad
  %\text{if $\gamma(\nchi_D\zeta)>0$}.
\end{equation}
%%
% \begin{equation}
%   \label{eq:cap3:12}
%   \mu(\zeta):=\int_\Rd \zeta(x)\,\d\mu(x),
%   \quad
%   D_q(\zeta):=\sup_{x\in \Rd}|\D\zeta(x)|^q,\quad
%   G_{p,q}(\zeta):=\int_{\Rd}|\zeta(x)|^p|\D\zeta|^q\,\d\gamma(x).
%   % \quad
%   %\mu(\zeta|\gamma,D):=\mu(\zeta)/\gamma(\nchi_D \zeta)\quad
%   %\text{if $\gamma(\nchi_D\zeta)>0$}.
% \end{equation}
%%
\begin{theorem}
  \label{thm:crucial_estimate}
  Let $\zeta\in C^1_{\rm c}(\Rd)$ be a nonnegative function
  with $Z,G(\zeta),D(\zeta)$ defined as in \eqref{eq:cap3:12}, and let $\mu,\nnu\in
  \CE\phi\gamma0T\Rd$.
  Setting
  \begin{equation}
    \label{eq:cap3:13}
    \Energy_Z:=\int_0^T\Phi(\mu_t,\nnu_t|\gamma,\SZ)\,\d t\le
    \Energy<+\infty,
  \end{equation}
  we have
  \begin{equation}
    \label{eq:cap3:20}
    \left|\tfrac \d{\d t}\mu_{t}(\zeta^p)\right|\le
    p\, D(\zeta)\,\Phi(\mu_t,\nnu_t|\gamma,\SZ)^{1/p}\,
    H\big(G_p(\zeta),\mu_t(\zeta^p)\big)^{1/q}.
  \end{equation}
  In particular,
  there exists a constant $\sfC_1>0$ only depending (in a monotone way) on
  $\sfh, p, T$ such that
   \begin{equation}
    \label{eq:cap4:2bis}
    \sup_{t\in [0,T]}\mu_t(\zeta^p)\le
    \sfC_1\Big(\mu_0(\zeta^p)+D(\zeta)G_p(\zeta)^{1/q}
    \Energy^{1/p}_\SZ+
    D^p(\zeta)\Energy_\SZ\Big).
  \end{equation}
 Moreover, if $G_p(\zeta)>0$,
  \begin{equation}
    \label{eq:cap3:16}
    \left|\tfrac \d{\d t}\omega\left(\mu_t(\zeta^p)/G_p(\zeta)\right)\right|\le
    \frac
    {p\,D(\zeta)}{G_p(\zeta)^{1/p}}\Phi(\mu_t,\nnu_t|\gamma,\SZ)^{1/p}\quad
    \text{for a.e.\ $t\in (0,T)$}.
  \end{equation}
  In particular, in the $(\alpha\text{-}\theta)$-homogeneous case,
  for every $0\le s\le t\le T$ we have
  \begin{equation}
    \label{eq:cap4:15}
    \Big|\|\zeta\|^\theta_{L^p(\mu_t)}-
    \|\zeta\|^\theta_{L^p(\mu_s)}\Big|\le
    \theta\, D(\zeta)\, \|\zeta\|_{L^p(\gamma,\SZ)}^{\theta-1}\,
    \int_s^t\Phi(\mu_r,\nnu_r|\gamma,\SZ)^{1/p}\,\d r.
  \end{equation}
  \end{theorem}
\begin{proof}
  Setting $m_t:=\mu_t(\zeta^p),\ G=G_p(\zeta),\ D=D(\zeta)$
  we easily have by \eqref{eq:Cap3:1}
  \begin{displaymath}
    \frac \d{\d t}m_t=\frac \d{\d t}
    \int_\Rd \zeta^p\,\d \mu_t=
    p\, \int_\SZ \zeta^{p-1}\nabla\zeta\cdot
    \d\nnu_t\le p\,D\,
    \Phi(\mu_t,\nnu_t|\gamma,\SZ)^{1/p}\,
    H\big(G,m_t\big)^{1/q},
  \end{displaymath}
  since $(\zeta^{p-1})^q=\zeta^p$.
  Since $H\big(G,m_t)=G\sfh (m_t/G),$
  we get
  \begin{displaymath}
    \sfh^{-1/q}(m_t/G)\frac \d{\d t}m_t\le
    p\,D\,G^{1/q}\Phi(\mu_t,\nnu_t|\gamma,\SZ)^{1/p}.
  \end{displaymath}
  Recalling that $\frac \d{\d r}\omega(r)=\sfh^{-1/q}(r)$
  we get
  \eqref{eq:cap3:16}.

  In order to prove \eqref{eq:cap4:2bis} we set
  $M:=\sup_{t\in [0,T]}m_t$ and we
  choose constants $(a,b)\in \Gamma_\phi$;
  integrating \eqref{eq:cap3:20} we get
  \begin{equation}
    \label{eq:cap4:7}
    \sup_{t\in [0,T]}\big|m_t-m_0\big|\le
    p\,D\, T^{1/q}\,\Big( \big(a G\big)^{1/q}\Energy^{1/p}_\SZ+
    \big(b  M\big)^{1/q}\Energy^{1/p}_\SZ\Big).
  \end{equation}
  By using the inequality $xy\le p^{-1}x^p+q^{-1}y^q$ we obtain
  \begin{equation}
    \label{eq:cap4:13}
    M\le m_0+ p\, D\,\big(a \,T \,G\big)^{1/q}\Energy^{1/p}_\SZ+
    \frac 1q M+p^{p-1}\,D^p\,\big(b\, T \big)^{p/q}\Energy_\SZ
  \end{equation}
  which yields \eqref{eq:cap4:2bis} with
  $\sfC_1:=p\max\big(1,p(aT)^{1/q},p^{p-1}(b T)^{p/q}\big)$.

  Finally, let us assume that $\phi$ satisfies the
  $(\alpha\text{-}\theta)$-homogeneity condition, so that $\omega(s)=\frac
  p\theta s^{\theta/p}$ as in \eqref{eq:cap3:15}.
  It follows that
  \begin{equation}
    \label{eq:cap4:16}
    \omega(G^{-1}\,m_t)=\frac p\theta
    \|\zeta\|_{L^p(\mu_t)}^\theta \|\zeta\|_{L^p(\gamma,\SZ)}^{-\theta}.
  \end{equation}
  Integrating \eqref{eq:cap3:16} we conclude.
\qed\end{proof}
We extend the definition of $\sfm_r(\mu)$ also for negative values
of $r$ by setting
\begin{equation}
  \label{eq:cap4:69}
  \sftm_r(\mu):=\mu(\Ball 1)+\int_{\Rd\setminus \Ball
    1}|x|^r\,\d\mu(x)=
  \int_\Rd \big(1\lor |x|\big)^r\,\d \mu(x)\quad
  \forall\, r\in \R.
\end{equation}
Notice that $\sftm_0(\mu)=\mu(\Rd)$ and $\sfm_r(\mu)\le \sftm_r(\mu)\le 
\mu(B_1)+\sfm_r(\mu)$ when $r>0$.
\begin{theorem}
\label{thm:moment_estimate}
  Let us assume that $\sftm_r(\gamma)<+\infty$ for some $r\le p$
  and let $(\mu,\nnu)\in \CE\phi\gamma0T\Rd$ satisfy
  \eqref{eq:FE}.
  For every $\delta\le  1+r/q$,
  if $\sftm_\delta(\mu_0)<+\infty$
  then also $\sftm_\delta(\mu_t)<+\infty$ and
  there exists a constant $\sfC_2$ only depending
  in a monotone way 
  on $\sfh, p, T, A,|\delta|$ such that
  \begin{equation}
    \label{eq:cap4:70}
    \sftm_\delta(\mu_t)\le \sfC_2\Big(
    \sftm_\delta(\mu_0)+\sftm_r(\gamma)^{1/q}\Energy^{1/p}+\Energy\Big).
  \end{equation}
  Moreover, if $r\ge -q$ and $\mu_0(\Rd)<+\infty$,
  then $\mu_t(\Rd)$ is finite and constant for every $t\in [0,T]$.
  \end{theorem}
\begin{proof}
  Let us first set
  \begin{equation}
    \label{eq:cap4:30}
    K_{n}:=2^{nr}\gamma(\Ball{2^{n+1}}\setminus\Ball {2^n})
    %\frac1{|x|^{\kappa}}\,\d \gamma(x),
  \end{equation}
  observing that
  \begin{equation}
    \label{eq:cap4:31}
    K_{n}\le \sum_{j=0}^{+\infty}K_j\le 2^{r^-}\,\sftm_r(\gamma) ,\quad
    \limsup_{n\up+\infty}K_{n}=0.
  \end{equation}
  We consider the usual cutoff functions $\zeta_n\in C^\infty_{\rm
    c}(\Rd)$
  as in (\ref{eq:cap4:21},b) and we set 
  \begin{equation}
  \label{eq:cap4:20bis}
    D_n=
    D(\zeta_n)=\sup\dphinorm{\D\zeta_n}\le A\,2^{-n},\quad
    G_n=G_p(\zeta_n)\le 
    \gamma(\Ball{2^{n+1}}\setminus\Ball {2^n})
    =2^{-nr}\,K_n.
  \end{equation}
  By \eqref{eq:cap4:2bis}
  we obtain
  \begin{equation}
    \label{eq:post}
    \sup_{t\in [0,T]}\mu_t(\Ball{2^n})\le
    \sfC_{1}\Big(\mu_0(\Ball{2^{n+1}})+
    A\,2^{-n(1+r/q)}K_n^{1/q}E^{1/q}+A^p\,2^{-np}E\Big);
  \end{equation}
  in particular, if $r\ge -q$ and $\mu_0(\Rd)<+\infty$,
  we can derive the uniform upper bound
  $\mu_t(\Rd)\le \sfC_{1}\, \mu_0(\Rd)$
  letting $n\up+\infty$.
  We can then deduce that $\mu_t(\Rd)$
  is constant by applying the estimate \eqref{eq:cap3:20}, which
  yields after an integration in time and for every $(a,b)\in \Gamma_\phi$
  \begin{displaymath}
    \sup_{t\in [0,T]}\big|\mu_t(\zeta_n^p)-\mu_0(\zeta_n^p)\big|
    \le p\,A\,T^{1/q}\,\Energy^{1/p}\, 2^{-n}\big(a 2^{-nr}\,K_n +b\sfC_1
    \mu_0(\Rd)\big)^{1/q}.
  \end{displaymath}
  In order to show \eqref{eq:cap4:70}, we argue
  as before, by introducing the new family of
  test functions induced by $\upsilon_n(x):=\upsilon_0(x/2^n)\in C^\infty_{\rm c}(\Rd)$
  \begin{equation}
    \label{eq:cap4:20tris}
    0\le \upsilon_n\le 1,\quad
    \begin{cases}
      \upsilon_n(x)\equiv 1&\text{if }2^n\le |x|\le 2^{n+1},\\
      \upsilon_n(x)\equiv 0&\text{if }|x|\le 2^{n-1}\ \text{or }|x|\ge
      2^{n+2},
    \end{cases}
    \quad
    \dphinorm{\D\upsilon_n}\le A\,2^{-n}.
  \end{equation}
  Observe that $    1\le
  \sum_{n=1}^{+\infty}\big(\upsilon_n(x)\big)^p\le 3 $
  and for some constant $A_\delta>1$
  \begin{equation}
    \label{eq:cap4:34}
    A_\delta^{-1}\,|x|^\delta\le
	\sum_{n=1}^{+\infty}
    2^{\delta n}\big(\upsilon_n(x)\big)^p\le
	A_{\delta}\, |x|^\delta\quad
    \forall\, x\in \Rd,\ |x|\ge 2.
  \end{equation}
% and		
%   \begin{equation}
%     \label{eq:cap4:26}
%     2^{-\delta}|x|^\delta\le
% 	\sum_{n=1}^{+\infty}
%     2^{\delta n}\big(\upsilon_n(x)\big)^p\le
% 	3 \, 2^\delta |x|^\delta.
%   \end{equation}
  As before, setting $K_n':=K_{n+1}+K_{n-1},$ we have
  $D(\upsilon_n)\le A\,2^{-n}$ and
  \begin{equation}
    G_p(\upsilon_n)\le \Big(2^{-(n+1)r}K_{n+1}+2^{-(n-1)r}K_{n-1}\Big)\le
   2^{|r|}\, 2^{-nr}\, K'_n.
  \end{equation} 
  Applying \eqref{eq:cap4:2bis} we get for every $t\in [0,T]$ 
  \begin{align*}
   2^{n\delta}\mu_t(\upsilon_n^p)\le 
   \sfC_1 \Big(2^{n\delta}\mu_0(\upsilon_n^p)+
   A\,2^{(\delta-1-r/q)n}(K'_n)^{1/q}\,(\Energy'_n)^{1/p}+
   A^p\,2^{(\delta-p)n}\Energy_n'\Big),
  \end{align*}
  where
  \begin{equation}
   \label{eq:cap4:100} 
   \Energy_n:=\int_0^T\Phi(\mu_t,\nnu_t|\gamma,B_{2^{n+1}}\setminus B_{2^{n}})\,\d t,\quad
   \Energy'_n:=\Energy_{n+1}+\Energy_{n-1}.
  \end{equation}
  Since $\delta\le 1+r/q$ and $\delta\le p$, summing up with respect
  to $n$
  and recalling \eqref{eq:post}
  we get
  \begin{equation}
    \label{eq:cap4:34bis}
    \sftm_\delta(\mu_t)\le 
    \sfC_2\Big(\sftm_\delta(\mu_0)+
    (\sftm_r(\gamma))^{1/q}\Energy^{1/p}+\Energy\Big).\quad\qed
  \end{equation} 
\end{proof}
In the the $\theta$-homogeneous case we have a more refined estimate:
\begin{theorem}
\label{thm:moment_estimate2}
Let us assume that
$\phi$ is $\theta$-homogeneous for some $\theta\in (1,p]$, 
the measure $\gamma$ satisfies
the $r$-moment condition $\sftm_r(\gamma)<+\infty$,
and let $(\mu,\nnu)\in \CE\phi\gamma0T\Rd$ satisfy \eqref{eq:FE}.
For every $\delta\le \bar\delta:= \frac 1\theta p+(1-\frac
1\theta)r,$
if $\sftm_\delta(\mu_0)<+\infty$ 
then $\sftm_\delta(\mu_t)$ is finite and there exists a constant
$\sfC_3>0$ such that
\begin{equation}
  \label{eq:cap4:102tris} 
  \sftm_{\delta}(\mu_t)
  \le \sfC_3\Big( \sftm_\delta(\mu_0)
  +\sftm_r(\gamma)^{1-1/\theta}\Energy^{1/\theta}\Big).
\end{equation}
Moreover, if $\bar\delta\ge0$ (i.e.\ $r\ge -p/(\theta-1)$)
and $\mu_0(\Rd)<+\infty$ 
then $\mu_t(\Rd)$ is finite and constant for $t\in [0,T]$.
  \end{theorem}
\begin{proof}
  We argue as in the proof of Theorem \eqref{thm:moment_estimate2},
  keeping the same notation and using the crucial estimate 
  \eqref{eq:cap4:15}.   
   If $\zeta_n$ are the
   test functions  of (\ref{eq:cap4:21},b), 
  \begin{equation}
    \label{eq:cap4:22}
    \|\zeta_n\|_{L^p(\gamma,Z_n)}^{\theta-1}=G_n^{(\theta-1)/p}
    \topref{eq:cap4:20bis}=
    2^{-nr(\theta-1)/p}\big(K_n\big)^{(\theta-1)/p},
  \end{equation}
  so that, since $\bar\delta\theta/p=1+(\theta-1)r/p$, \eqref{eq:cap4:15} yields  \begin{equation}
  \label{eq:cap4:101} 
   \Big|\big(\mu_t(\zeta^p_n)\big)^{\theta/p}-
   \big(\mu_0(\zeta^p_n)\big)^{\theta/p}\Big|\le 
   A\,\theta 2^{-\bar\delta \theta n/p}\,\big(K_n\big)^{(\theta-1)/p} \Energy^{1/p}.
  \end{equation} 
  Since $\bar\delta\ge0$, passing to the limit as $n\up\infty$ and recalling
  \eqref{eq:cap4:31},
  % that $\lim_{n\up+\infty}K_n=0$,
  we get  $\mu_t(\Rd)\equiv\mu_0(\Rd)$.
  Concerning the moment estimate, we replace $\zeta_n$ by
  $\upsilon_n$,
  defined by
  in \eqref{eq:cap4:20tris}, obtaining
  \begin{equation}
  \label{eq:cap4:101bis} 
   \Big|\big(\mu_t(\upsilon^p_n)\big)^{\theta/p}-
   \big(\mu_0(\upsilon^p_n)\big)^{\theta/p}\Big|\le 
   \sfC_{3.1} 2^{-\bar\delta \theta n/p}\,\big(K'_n\big)^{(\theta-1)/p} \big(\Energy'_n\big)^{1/p},
  \end{equation}
  and therefore
  \begin{equation}
  \label{eq:cap4:102} 
   \mu_t(\upsilon^p_n)\le \sfC_{3.2}\Big( \mu_0(\upsilon^p_n)
   +2^{-\bar\delta n} \big(K_n'\big)^{1-1/\theta}\big(\Energy_n'\big)^{1/\theta}\Big).
  \end{equation} 
  Multiplying this inequality by $2^{n\delta}$, summing up w.r.t.\ $n$,
  and recalling \eqref{eq:cap4:34}, we obtain
  \begin{equation}
  \label{eq:cap4:102bis} 
   \sftm_{\delta}(\mu_t)
   \le \sfC_3\Big( \sftm_\delta(\mu_0)
   +\sftm_r(\gamma)^{1-1/\theta}\Energy^{1/\theta}\Big).
   \quad\qed
  \end{equation} 
\end{proof}
\begin{corollary}[Compactness for
  solutions of the continuity equation (II)]
  \label{cor:compactness_continuity2}
  Let $(\mu^n,\nnu^n)$ be a sequence in $\CE\phi\gamma 0T\Rd$
  and let $\gamma^n\weaksto \gamma$ in $\RPM(\Rd)$ such that
  \begin{equation}
    \label{eq:cap3:6bis}
    \sup_{n\in \N}\mu^n_0(\Ball R)<+\infty\quad
    \forall\, R>0,\qquad
    \sup_{n\in \N}\int_0^T \Phi(\mu^n_t,\nnu^n_t|\gamma^n)\,\d t<+\infty.
  \end{equation}
  %for a sequence of nonnegative Radon measures
  %$\gamma^n\weaksto\gamma$.
  Then conditions 1. and 2. of Lemma \ref{le:compactness_continuity} are
  satisfied and
  therefore there exists a subsequence (still indexed by $n$)
  and a couple $(\mu_t,\nnu_t)\in \CE\phi\gamma0T\Rd$
  such that
  \begin{equation}
    \label{eq:scheme3:34bis}
    \begin{aligned}
      \mu^n_t\weaksto \mu_t\quad
      &\text{weakly$^*$ in }\RPM(\Rd)
      \quad\forall\, t\in [0,T],\\
      \nnu^n\weaksto \nnu\quad&\text{weakly$^*$ in }
      \Mloc(\Rd\times(0,T) ;\Rd),
    \end{aligned}
  \end{equation}
  \begin{equation}
    \label{eq:cap4:32bis}
    \int_0^T \Phi(\mu_t,\nnu_t|\gamma)\,\d t\le
    \liminf_{n\up+\infty}
    \int_0^T \Phi(\mu^n_t,\nnu^n_t|\gamma^n)\,\d t.
  \end{equation}
  Suppose moreover that
  $\mu^n_0(\Rd)\to \mu_0(\Rd)$
  and $\sup_n\sftm_\kappa(\gamma^n)<+\infty$ where $\kappa=-q$ or
  $\kappa =-p/(\theta-1)$ in the $\theta$-homogeneous case, then 
  (along the same subsequence)
  $\mu^n_t(\Rd)
  \to \mu_t(\Rd)$ for every $t\in [0,T]$.
 \end{corollary}
 \begin{proof}
   Since $P_R:=\sup_n\gamma^n(\Ball R)<+\infty$ for every $R>0$,
   the estimate \eqref{eq:cap4:70} for $\delta=0$ and the assumption
   \eqref{eq:cap3:6bis} show that
   $M_R=\sup_{n\in \N,t\in [0,T]}\mu^n_t(\Ball R)<+\infty$ for every $R>0$.
   We can therefore obtain a bound of $\phinorm{\nnu^n_t}(\Ball R)$
   by \eqref{eq:scheme3:5}, which yields 
   \begin{displaymath}
     \phinorm{\nnu^n_t}(\Ball R)\le
     H(P_R,M_R)^{1/q} \Phi(\mu_t,\nnu_t|\gamma)^{1/p},
   \end{displaymath}
   so that the maps $t\mapsto \phinorm{\nnu^n_t}(\Ball R)$ are uniformly
   bounded by
   a function in $L^p(0,T)$. % (and independent of $n$).
   The last assertion follows by the fact that
   $t\mapsto \mu^n_t(\Rd)$ is independent of time, thanks to
   Theorem \ref{thm:moment_estimate2} (in the
   $(\alpha\text{-}\theta)$-homogeneous case) or
   Theorem \ref{thm:moment_estimate} (for general density functions
   $\phi$).
 \qed\end{proof}

%%%%%%%%%%%%%%%%%%%%%%%%%%%%%%%%%%%%%%%%%%%%%%%%%%%%%%%%%%%%%%%%%%%%%%%%%%
%%%%%%%%%%%%%%%%%%%%%%%%%%%%%%%%%%%%%%%%%%%%%%%%%%%%%%%%%%%%%%%%%%%%%%%%%% 

\section{The $(\phi$-$\gamma)$-weighted Wasserstein distance}
\label{sec:modWass}

As we already mentioned in the Introduction,
\textsc{Benamou-Brenier} \cite{Benamou-Brenier00}
showed that
the Wasserstein distance $W_p$ \eqref{defwaspre}
can be equivalently characterized by a ``dynamic''
point of view through \eqref{eq:cap1:13}, involving the 
$1$-homogeneous action functional
\eqref{eq:cap1:10}.
% The integral functional
% functional $\Phi_{p,1}(\mu,\nnu|\gamma)$ \eqref{eq:scheme:8}
% depends only on $G=\supp(\gamma)$:
% it imposes the constraint
% $\supp(\mu)\subset G,\supp(\nnu)\subset G$ 
% but its value is then independent of
% the choice of the reference measure $\gamma$
% among all the measure with the same support $G$.
% If $G$ is convex the two definitions coincide.
%
The same approach can be applied to arbitrary
action functionals.
\begin{definition}[Weighted Wasserstein distances]
  \label{def:main}
  Let $\gamma\in \RPM(\Rd)$ be a fixed reference measure
  and $\phi:(0,+\infty)\times\Rd\to
  [0,+\infty)$ a function satisfying Conditions
  (\ref{subeq:phi_prop}a,b,c).
  The $(\phi,\gamma)$-Wasserstein (pseudo-) distance
  between $\mu_0,\mu_1\in \RPM(\Rd)$ is defined as
  \begin{equation}
    \label{eq:10}
    \begin{aligned}
      \cW^p_{\phi,\gamma}(\mu_0,\mu_1):=\inf\Big\{
      \int_0^1\Phi(\mu_t,\nnu_t|\gamma)\,\d t:
      \quad (\mu,\nnu)\in \cce01\Rd{\mu_0}{\mu_1}
      %,\\&
      % \mu\restr{t=0}=\mu_0,\ \mu\restr{t=1}=\mu_1
      \Big\}.
    \end{aligned}
  \end{equation}
  We denote by $\mathcal M_{\phi,\gamma}[\mu_0]$
  the set of all the measures $\mu\in \RPM(\Rd)$
  which are at finite $\cW_{\phi,\gamma}$-distance from $\mu_0$.  
\end{definition}
\begin{remark}
  \label{rem:known_cases}
  \upshape
  Let us recall the notation $W_{p,\alpha;\gamma}$ of \eqref{eq:cap1:25}
  in the case
  $\phi_{p,\alpha}(\rho,\ww)=\rho^\alpha|\ww/\rho^\alpha|^p$.
  When $\alpha=0$ we find the dual homogeneous Sobolev
  (pseudo-)distance
  \eqref{eq:cap1:8} and in the case
  $\alpha=1$ and $\supp(\gamma)=\Rd$ we get
  the usual Wasserstein distance:
  \begin{displaymath}
    \|\mu_0-\mu_1\|_{\dot
      W^{-1,p}_\gamma}=W_{p,0;\gamma}(\mu_0,\mu_1),\qquad
     W_p(\mu_0,\mu_1)=W_{p,1;\gamma}(\mu_0,\mu_1).
  \end{displaymath}
\end{remark}
\begin{remark}
  \label{rem:energy_rep}
  \upshape
  Taking into account Lemma \ref{le:time_rescaling},
  a linear time rescaling shows that
  \begin{equation}
    \label{eq:cap3:31}
     W^p_{\phi,\gamma}(\mu_0,\mu_T):=\inf\Big\{
    T^{p-1}\int_0^T \Phi(\mu_t,\nnu_t|\gamma)\,\d t:
    (\mu,\nnu)\in \cce0T\Rd{\mu_0}{\mu_T}
    %,\quad
    %\mu\restr{t=0}=\sigma,\ \mu\restr{t=T}=\eta
    \Big\}.
  \end{equation}
\end{remark}
\begin{theorem}[Existence of minimizers]
  \label{thm:0}
  Whenever the infimum in \eqref{eq:10} is a
  finite value $W<+\infty$, it is attained
  by a curve $(\mu,\nnu)\in \CE\phi\gamma01\Rd$
  such that
  \begin{equation}
    \label{eq:cap3:37bis}
    \Phi(\mu_t,\nnu_t|\gamma)= W
    \quad\text{for $\Leb 1$-a.e.\ $t\in (0,1)$}.
  \end{equation}
  The curve $(\mu_t)_{t\in [0,1]}$ associated to a minimum for
  \eqref{eq:10}
  is a constant speed mimimal geodesic for $\cW_{\phi,\gamma}$
  since it satisfies
  \begin{equation}
     \label{eq:cap5:7}
     \cW_{\phi,\gamma}(\mu_s,\mu_t)=|t-s|\,\cW_{\phi,\gamma}(\mu_0,\mu_1)
     \quad\forall\, s,t\in [0,1].
   \end{equation}
   We have also the equivalent characterization
  \begin{equation}
    \label{eq:cap3:29}
     \cW_{\phi,\gamma}(\sigma,\eta)=\inf\Big\{
    \int_0^T \Big(\Phi(\mu_t,\nnu_t|\gamma)\Big)^{1/p}\,\d t:
    (\mu,\nnu)\in \cce0T\Rd\sigma\eta
    \Big\}.
  \end{equation}
\end{theorem}
\begin{proof}
  When $\cW_{\phi,\gamma}(\mu_0,\mu_1)<+\infty$,
  Corollary \ref{cor:compactness_continuity2}
  immediately yields the existence
  of a minimizing curve $(\mu,\nnu)$.
  % In order to prove that
%   $\Phi_t:=\Phi(\mu_t,\nnu_t|\gamma)$ is independent of time
%   let us introduce a function $\psi\in C^1_{\rm c}(0,1)$
%   and let us consider the map
%   \begin{equation}
%     \label{eq:cap3:39}
%     \sfs_\eps(t):=t+\eps\psi(t),\quad \eps\in \R,
%   \end{equation}
%   which is an admissible diffeomorphism of $(0,1)$ if $\eps$ is
%   sufficiently small: we denote by $\sft_\eps$ its inverse,
%   which satisfies $\sft_\eps'(\sfs_\eps(t))=(1+\eps\psi'(t))^{-1}$,
%   and we set $\hat\mu^\eps:=\mu\circ \sft_\eps,\hat\nu^\eps=
%   \sft_\eps'\nu\circ\sft_\eps$.
%   By the minimality property, the rescaling Lemma
%   \ref{le:time_rescaling},
%   and the homogeneity
%   of $\Phi$ we get
%   \begin{displaymath}
%     \int_0^1 \Phi_t\,\d t\le
%     \int_0^1 \Phi(\hat\mu^\eps_s,\hat\nnu^\eps_s|\gamma)\,\d s=
%     \int_0^1 \Phi_{\sft_\eps(s)} |\sft_\eps'(s)|^{p-1} \sft_\eps'(s)\,\d s=
%     \int_0^1 \Phi_t (1+\eps\psi'(t))^{1-p}\,\d t.
%   \end{displaymath}
%   A differentiation with respect to $\eps$ at $\eps=0$ yields
%   %\begin{displaymath}
%   $  -p\int_0^1 \Phi_t\, \psi'(t)\,\d t=0\quad
%     \forall\,\psi\in C^1_{\rm c}(0,1),$
% %  \end{displaymath}
%   which shows that $\Phi_t$ is independent of $t$.
  Just for the proof of \eqref{eq:cap3:29},
  let us denote by $\bar \cW_{\phi,\gamma}(\sigma,\eta)$
  the infimum of the right-hand side of
  \eqref{eq:cap3:29}.
  H\"older inequality 
  immediately shows that $\cW_{\phi,\gamma}(\sigma,\eta)\ge \bar
  \cW_{\phi,\gamma}(\sigma,\eta)$.
  In order to prove the opposite inequality,
  we argue as in \cite[Lemma 1.1.4]{Ambrosio-Gigli-Savare05}, defining
  for $(\mu,\nnu)\in \cce0T\Rd\sigma\eta$
  \begin{equation}
    \label{eq:cap3:20bis}
    \sfs_{\eps}(t):=\int_0^t \Big(\eps+
    \Phi(\mu_r,\nnu_r|\gamma)\Big)^{1/p}\,\d r,\quad
    t\in [0,T];
  \end{equation}
  $\sfs_\eps$ is strictly increasing with $\sfs_\eps'\ge \eps$,
  $\sfs_\eps(0,T)=(0,S_\eps)$ with $S_\eps:=\sfs_\eps(T)$, so
  that its inverse map $\sft_\eps:[0,S_\eps]\to [0,T]$ is well
  defined
  and Lipschitz continuous, with
  \begin{equation}
    \label{eq:cap3:32}
    \sft_\eps'\circ \sfs_\eps=\Big(\eps+\Phi(\mu_t,\nnu_t)\Big)^{-1/p}
    \quad \text{a.e.\ in $(0,T)$}.
  \end{equation}
  If $\hat \mu^\eps=\mu\circ \sft_\eps,
  \hat\nnu^\eps:=\sft_\eps'\,\nnu\circ \sft_\eps$, we
  know that $(\hat\mu^\eps,\hat\nnu^\eps)\in \cce0{S_\eps}\Rd\sigma\eta$ so
  that
  \begin{align*}
    \cW_{\phi,\gamma}^p(\sigma,\eta)\le
    S_\eps^{p-1}\int_0^{S_\eps}\Phi(\hat\mu^\eps_s,\hat\nnu_s^\eps|\gamma)
    \,\d s=
    S_\eps^{p-1}\int_0^{T}
    \frac{\Phi(\mu^\eps_t,\nnu_t^\eps|\gamma)}
    {\eps+\Phi(\mu^\eps_t,\nnu^\eps_t|\gamma)}
    \Big(\eps+\Phi(\mu_t,\nnu_t)\Big)^{1/p}
    \,\d t,
  \end{align*}
  the latter integral being less than $S_\eps^p.$
  Passing to the limit as $\eps\downarrow0$, we get
  \begin{equation}
    \label{eq:newclass_times:4}
    \cW_{\phi,\gamma}(\sigma,\eta)\le \int_0^T
    \Phi(\mu_t,\nnu_t|\gamma)^{1/p}
    \,\d t\quad
    \forall\, (\mu,\nnu)\in \cce0T\Rd\sigma\eta,
  \end{equation}
  and therefore $\cW_{\phi,\gamma}(\sigma,\eta)\le
  \bar\cW_{\phi,\gamma}(\sigma,\eta)$.
  If $(\mu,\nnu)\in\CCE\phi\gamma01\Rd{\mu_0}{\mu_1}$ is a minimizer of
  \eqref{eq:10}, then \eqref{eq:newclass_times:4} yields
  \begin{displaymath}
    W^{1/p}=\cW_{\phi,\gamma}(\mu_0,\mu_1)= \Big(\int_0^1
    \Phi(\mu_t,\nnu_t|\gamma)\,\d t\Big)^{1/p}=
    \int_0^1
    \Phi(\mu_t,\nnu_t|\gamma)^{1/p}\,\d t,
  \end{displaymath}
  so that \eqref{eq:cap3:37bis} holds.
\qed\end{proof}
\subsection{\bfseries Topological properties}
\label{subsec:top}
\begin{theorem}[Distance and weak convergence]
  \label{thm:1}
  %Under the assmptions of definition \ref{def:main},
  The functional
  $\cW_{\phi,\gamma}$ is a (pseudo)-distance on $\RPM(\Rd)$ which
  induces
  a stronger topology than the weak$^*$ one. Bounded sets 
  with respect to $\cW_{\phi,\gamma}$ are weakly$^*$ relatively compact.
\end{theorem}
\begin{proof}
  It is immediate to check that $\cW_{\phi,\gamma}$ is symmetric (since
  $\phi(\rho,-\ww)=\phi(\rho,\ww)$)
  and $\cW_{\phi,\gamma}(\sigma,\eta)=0\ \Rightarrow\ \sigma\equiv
  \eta$.
  The triangular inequality follows as well from
  the characterization \eqref{eq:cap3:29} and
  the gluing Lemma \ref{le:glueing}.

  From \eqref{eq:cap3:16} (keeping the same notation
  \eqref{eq:cap3:12})
  and \eqref{eq:cap3:29} we immediately get
  for every $\mu_0,\mu_1\in \RPM(\Rd)$ and nonnegative
  $\zeta\in C^1_{\rm c}(\Rd)$
  with $\|\zeta\|_{L^p(\gamma)}>0$
  \begin{displaymath}
    \Big|\omega\big(\mu_1(\zeta^p)/G_p(\zeta)\big)-\omega
    \big(\mu_0(\zeta^p)/G_p(\zeta)\big)\Big|\le
    \frac{p\, D(\zeta)}{\|\zeta\|_{L^p(\gamma)}}
    \cW_{\phi,\gamma}(\sigma,\eta),
  \end{displaymath}
  which shows the last assertion, since $\omega$ is strictly
  increasing and
  the set
  \begin{displaymath}
    \big\{\zeta^p:\zeta\in C^1_{\rm c}(\Rd),\quad \zeta\ge0,\quad
    \|\zeta\|_{L^p(\gamma)}>0\big\}
  \end{displaymath}
  is dense in the space of nonnegative continuous functions
  with compact support (endowed with the uniform topology).
\qed\end{proof}
\begin{theorem}[Lower semicontinuity]
  \label{thm:2}
  The map $(\mu_0,\mu_1)\mapsto \cW_{\phi,\gamma}(\mu_0,\mu_1)$
  is lower semicontinuous with respect to weak$^*$ convergence
  in $\RPM(\Rd)$. More generally,
  suppose that $\gamma^n\weaksto \gamma$ in $\RPM(\Rd)$,
  $\phi^n$ is monotonically increasing w.r.t.\ $n$
  and pointwise converging to $\phi$,
  and $\mu^n_0\weaksto\mu_0,\mu^n_1\weaksto\mu_1$ in $\RPM(\Rd)$ as
  $n\up+\infty$.
  Then
  \begin{equation}
    \label{eq:cap5:9}
    \liminf_{n\up+\infty}\cW_{\phi^n,\gamma^n}(\mu^n_0,\mu^n_1)\ge
    \cW_{\phi,\gamma}(\mu_0,\mu_1).
  \end{equation}
\end{theorem}
\begin{proof}
  It is not restrictive to assume that
  $\cW_{\phi^n,\gamma^n}(\mu^n_0,\mu^n_1)< S<+\infty$,
  so that
  we can find a sequence
  $(\mu^n,\nnu^n)\in \CCE{\phi^n}{\gamma^n}01\Rd{\mu^n_0}{\mu^n_1}$ such that
  \begin{equation}
    \label{eq:cap3:36}
    \Phi^m(\mu^n_t,\nnu^n_t|\gamma^n)\le S\quad
    \text{a.e.\ in }(0,1),\quad
    \forall\, m\le n\in\N,
  \end{equation}
  where $\Phi^m$ denotes the integral functional
  associated to $\phi^m$.
  We can apply Theorem \ref{cor:compactness_continuity2}
  and we can extract a suitable
  subsequence (still denoted bu $\mu^n,\nnu^n$)
  and a limit curve $(\mu,\nnu)\in \CCE\phi\gamma01\Rd{\mu_0}{\mu_1}$ such that
  \eqref{eq:scheme3:34bis} holds.
  We eventually have
  \begin{equation}
    \label{eq:cap3:38}
    \cW_{\phi^m,\gamma}^p(\mu_0,\mu_1)\le
    \int_0^1 \Phi^m(\mu_t,\nnu_t|\gamma)\,\d t\le S.
  \end{equation}
  Passing to the limit w.r.t.\ $m\up+\infty$ we conclude.
\qed\end{proof}
\begin{theorem}[Completeness]
 \label{thm:completeness}
For every $\sigma\in \RPM(\Rd)$ 
the space $\mathcal M_{\phi,\gamma}[\sigma]$
endowed with the distance $\cW_{\phi,\gamma}$ is complete.
\end{theorem}
\begin{proof}
 Let $(\mu_n)_{n\in \N}$ be a Cauchy sequence in $\mathcal M_{\phi,\gamma}[\sigma]$
 w.r.t.\ the distance $\cW_{\phi,\gamma}$; in particular, $(\mu_n)$
 is bounded so that we can extract a suitable convergence subsequence 
 $\mu_{n_k}$ weakly$^*$ converging to $\mu_\infty$ in 
 $\RPM(\Rd)$. Thanks to Theorem \ref{thm:2} we easily get
$
  \cW_{\phi,\gamma}(\mu_m,\mu_\infty)\le
  \liminf_{k\to\infty}\cW_{\phi,\gamma}(\mu_m,\mu_{n_k}),$
  and therefore, taking into account the Cauchy condition,\\
  $\limsup_{m\to\infty}\cW_{\phi,\gamma}(\mu_m,\mu_\infty)\le
  \limsup_{n,m\to\infty}\cW_{\phi,\gamma}(\mu_m,\mu_n)=0$ so that
$\mu_n$ converges to $\mu_\infty$.
\qed\end{proof}
Let us now consider the case of measures with
finite mass (just to fix
the constant, probability measures in $\Probabilities\Rd$).
We introduce the parameter
\begin{equation}
  \label{eq:cap4:71}
  \kappa:=
  \begin{cases}
    \displaystyle\frac p{\theta-1}=\frac{q}{1-\alpha}&\text{if $\phi$ is
      $(\alpha$-$\theta)$-homogeneous},\\
    \frac p{p-1}=q&\text{otherwise.}
  \end{cases}
\end{equation}
\begin{theorem}[Distance and total mass]
\label{thm:narrow_convergence}
 Let us assume that 
% \begin{equation}
%   \label{eq:cap5:10}
$   \tilde\sfm_{-\kappa}(\gamma)<+\infty$
%\int_{\Rd}\Big(1\lor %|x|\Big)^{-\kappa}\,\d\gamma(x)<+\infty$
% \quad
%   \text{where }
%   \begin{cases}
%     r:=-\frac{q}{1-\alpha}=-\frac{p}{\theta-1}
%     &\text{if $\phi$ is $\theta$-homogeneous},\\
%    r:=-q&\text{otherwise},
%   \end{cases}
% \end{equation}
 and let us suppose that $\sigma\in \Probabilities\Rd$.
 Then $\mathcal M_{\phi,\gamma}[\sigma]\subset \Probabilities\Rd$, 
 the weighted Wasserstein distance $\cW_{\phi,\gamma}$ is stronger than
 the
 narrow convergence
 in $\Probabilities\Rd$, and $\Probabilities\Rd$ endowed with
 the (pseudo-) distance $\cW_{\phi,\gamma}$ is a complete (pseudo-)metric space.
\end{theorem}
\begin{proof}
 If $\eta\in \mathcal M_{\phi,\gamma}[\sigma]$ then Theorem \ref{thm:moment_estimate2}
 (in the $\theta$-homogeneous case) or \ref{thm:moment_estimate}
 (in the general case) yields $\eta(\Rd)=\sigma(\Rd)=1$, so that 
 $\mathcal M_{\phi,\gamma}[\sigma]\subset \Probabilities\Rd$.
 Since the narrow topology coincide with the weak$^*$ one in
 $\Probabilities\Rd$, Theorem \ref{thm:1} proves the second statement.
 The completeness of $\Probabilities\Rd$ with respect to 
 the (pseudo) distance $\cW_{\phi,\gamma}$ follows by Theorem \ref{thm:completeness}. 
\qed\end{proof}
We can also prove 
some useful moment estimates.
\begin{theorem}[Moment estimates]
\label{thm:moment_convergence}
 Let us assume that 
% \begin{equation}
$   \tilde\sfm_r(\gamma)<+\infty$
%   \int_{\Rd}\Big(1\lor |x|\Big)^r\,\d\gamma(x)<+\infty\quad
%  \text{where }
%  \begin{cases}
%    r>-\kappa&
%    \text{if $\phi$ is $\theta$-homogeneous},\\
%   r\in(-q,p]&\text{otherwise},
%  \end{cases}
% \end{equation}
for some $r\in \R$
and let us set
 \begin{equation}
   \bar\delta:=
  \begin{cases}
    \frac 1\theta p+(1-\frac 1\theta)r=
    \frac p\theta(1+r/\kappa)&\text{if $\phi$ is $\theta$-homogeneous},\\
    1+r/q\le p&\text{otherwise.}
  \end{cases}
 \end{equation}
 If 
 $\sftm_\delta(\sigma)<+\infty$ for some $\delta\le \bar\delta$, 
 and  $\eta\in \mathcal M_{\phi,\gamma}[\sigma]$,
 then $\sftm_\delta(\eta)$ is finite and there exists a constant $\sfC$
 only depending on $\phi,\delta$  such that
 \begin{equation}
   \label{eq:cap4:35}
   \left\{
   \begin{aligned}
     \sftm_\delta(\eta)&\le
     \sfC\Big(\sftm_\delta(\sigma)+
     \sftm_r(\gamma)+\cW_{\phi,\gamma}^p(\sigma,\eta)\Big)\\
     \sftm_\delta(\eta)&\le
     \sfC\Big(\sftm_\delta(\sigma)+
     \sftm_r(\gamma)^{1-1/\theta}
     W_{p,\alpha}^{p/\theta}(\sigma,\eta)\Big).
   \end{aligned}
\right.
\end{equation}
Moreover, when $\delta\ge1$,
the topology induced by $\cW_{\phi,\gamma}$ in
$\mathcal M_{\phi,\gamma}[\sigma]$ is stronger
than the one induced by the Wasserstein distance $W_\delta$.
\end{theorem}
\begin{proof}
  Let us first consider the general case:
  applying \eqref{eq:cap4:70} 
  we easily obtain \eqref{eq:cap4:35}.
  In order to prove the assertion about the convergence of the
  moments induced by $\cW_{\phi,\gamma}$ (which is equivalent
  to the convergence in $W_\delta$ when $\delta\ge1$),
  a simple modification of \eqref{eq:cap4:34bis} yields
  \begin{equation}
    \label{eq:cap4:36}
    \int_{|x|\ge 2^n}|x|^\delta\,\d \eta\le
    \sfC_3\Big(
    \int_{|x|\ge 2^{n-1}}|x|^\delta\,\d \sigma+
    \sftm_r(\gamma)^{1/q}\cW_{\phi,\gamma}(\sigma,\eta)+
    W^p_{\phi,\gamma}(\sigma,\eta)\Big),
  \end{equation}
  which shows that every sequence $\eta_n$ converging to $\sigma$ has
  $\delta$-moments equi-integrable and therefore it is
  relatively compact with respect to the $\delta$-Wasserstein
  distance when $\delta\ge1$.

  The $\theta$-homogeneous case
  follows by the same argument and Theorem
  \ref{thm:moment_estimate2}.
\qed\end{proof}
\begin{remark}
  \label{rem:discussion}
  \upshape
  There are interesting particular cases
  covered by the previous result:
  \begin{enumerate}
  \item When $\gamma(\Rd)<+\infty$ then
    $\cW_{\phi,\gamma}$ 
    is always stronger than the
    $1$-Wasserstein distance $W_1$; in
    the $\theta$-homogeneous case,
    $W_{p,\alpha;\gamma}$ also controls the $W_{p/\theta}$
    distance.
  \item When $\sfm_p(\gamma)<+\infty$, then
    $\cW_{\phi,\gamma}$ is always stronger than $W_p$.
  \item When $\phi$ is $\theta$-homogeneous with $\theta>1$ and
    $\gamma$ is a probability measure with finite moments of arbitrary orders
    (this is the case of a log-concave probability measure),
    then all the measures $\sigma\in \mathcal M_{\phi,\gamma}[\gamma]$
    have finite moments of arbitrary orders
    and the convergence with respect to $\cW_{\phi,\gamma}$
    yields the convergence in $\mathcal P_\delta(\Rd)$ for every $\delta>0$.
  \end{enumerate}
\end{remark}

\subsection{\bfseries Geometric properties}
\label{subsec:geo}
\begin{theorem}[Convexity of the distance and uniqueness of geodesics]
  \label{thm:3}
  $\cW^p_{\phi,\gamma}(\cdot,\cdot)$ is 
  convex, i.e.
  for every $\mu_i^j\in \RPM(\Rd)$, $i,j=0,1$,
  and $\tau\in [0,1]$, if
  $\mu^\tau_i=(1-\tau)\mu^0_i+\tau\mu^1_i$, 
  \begin{equation}
    \label{eq:40}
    \cW^p_{\phi,\gamma}(\mu^\tau_0,\mu^\tau_1)\le
    (1-\tau)\cW^p_{\phi,\gamma}(\mu^0_0,\mu^0_1)+\tau
    \cW^p_{\phi,\gamma}(\mu^1_0,\mu^1_1).
   \end{equation}
   If $\phi$ is strictly convex and $\tilde\phi$ has a
   sublinear growth w.r.t.\ $\rho$
   (i.e.\ $\tilde\varphi_\infty\equiv0$), then
   for every $\mu_0,\mu_1\in \RPM(\Rd)$ with
   $\cW_{\phi,\gamma}(\mu_0,\mu_1)<+\infty$ there exists
   a \emph{unique} mimimizer $(\mu,\nnu)\in \CE\phi\gamma01\Rd$
   of \eqref{eq:10}.
\end{theorem}
\begin{proof}
  Let $(\mu^j,\nnu^j)\in \CCE\phi\gamma01\Rd{\mu_0^j}{\mu_1^j}$
  be two minimizers of \eqref{eq:10}, $j=0,1$.
  For $\tau\in[0,1]$ we set
  $\mu^\tau_t:=(1-\tau)\mu^0_t+\tau\mu^1_t$, $\nnu^\tau_t:=
  (1-\tau)\nnu^0_t+\tau\nnu^1_t$.
  Since $(\mu^\tau,\nnu^\tau)\in \cce01\Rd{\mu_0^\tau}{\mu_1^\tau}$,
  the convexity of $\phi$ yields
  \begin{align*}
    \cW^p_{\phi,\gamma}(\mu_0^\tau,\mu^\tau_1)&\le
    \int_0^1\Phi(\mu^\tau_t,\nnu^\tau_t|\gamma)\,\d t\le 
    \int_0^1 
    \Big((1-\tau)\Phi(\mu^0_t,\nnu^0_t|\gamma)+
    \tau\Phi(\mu^1_t,\nnu^1_t|\gamma)\Big)\,\d t
    \\&=
    (1-\tau)\cW^p_{\phi,\gamma}(\mu^0_0,\mu^0_1)+\tau
    \cW^p_{\phi,\gamma}(\mu^1_0,\mu^1_1).
  \end{align*}
  Let us now suppose that $\phi$ is strictly convex
  and sublinear. Setting, as usual,
  $\mu^\tau_t=\rho^\tau_t\gamma+(\mu^\tau_t)^\perp$,
  $\nnu^\tau_t=\ww^\tau_t\gamma$,
  we have for a.e.\ $t\in (0,1)$
  \begin{equation}
    \label{eq:cap5:6}
    \Phi(\mu_t^\tau,\nnu^\tau_t|\gamma)\le
    (1-\tau)\int_\Rd \phi(\rho^0_t,\ww^0_t)\,\d\gamma+
    \tau\int_\Rd \phi(\rho^1_t,\ww^1_t)\,\d\gamma
  \end{equation}
  and the inequality is strict unless
  $\rho^0_t\equiv \rho^1_t$ and $\ww^0_t\equiv \ww^1_t$ for
  $\gamma$-a.e.\ $x\in \Rd$.
  If $\mu^0_0=\mu^1_0$ and $\mu^0_1=\mu^1_1$,
  two minimizers should satisfy
  \begin{displaymath}
    \rho^0_t(x)=\rho^1_t(x),\
    \ww^0_t(x)=\ww^1_t(x)\quad \gamma\text{-a.e.,}
    \quad 
    \nnu^0_t=\nnu^1_t\quad\text{for $\Leb 1$-a.e. }t\in (0,1).
  \end{displaymath}
  Since $(\mu^i,\nnu^i)$ are solutions of the continuity equation,
  taking the difference we obtain
  \begin{displaymath}
    \partial_t\big((\mu^0_t)^\perp
    -(\mu^1_t)^\perp\big)=
    \partial_t (\mu^0_t-\mu^1_t)=
    -\nabla\cdot(\nnu^0_t-\nnu^1_t)=0\quad\text{in }\Rd\times (0,1).
  \end{displaymath}
  The difference $(\mu^0_t)^\perp-(\mu^1_t)^\perp$ is then
  independent of time and vanishes
  at $t=0$, so that $\mu^0_t=\mu^1_t$ for
  every $t\in [0,1]$.  
\qed\end{proof}
\begin{theorem}[Subadditivity]
  \label{thm:subadditivity}
  For every $\mu_i^j\in \RPM(\Rd)$, $i,j=0,1$,
  we have
  \begin{equation}
    \label{eq:cap4:44}
    \cW_{\phi,\gamma}(\mu^0_0+\mu_0^1,\mu_1^0+\mu_1^1)\le
    \cW_{\phi,\gamma}(\mu_0^0,\mu_1^0)+\cW_{\phi,\gamma}(\mu_0^1,\mu_1^1).
  \end{equation}
  In particular
  \begin{equation}
    \label{eq:cap4:45}
    \cW_{\phi,\gamma}(\mu_0+\sigma,\mu_1+\sigma)\le
    \cW_{\phi,\gamma}(\mu_0,\mu_1)\quad
    \forall\, \sigma\in \RPM(\Rd).
  \end{equation}
\end{theorem}
\begin{proof}
  Let $(\mu^j,\nnu^j)\in \CCE\phi\gamma01\Rd{\mu_0^j}{\mu_1^j}$
  be as in the proof
  of the previous Theorem. Since $(\mu^0+\mu^1,\nnu^0+\nnu^1)
  \in \cce01\Rd{\mu_0^0+\mu_0^1}{\mu_1^0+\mu_1^1}$,
  we get
  \begin{align*}
    &\cW_{\phi,\gamma}(\mu^0_0+\mu_0^1,\mu_1^0+\mu_1^1)\le
    \int_0^1\Big(\Phi(\mu^0_t+\mu^1_t,\nnu^0_t+\nnu^1_t|\gamma)\Big)^{1/p}\,\d t
    \\&\le 
    \int_0^1\Big[\Big(\Phi(\mu^0_t+\mu^1_t,\nnu^0_t|\gamma)\Big)^{1/p}
    +
    \Big(\Phi(\mu^0_t+\mu^1_t,\nnu^1_t|\gamma)\Big)^{1/p}\Big]\,\d t
    \\&\le
    \int_0^1\Big[\Big(\Phi(\mu^0_t,\nnu^0_t|\gamma)\Big)^{1/p}+
    \Big(\Phi(\mu^1_t,\nnu^1_t|\gamma)\Big)^{1/p}\Big]\,\d t
    =
    \cW_{\phi,\gamma}(\mu^0_0,\mu^0_1)+\cW_{\phi,\gamma}(\mu^1_0,\mu^1_1).
    \quad\qed
  \end{align*}
\end{proof}
\begin{proposition}[Rescaling]
 For every $\mu_0,\mu_1\in \RPM(\Rd)$ and $\lambda>0$ we have
 \begin{equation}
 \label{eq:resc1}
  W^p_{\phi,\lambda\gamma}(\lambda\mu_0,\lambda\mu_1)=
  \lambda \cW_{\phi,\gamma}^p(\mu_0,\mu_1),
\end{equation}
\begin{equation}
\label{eq:resc2}
 \begin{cases}
  W^p_{\phi,\gamma}(\lambda\mu_0,\lambda\mu_1)\le
  \lambda^p \cW_{\phi,\gamma}^p(\mu_0,\mu_1)&\text{if }\lambda\ge1\\
  W^p_{\phi,\gamma}(\lambda\mu_0,\lambda\mu_1)\le
  \lambda \cW_{\phi,\gamma}^p(\mu_0,\mu_1)&\text{if }\lambda\le 1.
 \end{cases}
\end{equation}
\end{proposition}
\begin{proof}
 \eqref{eq:resc1} follows from the corresponding property
% \begin{displaymath}
$ \Phi(\lambda\mu,\lambda\nnu|\lambda\gamma)=\lambda \Phi(\mu,\nnu|\gamma)$.
%\end{displaymath}
Analogously, the monotonicity and homogeneity properties of $\phi$ yield
\begin{displaymath}
 \phi(\lambda\rho,\lambda\ww)\le \phi(\rho,\lambda\ww)=\lambda^p\phi(\rho,\ww)\quad
 \text{if }\lambda>1;
\end{displaymath}
the convexity of $\phi$ and the fact that $\phi(0,0)=0$ yield
\begin{displaymath}
 \phi(\lambda\rho,\lambda\ww)\le \lambda \phi(\rho,\ww)
 \quad\text{if }\lambda<1.
\end{displaymath}
\eqref{eq:resc2} follows immediately by the previous inequalities.
\qed\end{proof}

\begin{proposition}[Monotonicity]
  \label{prop:monotonicity}
If $\gamma_1\ge \gamma_2$ and $\phi_1\le \phi_2$, then
for every $\mu_0,\mu_1\in \RPM(\Rd)$ we have
\begin{equation}
 \label{eq:comparison}
 \cW_{\phi_1,\gamma_1}(\mu_0,\mu_1)\le 
  \cW_{\phi_2,\gamma_2}(\mu_0,\mu_1).
\end{equation}
\end{proposition}
\begin{theorem}[Convolution]
  \label{thm:wass_conv}
  Let $k\in C^\infty_{\rm c}(\Rd)$ be a nonnegative convolution kernel
  with $\int_\Rd k(x)\,\d x=1$ and let
  $k_\eps(x):=\eps^{-d}k(x/\eps)$.
  For every $\mu_0,\mu_1\in \RPM(\Rd)$ we have
  \begin{align}
    \label{eq:wass_kernel}
    \cW_{\phi,\gamma\ast k_\eps}(\mu_0\ast k_\eps,\mu_1\ast k_\eps)&\le 
    \cW_{\phi,\gamma}(\mu_0,\mu_1)\quad
    \forall\, \eps>0;\\
    \label{eq:wass_kernel2}
    \lim_{\eps\down0} \cW_{\phi,\gamma\ast k_\eps}(\mu_0\ast k_\eps,\mu_1\ast k_\eps)&= 
    \cW_{\phi,\gamma}(\mu_0,\mu_1).
\end{align}
\end{theorem}
\begin{proof}
  Let $(\mu,\nnu)\in \CCE\phi\gamma01\Rd{\mu_0}{\mu_1}$
  be an optimal connecting curve as in Theorem \ref{thm:3}
 and let us set $\mu^\eps_t=\mu_t\ast k_\eps,\nnu^\eps_t:=\nnu_t\ast k_\eps$.
 Since $(\mu^\eps,\nnu^\eps)\in
 \cce01\Rd{\mu_0^\eps}{\mu_1^\eps}$,
 \eqref{eq:wass_kernel} then follows by \eqref{eq:cap2:57}
 whereas
 \eqref{eq:wass_kernel2}
 is a consequence of Theorem \ref{thm:2}
\qed\end{proof}
\begin{remark}[Smooth approximations]
  \upshape
  For a given curve $(\mu,\nnu)\in \CE\phi\gamma01\Rd$
  the convolution technique of the previous Theorem exhibits an
  approximations $(\mu^\eps,\nnu^\eps)$ in
  $\CE\phi{\gamma^\eps}01\Rd$, $\gamma^\eps:=\gamma\ast k_\eps$,
  which enjoy some useful properties:
  \begin{enumerate}
  \item $\mu^\eps=\rho^\eps\Leb d$, $\nnu^\eps=\ww^\eps\Leb d$ with
    $\rho^\eps,\ww^\eps\in C^\infty(\Rd)$; if $\mu_0(\Rd)<+\infty$ and
    $\sftm_{-\kappa}(\gamma)<+\infty$ (recall Theorem
    \ref{thm:narrow_convergence}),
    then $\rho^\eps$ is also uniformly bounded.
  \item If $\supp(k)\subset \overline{\Ball 1}$ then $\rho^\eps,\ww^\eps$ are supported in
    $G^\eps:=\big\{x\in \Rd:{\rm dist}(x,G)\le \eps\big\}$, $G=\supp(\gamma)$.
  \item  $\rho^\eps,\ww^\eps$ are classical solution of the continuity equation
    \begin{displaymath}
      \partial_t\rho^\eps+\nabla \cdot\ww^\eps=0\quad
      \text{in }\Rd\times(0,1).
    \end{displaymath}
  \item If $(\mu,\nnu)$ is also a geodesic,
    $\displaystyle \int_\Rd
    \phi(\rho^\eps_t,\ww^\eps_t)\,\d\gamma^\eps\le
    \Phi(\rho_t,\nnu_t|\gamma)=\cW_{\phi,\gamma}^p(\mu_0,\mu_1)$ for
    every $t\in [0,1]$.
  \end{enumerate}
\end{remark}
\subsection{\bfseries Absolutely continuous curves and geodesics}
\label{subsec:curves}
We now study absolutely continuous curves with respect to
$\cW_{\phi,\gamma}$ and their length.
Let us first recall
(see e.g.\ \cite[Chap.\ 1]{Ambrosio-Gigli-Savare05})
that
a curve $t\mapsto \mu_t\in \Mloc(\Rd)$, $t\in [0,T]$, is absolutely
continuous w.r.t.\ $\cW_{\phi,\gamma}$
if there exists a function $m\in L^1(0,T)$ such that
\begin{equation}
  \label{eq:cap4:37}
  \cW_{\phi,\gamma}(\mu_{t_1},\mu_{t_0})\le
  \int_{t_0}^{t_1}m(t)\,\d t\quad
  \forall\, 0\le t_0<t_1\le T.
\end{equation}
The curve has finite $p$-energy if moreover $m\in L^p(0,T)$.
The metric derivative $|\mu'|$ of an absolutely continuous curve
is defined as
\begin{equation}
  \label{eq:cap4:38}
  |\mu_t'|:=\lim_{h\to0}\frac{\cW_{\phi,\gamma}(\mu_{t+h},\mu_t)}{|h|},
\end{equation}
and it is possible to prove that
$|\mu_t'|$ exists  and satisfies
$|\mu_t'|\le m(t)$ for $\Leb 1$-a.e.\ $t\in (0,T)$.
The length of $\mu$ is then defined as
the integral of $|\mu'|$ in the interval $(0,T)$.

\begin{theorem}[Absolutely continuous curves and
  their metric velocity]
  A curve $t\mapsto \mu_t$, $t\in [0,T]$, is absolutely continuous
  with respect to $\cW_{\phi,\gamma}$ if and only if
  there exists a Borel family of measures
  $(\nnu_t)_{t\in (0,T)}$ in $\Mloc(\Rd;\Rd)$ such that
  $(\mu,\nnu)\in \CE\phi\gamma0T\Rd$ and
  \begin{equation}
    \label{eq:cap4:39}
    \int_0^T \Big(\Phi(\mu_t,\nnu_t|\gamma)\Big)^{1/p}\,\d t<+\infty.
  \end{equation}
  In this case we have
  \begin{equation}
    \label{eq:cap4:40}
    |\mu_t'|^p\le \Phi(\mu_t,\nnu_t|\gamma)\quad
    \text{for $\Leb 1$-a.e.\ }t\in (0,T),
  \end{equation}
  and there exists a unique Borel family $\tilde\nnu_t$ such that
  \begin{equation}
    \label{eq:cap4:41}
    |\mu_t'|^p= \Phi(\mu_t,\tilde\nnu_t|\gamma)\quad
    \text{for $\Leb 1$-a.e.\ }t\in (0,T).
  \end{equation}
\end{theorem}
\begin{proof}
  One implication is trivial: if $(\mu,\nnu)\in \CE\phi\gamma0T\Rd$
  and \eqref{eq:cap4:39} holds, then
  \eqref{eq:cap3:29} yields
  \begin{equation}
    \label{eq:cap4:42}
     \cW_{\phi,\gamma}(\mu_{t_1},\mu_{t_0})\le
     \int_{t_0}^{t_1} \Big(\Phi(\mu_t,\nnu_t|\gamma)\Big)^{1/p}\,\d t,
   \end{equation}
   so that $\mu$ is absolutely continuous and
   \eqref{eq:cap4:40} holds.

   Conversely, let us assume that $\mu$ is an absolutely continuous
   curve with length $L$. A standard reparametrization results
   \cite[Lemma 1.1.4]{Ambrosio-Gigli-Savare05}
   shows that it is not restrictive to assume
   that $\mu$ is a Lipschitz map.
   We fix an integer $N>0$, a step size $\tau:=2^{-N}T$,
   and a family of geodesics $(\mu^{k,N},\nnu^{k,N})\in
   \CCE\phi\gamma{(k-1)\tau}{k\tau}\Rd{\mu_{(k-1)\tau}}{\mu_{k\tau}}$,
   $k=1,\cdots, 2^N$, such that
   \begin{equation}
     \label{eq:cap4:43}
     \tau \Phi(\mu^{k,N}_t,\nnu^{k,N}_t|\gamma)
     =\tau^{1-p}\,W^p_{\phi,\gamma}(\mu_{(k-1)\tau},\mu_{k\tau})\le
     \int_{(k-1)\tau}^{k\tau}|\mu_t'|^p\,\d t.
   \end{equation}
   Let $(\mu^N,\nnu^N)\in \CE\phi\gamma0T\Rd$ be the curve
   obtained by gluing together all the geodesics
   $(\mu^{k,N},\nnu^{k,N})$.
   Applying Corollary \ref{cor:compactness_continuity2}, we can find
   a subsequence $(\mu^{N_h},\nnu^{N_h})$ and a couple
   $(\mu,\nnu)\in \CE\phi\gamma0T\Rd$ such that
   $\mu^{N_h}_t\weaksto \tilde\mu_t$ for every $t\in [0,T]$
   and $\nnu^{N_h}\weaksto \nnu$ in $\Mloc(\Rd\times(0,T);\Rd)$.
   It is immediate to check that $\mu_t\equiv\tilde\mu_t$ for every
   $t\in [0,T]$ and
   \begin{displaymath}
     \int_0^T|\mu_t'|^p\,\d t\ge
     \liminf_{h\up+\infty}
     \int_0^T \Phi(\mu^{N_h}_t,\nnu^{N_h}_t|\gamma)\,\d t\ge
     \int_0^T \Phi(\mu_t,\nnu_t|\gamma)\,\d t
     \ge \int_0^T |\tilde\mu_t'|^p\,\d t,
   \end{displaymath}
   which concludes the proof.
 \qed\end{proof}
 \begin{corollary}[Geodesics]
   For every $\mu\in \RPM(\Rd)$ the space
   $\mathcal M_{\phi,\gamma}[\mu]$ is a geodesic space,
   i.e.\ every couple $\mu_0,\mu_1\in \mathcal
   M_{\phi,\gamma}[\mu]$
   can be connected by a (minimal, constant speed)
   geodesic $t\in [0,1]\mapsto \mu_t\in \mathcal
   M_{\phi,\gamma}[\mu]$
   such that
   \begin{equation}
     \label{eq:cap5:7bis}
     \cW_{\phi,\gamma}(\mu_s,\mu_t)=|t-s|\,\cW_{\phi,\gamma}(\mu_0,\mu_1)
     \quad\forall\, s,t\in [0,1].
   \end{equation}
   All the (minimal, constant speed)
   geodesics satisfies the continuity equation
   \eqref{eq:continuity1}
   for
   a Borel family of vector valued measures $(\nnu_t)_{t\in (0,1)}$ such that
   \begin{equation}
     \label{eq:cap5:8}
     \Phi(\mu_t,\nnu_t|\gamma)=\cW_{\phi,\gamma}^p(\mu_0,\mu_1)\quad
     \text{for a.e. }t\in (0,1).
   \end{equation}
   If $\phi$ is strictly convex and sublinear, geodesics are unique.  
 \end{corollary}
 \begin{remark}[A formal differential characterization of geodesics]
   \upshape
   Arguing as in \cite[Chap. 3]{Otto-Villani00}, 
   it would not    be difficult to show that a geodesic 
   $\mu_t=\rho_t\Leb d$ with respect to 
   $W_{2,\alpha;\Leb d}$ should satisfy the system of 
   nonlinear PDE's in 
   $\Rd\times (0,1)$
   \begin{displaymath}
     \left\{\begin{aligned}
       \partial_t \rho+
       \nabla \cdot( \rho^\alpha\nabla      \psi)&=0,\\
       \partial_t\psi +\frac \alpha2               \rho^{\alpha-1}\big|\nabla \psi \big|^2&=0,      \end{aligned}
\right.
   \end{displaymath}
   for some potential $\psi$. 
   Unlike the    Wasserstein case, however, 
   the two equations are coupled, and it is not possible
   to solve the second Hamilton-Jacobi equation in 
   $\psi$ independently of the first scalar 
   cosnervation law. 
   In the present paper, we do not explore this 
   direction.
 \end{remark}
We can give a more precise description of the vector measure
$\tilde\nnu$ satisfying the optimality condition
\eqref{eq:cap4:41}.
For every measure $\mu\in \RPM(\Rd)$ we set
\begin{equation}
  \label{eq:cap4:46}
  \begin{aligned}
    \Tan\phi\gamma\mu:=\Big\{& \nnu\in \Mloc(\Rd;\Rd):
    \Phi(\mu,\nnu|\gamma)<+\infty,\\&
    \Phi(\mu,\nnu|\gamma)\le
    \Phi(\mu,\nnu+\eeta|\gamma)\quad \forall\, \eeta\in
    \Mloc(\Rd;\Rd):\nabla\cdot\eeta=0\Big\}.
  \end{aligned}
\end{equation}
Observe that for every $\nnu\in \Mloc(\Rd;\Rd)$ such that
$\Phi(\mu,\nnu|\gamma)<+\infty$ there exists a unique
$\tilde\nnu:=\Pi(\nnu)\in \Tan\phi\gamma\mu$ such that
$\nabla\cdot(\tilde\nnu-\nnu)=0$.
In fact, the set $K(\nnu):= \big\{\nnu'\in \Mloc(\Rd;\Rd):
\nabla(\nnu'-\nnu)=0\big\}$ is weakly$^*$ closed
and, by the estimate \eqref{eq:Cap3:1} the sublevels
of the functional $\nnu'\mapsto\Phi(\mu,\nnu'|\gamma)$
are weakly$^*$ relatively compact.
Therefore, a minimizer $\tilde\nnu$ exists and it is also unique,
being $\Phi(\mu,\cdot|\gamma)$ strictly convex.
\begin{corollary}
  \label{thm:tangent_measure}
  Let $(\mu,\nnu)\in \CE\phi\gamma0T\Rd$ so that 
  $\mu$ is absolutely continuous w.r.t.\
  $\cW_{\phi,\gamma}$.
  The vector measure $\nnu$ satisfies the
  optimality condition \eqref{eq:cap4:41} if and only if
  $\nnu_t\in \Tan\phi\gamma{\mu_t}$ for $\Leb 1$-a.e.\ $t\in (0,T)$.
\end{corollary}
Let us consider the particular case of Example \ref{ex:hp}
in the case of a differentiable norm
$\|\cdot\|$ with associated duality map $j_1={\mathrm D}\|\cdot\|$.
We denote by $j_p(\ww)=\|\ww\|^{p-2}j_1(\ww)$ the $p$-duality map,
i.e.
the differential of $\frac 1p\|\cdot\|^p$ and
% the Euclidean norm $\|\cdot\|:=|\cdot|$,
% i.e.
% \begin{equation}
%   \label{eq:cap4:47}
%   \phi(\rho,\ww):=h(\rho)|\ww/h(\rho)|^p,\quad
%   \tilde\phi(\rho,\zz):=h(\rho)|\zz|^q,
% \end{equation}
we suppose that the concave function $h:[0,+\infty)\to [0,+\infty)$ satisfies
\begin{equation}
  \label{eq:cap4:49}
  \lim_{r\down0}h(r)=\lim_{r\up+\infty}r^{-1}h(r)=0.
\end{equation}
For every nonnegative Radon measure $\mu\in \RPM(\Rd)$
whose support is a subset of $\supp(\gamma)$,
we define the Radon measure $h(\mu|\gamma)$ by
\begin{equation}
  \label{eq:cap4:48}
  h(\mu|\gamma):=h(\rho)\cdot\gamma\quad
  \text{where }\rho:=\frac{\d \mu}{\d \gamma}.
\end{equation}
Observe that $h(\mu|\gamma)\ll\gamma$ even if $\mu$ is singular
w.r.t.\ $\gamma$.
\begin{theorem}
  \label{thm:charnnu}
  Let $\mu\in \RPM(\Rd)$ and $\phi$ as in \eqref{eq:cap2:39}
  with $h$ satisfying \eqref{eq:cap4:49}.
  A vector measure $\nnu$ satisfies $\Phi(\mu,\nnu|\gamma)<+\infty$
  iff $\nnu=\vv\,h(\mu|\gamma)$ for some vector field
  $\vv\in L^p_{h(\mu|\gamma)}(\Rd;\Rd)$.
  Moreover, $\nnu\in \Tan\phi\gamma\mu$ if and only if
  the vector field $\vv$ satisfies
  \begin{equation}
    \label{eq:cap4:50}
    j_p(\vv)\in \overline{\big\{\nabla\zeta:\zeta\in C^\infty_{\rm
        c}(\Rd)
      \big\}}^{L^p_{h(\mu|\gamma)}(\Rd;\Rd)}.
  \end{equation}
\end{theorem}
\begin{proof}
  Being $h$ sublinear, the functional $\Phi$ admits the representation
  \begin{equation}
    \label{eq:cap4:51}
    \Phi(\mu,\nnu|\gamma)=
    \int_\Rd \phi(\rho,\ww)\,\d\gamma=
    \int_\Rd h(\rho)\|\ww/h(\rho)\|^p\,\d\gamma=
    \int_\Rd \|\vv\|^p\,\d h(\mu|\gamma),
  \end{equation}
  where $\mu=\rho\gamma+\mu^\perp$ and
  $\nnu=\ww\gamma=h(\rho)\vv\,\gamma$.
  The condition $\nnu=\vv h(\mu|\gamma)\in
  \Tan\phi\gamma\mu$ is then
  equivalent to
  \begin{displaymath}
    \int_\Rd \phinorm{\vv}^p\,\d h(\mu|\gamma)\le
    \int_\Rd
    \phinorm{\vv+\zz}^p\,\d h(\mu|\gamma)\qquad
    \forall\, \zz\in L^p(h(\mu|\gamma)):
    \nabla\cdot\big(\zz\, h(\mu|\gamma)\big)=0.
  \end{displaymath}
  Thanks to the convexity of $||\cdot||^p$, the previous condition is
  equivalent
  to
  \begin{equation}
    \label{eq:cap4:53}
    \int_\Rd j_p(\vv)\cdot \zz\,\d h(\mu|\gamma)=0\quad
    \forall\, \zz\in L^p_{h(\mu|\gamma)}(\Rd;\Rd):\quad
    \int_\Rd \zz\cdot\nabla\zeta\,\d h(\mu|\gamma)=0,
  \end{equation}
  i.e.\ $j_p(\vv)$ belongs to the closure of $\{\nabla\zeta: \zeta\in
  C^\infty_{\rm c}(\Rd)\}$ in $L^p_{h(\mu|\gamma)}(\Rd;\Rd)$.
\qed\end{proof}

\subsection{\bfseries Comparison with Wasserstein and $\dot W^{-1,p}$
  distances.}
\label{subsec:comp}
\begin{theorem}
  \label{thm:4pre}
  If $\gamma(\Rd)<+\infty$ then for every $\mu_0,\mu_1\in
  \RPM(\Rd)$ and $\alpha<1$ we have
  \begin{equation}
    \label{eq:cap5:4}
    W_{p/\theta}(\mu_0,\mu_1)\le
    W_{p/\theta,1;\gamma}(\mu_0,\mu_1)\le \gamma(\Rd)^{1/\kappa}
    W_{p,\alpha;\gamma}(\mu_0,\mu_1),
  \end{equation}
  where, as usual, $\theta=(1-\alpha)p+\alpha.$
\end{theorem}
\begin{proof}
  Let $(\mu,\nnu)\in
  \CCE{\phi_{p,\alpha}}\gamma01\Rd{\mu_0}{\mu_1}$
  be an optimal curve, so that
  \begin{equation}
    \label{eq:cap5:2}
    W_{p,\alpha;\gamma}^p(\mu_0,\mu_1)=
    \int_0^1 \Phi_{p,\alpha}(\mu_t,\nnu_t)\, \d t=
     \int_0^1 \int_\Rd (\rho_t)^{\theta-p}
    \left|\ww_t\right|^p\,\d\gamma
    \, \d t,
  \end{equation}
  where $\mu_t:=\rho_t\gamma+\mu_t^\perp$,
  $\nnu_t=\ww_t\gamma\ll\gamma$. 
  H\"older inequality yields
  \begin{displaymath}
    W_{p/\theta,1;\gamma}^{p/\theta}(\mu_0,\mu_1)\le 
    \int_0^1 \int_\Rd (\rho_t)^{1-p/\theta}
    \left|\ww_t\right|^{p/\theta}\,\d\gamma
    \, \d t\le
    \gamma(\Rd)^{1-1/\theta}W_{p,\alpha;\gamma}^{p/\theta}(\mu_0,\mu_1).
    \quad\qed
  \end{displaymath}
\end{proof}
\begin{theorem}
  \label{thm:4postpre}
  Let us suppose that $\sftm_{-k}(\gamma)<+\infty$,
  $\kappa=p/(\theta-1)=q/(1-\alpha)$,
  let $\mu_0,\mu_1\in \Probabilities\Rd$, and let
  $\kappa^*=\kappa/(\kappa-1)$
  be the H\"older's conjugate exponent of $\kappa$.
  % \begin{equation}
%     \label{eq:cap5:12}
%     r:=\kappa^*,\quad\text{i.e.}\quad
%     \frac 1r:=1-\frac 1\kappa.
%   \end{equation}
  Then 
  \begin{equation}
    \label{eq:cap5:11}
    \|\mu_0-\mu_1\|_{\dot W^{-1,\kappa^*}_\gamma}=
    W_{\kappa^*,0;\gamma}(\mu_0,\mu_1)\le
    W_{p,\alpha;\gamma}(\mu_0,\mu_1).
      \end{equation}
\end{theorem}
\begin{proof}
  We keep the same notation of the previous Theorem, setting
  \begin{displaymath}
    \tau= p/r:=1+p-\theta,\quad
    \tau^*:=\frac{\tau}{\tau-1}=1+\frac1{p-\theta},\quad
    x=(\tau^*)^{-1}=\frac{p-\theta}{1+p-\theta}.
  \end{displaymath}
  Observing that
  $\mu_t\in \Probabilities\Rd$ thanks to Theorem
  \ref{thm:moment_estimate2}, we obtain 
  \begin{align*}
    &W_{r,0;\gamma}^r(\mu_0,\mu_1) \le
     \int_0^1\int_\Rd |\ww_t|^r\,\d\gamma\,\d t=
     \int_0^1\int_\Rd \big(\rho_t\big)^{x}
     \big(\rho_t\big)^{-x}
    |\ww_t|^r\,\d\gamma\,\d t\\
     &\le  \int_0^1 \Big(\int_\Rd
     \rho^{-x\tau}|\ww_t|^{r\tau}\,\d\gamma\Big)^{1/\tau}\,\d t=
     \Big(\int_0^1 \int_\Rd
     \rho^{\theta -p}|\ww_t|^{p}\,\d\gamma\,\d t
       \Big)^{1/\tau}
     =W^{r}_{p,\alpha;\gamma}(\mu_0,\mu_1).
     \quad\qed
  \end{align*}
\end{proof}
\begin{theorem}[Comparison with $W_p$]
  \label{thm:4}
  Assume that $\gamma\in \RPM(\Rd)$ is 
  a bounded perturbation of a log-concave measure
  (e.g.\ $\gamma=fe^{-V}\Leb d$, where $V$ is a convex function and
  $f$ nonnegative and bounded).
  If $\mu_i=s_i\gamma\in \Probabilities\Rd$ with $s_i\in L^\infty(\gamma)$ and
  $\sfm_p(\mu_i)\le L<+\infty$ then $\cW_{\phi,\gamma}(\mu_0,\mu_1)<+\infty$
  and there exists a constant
  $\sfC$ only depending on $L,\ \phi,$ and $\gamma$ such that
  \begin{equation}
   \label{eq:upper_bound}
   \cW_{\phi,\gamma}(\mu_0,\mu_1)\le 
   \sfC\, W_p(\mu_0,\mu_1).
  \end{equation} 
\end{theorem}
\begin{proof}
 It is not restrictive to assume that $\gamma$ is log-concave.
 We can then consider the optimal plan $\Sigma\in \FPM(\Rd\times \Rd)$
 induced by the $p$-Wasserstein distance \eqref{defwaspre}
 between $\mu_0$ and $\mu_1$ and the interpolant $\mu_t$
 defined as
 \begin{equation}
  \mu_t(A)=\Sigma\big(\{(x_0,x_1)\in \Rd\times\Rd:
  (1-t)x_0+tx_1\in A\}\big)\quad
  \forall\, A\in \BorelSets\Rd.
 \end{equation}
 It is possible to prove (see e.g.\ \cite[Theorems 7.2.2, 8.3.1,
 9.4.12]{Ambrosio-Gigli-Savare05})
 that $\mu_t$ is the geodesic interpolant
 between $\mu_0$ and $\mu_1$,
 it satisfies the continuity equation
 \begin{displaymath}
  \partial_t\mu_t+\nabla\cdot \nnu_t=0\quad
  \text{in }\Rd\times(0,1)
 \end{displaymath}
 with respect to a vector valued measure $\nnu_t=\vv_t\mu_t\ll\mu_t$ 
 where the vector field $\vv_t$ satisfies
 \begin{displaymath}
  \int_0^1 \Phi_p(\mu_t,\nnu_t)\,\d t=
  \int_0^1 \int_\Rd |\vv_t(x)|^p\,\d\mu_t(x)\,\d t=W_p^p(\mu_0,\mu_1),
 \end{displaymath}
 and finally 
 $\mu_t=s_t\gamma$ with $\|s_t\|_{L^\infty(\gamma)}\le 
 L:=\max\big(\|s_0\|_{L^\infty(\gamma)},\|s_1\|_{L^\infty(\gamma)}\big)$.
 Observe that, being $s_t(x)\le L $
 for $\gamma$-a.e.\ $x\in \Rd$ and $\phi(0,0)=0$,
 Theorem \ref{le:phi_list} yields
 \begin{displaymath}
  \phi(s_t,s_t\vv_t)\le
  \frac{s_t}{L}\phi(L,L\vv_t)\le \sfC_L s_t |\vv_t|^p\quad
  \gamma\text{-a.e.},
 \end{displaymath}
 so that
 \begin{displaymath}
   \Phi(\mu_t,\nnu_t)=
   \int_\Rd \phi(s_t,s_t\vv_t)\,\d \gamma(x)\le 
   \sfC_L \int_\Rd |\vv_t|^p s_t\,\d \gamma=
    \sfC_L \int_\Rd |\vv_t|^p \,\d \mu_t,
 \end{displaymath}
 and therefore
 \begin{displaymath}
   \cW_{\phi,\gamma}^p(\mu_0,\mu_1)\le
   \int_0^1 \Phi(\mu_t,\nnu_t)\,\d t\le 
  \sfC_L \int_0^1 \int_\Rd |\vv_t|^p \,\d \mu_t\,\d t=
  \sfC_L W_p^p(\mu_0,\mu_1).
  \quad\qed
 \end{displaymath}
\end{proof}
\begin{corollary}
  \label{cor:very_obvious}
  If $\mu_i=s_i\Leb d\in \Probabilities\Rd$ have
  $L^\infty$-densities with compact support (or, more generally,
  finite $p$-momentum), then $\cW_{\phi,\Leb
    d}(\mu_0,\mu_1)<+\infty$.
\end{corollary}
\begin{theorem}
\label{thm:4bis}
  If $\mu_i=s_i\gamma$ with $s_i\ge
  L>0$ $\gamma$-a.e.\ in $\Rd$, then there exists a constant
  $C$ depending on $L$ and $\phi$ such that
  \begin{equation}
  \label{eq:thm4bis}
   \cW_{\phi,\gamma}(\mu_0,\mu_1)\le C_L
   \|\mu_0-\mu_1\|_{\dot W^{-1,p}_\gamma}.
  \end{equation}
\end{theorem}
\begin{proof}
Let us first observe that if $\|\mu_0-\mu_1\|_{\dot W^{-1,p}_\gamma}<+\infty$
then there exists $\ww\in L^p_\gamma(\Rd;\Rd)$ such that 
\begin{equation}
 \label{eq:dual_representation}
 -\nabla\cdot \nnu=\mu_1-\mu_0,\quad
 \nnu:=\ww\gamma,\quad
 \int_\Rd |\ww|^p\,\d\gamma=\|\mu_0-\mu_1\|_{\dot W^{-1,p}_\gamma}^p.
\end{equation} 
In fact, in the Banach space $X:=L^q_\gamma(\Rd;\Rd)$ 
we can consider the linear space 
%\begin{displaymath}
$ Y:=\big\{\D\zeta:\zeta\in C^1_{\rm c}(\Rd)\big\}$
%\end{displaymath}
and the linear functional
\begin{displaymath}
 \la \ell,\yy\ra:=\int_\Rd \zeta\,\d(\mu_1-\mu_0)\quad
 \text{if }\yy=\D\zeta\quad\text{for some $\zeta\in C^1_{\rm c}(\Rd)$}.
\end{displaymath}
$\ell$ is well defined and satisfies
%\begin{equation}
$ \big|\la\ell,\yy\ra\big|\le 
\|\mu_0-\mu_1\|_{\dot W^{-1,p}_\gamma}\|\yy\|_{L^q_\gamma(\Rd;\Rd)}$
for every $\yy\in Y.$
%\end{equation}
Hahn-Banach Theorem and Riesz representation Theorem yield 
the existence of $\ww\in L^p(\Rd;\Rd)$ such that 
%\begin{equation}
% \label{eq:ell_representation}
$ \la \ell,\yy\ra=\int_\Rd \ww\cdot\yy\,d\gamma,$
%\end{equation}
which yields \eqref{eq:dual_representation}.
Setting $\mu_t=(1-t)\mu_0+t\mu_1$, it is then immediate to check that 
$(\mu_t,\nnu)\in \cce01\Rd{\mu_0}{\mu_1}$;
we can then compute
\begin{displaymath}
 \cW_{\phi,\gamma}^p(\mu_0,\mu_1)\le 
 \int_0^1 \int_\Rd \phi((1-t)s_0+ts_1,\ww)\,\d\gamma\,\d t\le
 \int_\Rd \phi(L,\ww)\,\d\gamma\le C_L \int_\Rd |\ww|^p\,\d\gamma,
\end{displaymath}
where we used the fact that $(1-t)s_0+ts_1\ge L$ $\gamma$-almost everywhere
and the map $\rho\mapsto \phi(\rho,\ww)$ is nonincreasing.
\qed\end{proof}
\subsection{\bfseries The case $\gamma=\Leb d$ and the Heat equation as
  gradient flow}
\label{subsec:Leb}
One of the most interesting cases corresponds to the choice
\begin{equation}
  \label{eq:cap4:54}
  \gamma:=\Leb d,\quad
  h_\alpha(\rho):=\rho^\alpha,\ 0<\alpha<1,\qquad
  \phi_{p,\alpha}(\rho,\ww):=\rho^{\alpha}|\ww/\rho^\alpha|^p.
\end{equation}
In this case 
the expression of the weighted Wasserstein distance becomes
\begin{displaymath}
  \begin{aligned}
    W_{p,\alpha;\Leb d}^p(\mu_0,\mu_1):=\min\Big\{& \int_0^1\!\!\int_\Rd
    \rho_t^\alpha |\vv_t|^p\,\d x\,\d t:\ 
    \partial_t\mu+\nabla\!\cdot(\rho^\alpha\vv)=0\ \text{in
    }\Rd\times(0,1)\\&
    \mu_t=\rho_t\Leb d+\mu_t^\perp,\qquad
    \mu\restr{t=0}=\mu_0,\ \mu\restr{t=1}=\mu_1\Big\}.
  \end{aligned}
\end{displaymath}
The metric $W_{p,\alpha;\Leb d}$ restricted to $\Probabilities\Rd$ is
complete
if 
%\begin{equation}
%  \label{eq:cap4:56}
$  d<\kappa=\frac{p}{\theta-1}=\frac{q}{1-\alpha}.$
%\end{equation}
\begin{remark}[$\Probabilities\Rd$ is not complete w.r.t.\
  $W_{p,\alpha;\Leb d}$ if $d>\kappa$]
  \label{rem:LebComp}
  \upshape
  The above condition
  is almost sharp; here is a simple
  counterexample in the case $d>\kappa$.
  We consider an initial probability measure with
  compact support
  $\mu_0=\rho_0\Leb d$, $\rho_0\in L^\infty(\Rd)$, 
  and, for $t\ge 0$,
  the family
  \begin{equation}
    \label{eq:cap5:13}
    \mu_t:=\rho_t\Leb d,\quad
    \rho_t(x):=e^{-dt }\rho_0(e^{-t}x),\quad
    \nnu_t:=x\mu_t=x\rho_t(x)\Leb d.
  \end{equation}
  It is easy to check that
  $(\mu,\nnu)\in \ce0{+\infty}\Rd$, $\mu_t(\Rd)=1$.
  Evaluating the functional $\Phi_t:=\Phi_{p,\alpha}(\mu_t,\nnu_t|\Leb d)$ we get
  \begin{align*}
    \Phi_t&=\int_\Rd \rho_t^{\theta-p}|\rho_t x|^p\,\d x=
    \int_\Rd e^{-d \theta t}\rho_0^\theta(e^{-dt} x)|x|^p\,\d x
    \\&=
    e^{dt -d\theta t +pt}
    \int_\Rd  e^{-dt}\rho_0^\theta(e^{-dt} x)|e^{-dt}
    x|^p\,\d x=
    e^{(d(1-\theta)+p)t}\int_\Rd \rho_0^\theta(y)|y|^p\,\d y
  \end{align*}
  so that
  \begin{displaymath}
    \Phi_t^{1/p}=\sfc\, e^{(1-d/\kappa)t},\quad
    \int_0^{+\infty}\Phi_t^{1/p}\,\d t=
    \sfc \frac\kappa{d-\kappa}<+\infty\quad
    \text{if }d>\kappa.
  \end{displaymath}
  If $d>\kappa$ we obtain a curve $t\mapsto \mu_t\in \Probabilities\Rd$ of finite length
  w.r.t. $W_{p,\alpha;\Leb d}$ (in particular
  $(\mu_n)_{n\in \N}$ is a Cauchy sequence) such that
  $\lim\limits_{t\up+\infty}\mu_t=0$
  in the weak$^*$ topology.  
\end{remark}
In the remaining part of this section,
we want to study the properties
of $W_{p,\alpha;\Leb d}$ with respect to the
heat flow.
We thus introduce
\begin{displaymath}
  g(x)=g_1(x)=\frac{1}{(4\pi)^{d/2}}e^{-|x|^2/4},\qquad
  g_t(x):=\frac{1}{(4\pi t)^{d/2}}e^{-|x|^2/4t}=
  t^{-d/2}g_1(x/\sqrt t),
\end{displaymath}
and for every $\mu\in  \RPM(\Rd)$ with $\sftm_\delta(\mu)<+\infty$
for some $\delta\le 0$, we set
\begin{equation}
  \label{eq:cap4:59}
  \mathcal S_t[\mu]=\mu\ast g_t=u_t\Leb d,\quad
  u_t(x)=S_t[\mu](x):=\int_\Rd g_t(x-y)\,\d\mu(y).
\end{equation}
It is well known that $u\in
C^\infty(\Rd\times(0,+\infty))$ and
\begin{equation}
  \label{eq:cap4:60}
  \partial_t u-\Delta u=0\quad\text{in }\Rd\times (0,+\infty),\qquad
  \mathcal S_t[\mu]\weaksto \mu\quad\text{as }t\down0.
\end{equation}
\begin{theorem}[Contraction property]
  \label{thm:heat_contracting}
  Let $\mu^0,\mu^1\in \RPM(\Rd)$ with
  $\sftm_\delta(\mu^i)<+\infty$
  and $\cW_{\phi,\Leb d}(\mu^1,\mu^2)<+\infty$.
  If $\mu^i_t:=\mathcal S_t[\mu^i]$ are the 
  corresponding solutions
  of the heat flow,
  %with initial datum $\mu^i$,
  then
  \begin{equation}
    \label{eq:cap4:61}
    \cW_{\phi,\Leb d}(\mu^1_t,\mu^2_t)\le 
    \cW_{\phi,\Leb d}(\mu^1,\mu^2)\quad
    \forall\, t>0.
  \end{equation}
\end{theorem}
\begin{proof}
  It sufficient to approximate the Gaussian kernel $g$ by a family
  of $C^\infty$ kernels $k^n$ with compact support and then apply
  Theorem \ref{thm:wass_conv}, observing that
  $k^n\ast\Leb d=\Leb d$.
\qed\end{proof}
We consider now the particular case of the 
$W_{2,\alpha;\Leb d}$ weighted
distance
with $\alpha>1-2/d$.
Let us first introduce the convex density function
(recall \eqref{eq:newclass_times:1}) 
\begin{equation}
  \label{eq:cap4:64}
  \psi_\alpha(\rho):=\frac
  1{(2-\alpha)(1-\alpha)}\rho^{2-\alpha},\quad
  \text{such that}\quad
  \psi_\alpha''(\rho)=\frac 1{h(\rho)}=\rho^{-\alpha},
\end{equation}
and the corresponding entropy functional
\begin{equation}
  \label{eq:cap4:65}
  \Psi_\alpha(\mu)=\Psi_\alpha(\mu|\Leb d):=
  \int_\Rd \psi_\alpha(\rho)\,\d x,\quad
  \text{if }\mu=\rho\Leb d\ll\Leb d.
\end{equation}
We also introduce the set
%\begin{equation}
%  \label{eq:cap4:66}
$  \mathcal Q:=\big\{\mu\in \Probabilities\Rd:\Psi(\mu)<+\infty\big\}.$
%\end{equation}
\begin{theorem}
  \label{le:energy_estimate}
  If $\mu\in \Probabilities\Rd$ then
  $\mu_t=\mathcal S_t[\mu]=u_t\Leb d\in \mathcal Q$ for every $t>0$,
  the map $t\mapsto \Psi_\alpha(\mu_t)$ is nonincreasing, and 
  it satisfies the
  energy identity 
  \begin{equation}
    \label{eq:cap4:67}
    \Psi_\alpha(\mu_t)+\int_s^t\Phi_{2,\alpha}(u_r,\nabla u_r)\,\d r=
    \Psi_\alpha(\mu_s)\quad\forall\, 0<s\le t<+\infty;
  \end{equation}
  when $\mu\in \mathcal Q$ then the previous identity holds even for
  $s=0$.
  Moreover, $\mu_t$ satisfies the Evolution Variational Inequality
  \begin{equation}
    \label{eq:cap5:15}
    \frac 12{\frac {\d}{\d t}\!\!}^+W_{2,\alpha;\Leb d}^2(\mu_t,\sigma)+\Psi_\alpha(\mu_t)\le
    \Psi_\alpha(\sigma)\quad \forall\, t\ge0,\ 
    \forall\, \sigma\in \mathcal Q.
  \end{equation}
\end{theorem}
\begin{proof}
  Since $\psi_\alpha''(u)=u^{-\alpha},$ a direct computation shows
  \begin{displaymath}
    %\frac \d{\d t}\Psi(\mu_t)=
    \frac \d{\d t}\int_\Rd \psi_\alpha(u_t)\,\d x=
    -\frac \d{\d t}\int_\Rd \nabla u_t\cdot \nabla \psi_\alpha'(u_t)\,\d x=
    \int_\Rd |\nabla u_t|^2 u_t^{-\alpha}\,\d x=
    \Phi_{2,\alpha}(u_t,\nabla u_t).
  \end{displaymath}
  Concerning \eqref{eq:cap5:15}, we use the technique introduced by
  \cite[\S\,2]{Daneri-Savare08}: we consider
  a geodesic $(\sigma_s,\nnu_s)_{s\in [0,1]}
  \in \cce01\Rd\sigma\mu$, which satisfies
   $\sigma_s(\Rd)=1$ by Theorem \ref{thm:moment_estimate2}. 
   We set
  \begin{displaymath}
    \sigma^\eps_{s,t}=u^\eps_{s,t}\Leb d:=\mathcal
    S_{\eps+st}[\sigma_s],\quad
    \tilde\nnu^\eps_{s,t}=\tilde\ww_{s,t}^\eps\Leb d:=\mathcal
    S_{\eps+st}[\nnu_s],\quad
    \ww^\eps_{s,t}:=\tilde\ww^\eps_{s,t}  -t\nabla u^\eps_{s,t}  .
  \end{displaymath}
  It is not difficult to check that
  \begin{equation}
    \label{eq:cap5:17}
    \partial_s u^\eps_{s,t}+\nabla\cdot \ww^\eps_{s,t}=0\quad
    \text{in }\Rd\times (0,1),
  \end{equation}
  so that
  \begin{displaymath}
    W_{2,\alpha;\Leb d}^2(\mu_{\eps+t},\sigma)\le
    \int_0^1 A^\eps_{s,t}\,\d s,\quad
    A^\eps_{s,t}:=\int_\Rd
    \big(u^\eps_{s,t}\big)^{-\alpha}|\ww^\eps_{s,t}|^2\,\d x
    =\Phi_{2,\alpha}(\sigma^\eps_{s,t},\nnu^\eps_{s,t}|\Leb d).
  \end{displaymath}
  We thus evaluate
  \begin{align}
    \notag
    A^\eps_{s,t}&=
    \int_\Rd \big(u^\eps_{s,t}\big)^{-\alpha}
    \Big(-2t\nabla u^\eps_{s,t}\cdot \ww^\eps_{s,t}+
    |\tilde\ww^\eps_{s,t}|^2-t^2 |\nabla u^\eps_{s,t}|^2\Big)\,\d x
    \\\notag&\le
    -2t \int_\Rd \big(u^\eps_{s,t}\big)^{-\alpha}
    \nabla u^\eps_{s,t}\cdot \ww^\eps_{s,t}\,\d x+
    \int_\Rd \big(u^\eps_{s,t}\big)^{-\alpha}
    |\tilde\ww^\eps_{s,t}|^2\,\d x
    \\
    \label{eq:last}
    &\le
    -2t \, \partial_s\, \Psi(\sigma^\eps_{s,t})+
    \Phi_{2,\alpha}(\sigma_s,\nnu_s|\Leb d),
  \end{align}
  where we used the facts
  \begin{displaymath}
    %\partial_s \Psi(\sigma^\eps_{s,t})=
    \partial_s\int_\Rd\psi_\alpha(u^\eps_{s,t})\,\d x
    \topref{eq:cap5:17}=
    \int_\Rd \nabla \psi'_\alpha(u^\eps_{s,t})\cdot \ww^\eps_{s,t}\,\d
    x\topref{eq:cap4:64}=\int_\Rd \big(u^\eps_{s,t}\big)^{-\alpha}
    \nabla u^\eps_{s,t}\cdot \ww^\eps_{s,t}\,\d x,
  \end{displaymath}
  \begin{displaymath}
     \int_\Rd \big(u^\eps_{s,t}\big)^{-\alpha}
     |\tilde\ww^\eps_{s,t}|^2\,\d x=
     \Phi_{2,\alpha}(\sigma_s\ast g_{\eps+st},\nnu_s\ast
     g_{\eps+st}|\Leb d)\le
     \Phi_{2,\alpha}(\sigma_s,\nnu_s|\Leb d)
  \end{displaymath}
  thanks to the convolution
  contraction property
  of Theorem \ref{thm:mono_conv}.
  Integrating \eqref{eq:last}
  with respect to $s$ from $0$ to $1$ and recalling that
  $(\sigma_s,\nnu_s)_{s\in [0,1]}$ is a minimal geodesic and that 
  $\sigma^\eps_{1,t}=\mu_{\eps+t}$ and $\sigma^\eps_{0,t}=\sigma$, we get
  \begin{equation}
    \label{eq:cap5:21}
    \int_0^1 A^\eps_{s,t}\,\d s+2t\, \Psi_\alpha(\mu_{\eps+t})
    \le
    %2t\, \Psi_\alpha(\sigma)+
    %\int_0^1 \Phi_{2,\alpha}(\sigma_s,\nnu_s|\Leb d)\,\d s=
    2t\,\Psi_\alpha(\sigma)+W_{2,\alpha;\Leb d}^2(\mu,\sigma).
  \end{equation}
  We deduce that
  \begin{equation}
    \label{eq:cap5:22}
    \tfrac12W^2_{2,\alpha}(\mu_{\eps+t},\sigma)+t\,\Psi(\mu_{\eps+t})\le
    t\,\Psi(\sigma)+\tfrac 12W_{2,\alpha}^2(\mu,\sigma).
  \end{equation}
  Passing to the limit as $\eps\down0$ and then as $t\down0$
  after dividing the inequality by $t$ we get
  \eqref{eq:cap5:15} at $t=0$. Recalling
  the semigroup property of the heat equation, we
  obtain \eqref{eq:cap5:15} for every time $t\ge0$.
\qed\end{proof}
\eqref{eq:cap5:15} is the metric formulation of 
the gradient flow of the (geodesically convex) functional
$\Psi_\alpha$ in the metric space $(\mathcal Q,W_{2,\alpha;\Leb d})$, see 
\cite[Chap.~4]{Ambrosio-Gigli-Savare05}.
Applying \cite[Theorem 3.2]{Daneri-Savare08} we eventually obtain:
\begin{corollary}[Geodesic convexity of $\Psi_\alpha$]
  \label{thm:entropy_convexity}
  Let $\alpha>1-2/d$, $\mu_i=\rho_i\Leb d\in \Probabilities\Rd$ with
  $W_{2,\alpha;\Leb d}(\mu_0,\mu_1)<+\infty$ and
  $   \int_\Rd \rho_i^{2-\alpha}\,\d x<+\infty,$
 %\end{equation}
  and let $\mu_t$, $t\in [0,1]$, be the minimal
  speed geodesic connecting $\mu_0$ to $\mu_1$ w.r.t.\
  $W_{2,\alpha;\Leb d}$.
  Then for every $t\in [0,1]$
  $\mu_t=\rho_t\Leb d\ll\Leb d$ 
  \begin{equation}
    \label{eq:cap4:63}
    \int_\Rd \rho_t^{2-\alpha}\,\d x\le (1-t)
    \int_\Rd \rho_0^{2-\alpha}\,\d x+
    t \int_\Rd \rho_1^{2-\alpha}\,\d x.
  \end{equation}
\end{corollary}

%%%%%%%%%%%%%%%%%%%%%%%%%%%%%%%%%%%%%%%%%%%%%%%%%%%%%%%%%%%%%%%%%%%%%%%%%%
%%%%%%%%%%%%%%%%%%%%%%%%%%%%%%%%%%%%%%%%%%%%%%%%%%%%%%%%%%%%%%%%%%%%%%%%%% 

\bibliographystyle{siam}
\bibliography{bibliografia}

\end{document}

%% file: Macro_times.tex
\newcommand{\nchi}{{\raise.3ex\hbox{$\chi$}}}

\newcommand{\R}{\mathbb{R}}

\newcommand{\N}{\mathbb{N}}

\newcommand{\M}{\mathbb{M}}

\newcommand{\cW}{{\ensuremath{\mathcal W}} }

\newcommand{\EE}{\mathscr{E}}

\newcommand{\nn}{{\vec n}}

\renewcommand{\tt}{{\vec t}}

\newcommand{\vv}{{\vec v}}
\newcommand{\ww}{{\vec w}}

\newcommand{\zz}{{\vec z}}

\newcommand{\tauV}{{\kern-3pt\tau}}

\newcommand{\yy}{{\vec y}}
\newcommand{\ttheta}{{\pmb \vartheta}}

\newcommand{\mmu}{{\pmb \mu}}
\newcommand{\nnu}{{\pmb \nu}}

\newcommand{\xxi}{{\pmb\xi}}
\newcommand{\zzeta}{{\pmb  \zeta}}
\newcommand{\eeta}{{\pmb \eta}}
\newcommand{\supp}{\mathop{\rm supp}\nolimits}   % supporto 
\newcommand{\Leb}[1]{{\mathscr L}^{#1}}      % Misura di Lebesgue
\newcommand{\la}{{\langle}}                  % brackets
\newcommand{\ra}{{\rangle}}
\newcommand{\eps}{\varepsilon}  
\newcommand{\Rd}{{\R^d}}

\newcommand{\Rh}{{\R^h}}
\newcommand{\Rm}{{\R^m}}

\newcommand{\Scalar}[2]{\left\langle #1,#2\right\rangle}

\newcommand{\restr}[1]{\lower3pt\hbox{$|_{#1}$}}
\newcommand{\BorelSets}[1]{\mathcal B(#1)}

\newcommand{\Probabilities}[1]{\mathcal P(#1)}          % misure di probabilita'
\newcommand{\FinalT}{T}
\newcommand{\weaksto}{{\rightharpoonup^*}}

\newcommand{\tdiv}{\nabla\kern-3pt_\gamma\cdot}

\newcommand{\ClosedDomain}{{\overline\Omega}}

\newcommand{\FPM}{\mathcal M^+}
\newcommand{\RPM}{\mathcal M^+_{\rm loc}}
\newcommand{\Mloc}{\mathcal M_{\rm loc}}
\newcommand{\Bloc}{\mathcal B_{\rm c}}

\newcommand{\ce}[3]{\mathcal{CE}(#1,#2)}
\newcommand{\CE}[5]{\mathcal{CE}_{#1,#2}(#3,#4)}
\newcommand{\cce}[5]{\mathcal {CE}(#1,#2;#4\to#5)}
\newcommand{\CCE}[7]{\mathcal{CE}_{#1,#2}(#3,#4;#6\to#7)}
\newcommand{\up}{\uparrow}
\newcommand{\down}{\downarrow}
\newcommand{\sfh}{{\sf h}}
\newcommand{\sfc}{{\sf c}}

\newcommand{\sfs}{{\sf s}}
\newcommand{\sft}{{\sf t}}
\newcommand{\sfm}{{\sf m}}
\newcommand{\sftm}{\tilde{\sf m}}

\newcommand{\sfC}{{\sf C}}
\newcommand{\sfH}{{\sf H}}
\newcommand{\sfG}{{\sf G}}

\renewcommand{\d}{\mathrm d}
\newcommand{\D}{\mathrm D}
\newcommand{\topref}[2]{\stackrel{\eqref{#1}}#2}
\newcommand{\kernel}{k}
\newcommand{\Ball}[1]{B_{#1}}
\newcommand{\Tan}[3]{\mathop{\rm Tan}\nolimits_{#1,#2}(#3)}
\newcommand{\Energy}{E}
\newcommand{\SZ}{Z}
\newcommand{\dphinorm}[1]{\|#1\|_*}
\newcommand{\phinorm}[1]{\|#1\|}